\documentclass[a4paper,12pt]{amsart}
\usepackage{amsmath,amssymb,amsfonts,amsthm}
\usepackage[abbrev,nobysame]{amsrefs}
\usepackage{mathtools,cases}
\usepackage{hyperref}
\usepackage[]{todonotes}
\usepackage{enumerate}   

\numberwithin{equation}{section}

\hypersetup{
 hidelinks,
 bookmarksopen=true
 }

\numberwithin{equation}{section}
\newcommand{\p}{\partial}
\newcommand{\vphi}{\varphi}
\newcommand{\om}{\omega}
\newcommand{\tri}{\triangle}
\newcommand{\eps}{\epsilon}
\newcommand{\thmref}[1]{Theorem~\ref{#1}}

\newcommand{\lemref}[1]{Lemma~\ref{#1}}
\newcommand{\corref}[1]{Corollary~\ref{#1}}

\newcommand{\Om}{\Omega}

\def\p{\partial}
\def\b{\beta}

\DeclareMathOperator{\tr}{Tr}

\newtheorem{theorem}{Theorem}[section]
\newtheorem{thm}[theorem]{Theorem}
\newtheorem{rem}[theorem]{Remark}
\newtheorem{defn}[theorem]{Definition}

\newtheorem{lem}[theorem]{Lemma}

\newtheorem{prop}[theorem]{Proposition}

\newtheorem{cor}[theorem]{Corollary}
\newtheorem{ques}[theorem]{Question}

\def\Re{\mathop{\rm Re}\nolimits}

\def\tr{\mathop{\rm tr}\nolimits}

\let\ul=\underline

\let\vphi=\varphi

\def\a{\alpha}

\title[Singular scalar curvature equations]{Singular scalar curvature equations}

\author{Kai Zheng}  
\address{University of Chinese Academy of Sciences, Beijing 100049, P.R. China; Tongji University, Shanghai 200092, P.R. China}
\email{KaiZheng@amss.ac.cn}

\begin{document}
\maketitle

\begin{abstract} 
We develop estimates for the equation of scalar curvature of singular metrics with cone angle $\beta>1$, in a big and semi-positive cohomology class on a K\"ahler manifold.
We further derive the Laplacian estimate for the scalar curvature equation of degenerate K\"ahler metrics. We then have several applications of these estimates on the singular constant scalar curvature K\"ahler metrics, which also include the singular K\"ahler-Einstein metrics.
\end{abstract}

\section{Introduction}

Let $(X,\om)$ be a closed K\"ahler manifold. Calabi initialed the study of K\"ahler metrics with constant scalar curvature (cscK) in a given K\"ahler class $[\om]$ on $X$. Recently, Chen-Cheng \cite{MR4301557} solved the geodesic stability conjecture and the properness conjecture on existence of the cscK metrics, via establishing new a priori estimates for the 4th order cscK equation \cite{MR4301558}. However, the Yau-Tian-Donaldson conjecture expects the equivalence between the existence of cscK metrics and algebraic stabilities, which requires further understood of the possible singularity/degeneration developing from the cscK problem. In this article, we work on a priori estimates for the scalar curvature equation of both the singular and degenerate K\"ahler metrics.

Before we move further, we fix some notations. We are given an ample divisor $D$ in $X$, together with its associated line bundle $L_D$. We let $s$ be a holomorphic section of $L_D$ and $h$ be a Hermitian metric on $L_D$. 
We set $\Om$ to be a big and semi-positive cohomology class.

The \textbf{singular scalar curvature equation} in $\Om$ was introduced in \cite{arXiv:1803.09506} c.f. Definition \ref{Singular cscK eps defn}. Within the same article, we derive estimation of its solution when $0<\beta\leq1,$ which results in existence of cscK metrics on some normal varieties. If we further let $\Om$ to be K\"ahler, the singularities is called cone singularities and the solution of the singular scalar curvature equation is called \textbf{cscK cone metric}. The uniqueness, existence  and the necessary direction of the corresponding log version of the Yau-Tian-Donaldson conjecture were proven in \cite{MR3968885,MR4020314,arXiv:1803.09506,arXiv:2110.02518,MR4088684}. In this article, we continue studying the singular scalar curvature equation, when $$\beta>1.$$

The singular scalar curvature equation has the expression
\begin{equation}\label{Singular cscK}
\om_\vphi^n=(\om_{sr}+i\p\bar\p\vphi)^n=e^F \om_{\theta}^n,\quad
\tri_{\vphi} F=\tr_{\vphi}\theta-R,
\end{equation}
where, $\om_{sr}$ is a smooth representative in the class $\Om$ and $R$ is a real-valued function.

We introduce two parameters to approximate the singular equation \eqref{Singular cscK}, $t$ for the lose of positivity and $\eps$ for the degeneration, see Definition \ref{Singular cscK t eps defn}.
We say a solution $(\vphi_{t,\eps},F_{t,\eps})$ to the approximate equation \eqref{Singular cscK t eps} is \textbf{almost admissible}, if it has uniform weighted estimates independent of $t,\eps$, including the $L^\infty$-estimates of both $\vphi_{t,\eps}$ and $F_{t,\eps}$, the gradient estimate of $\vphi_{t,\eps}$ and the $W^{2,p}$-estimate. The precise definition is given in Definition \ref{a priori estimates approximation singular}.

\begin{thm}\label{intro almost admissible}
The solution to the approximate singular scalar curvature equation \eqref{Singular cscK t eps} with bounded entropy is almost admissible.
\end{thm}

The detailed statement of these estimates will be given in the $L^\infty$-estimates (\thmref{L infty estimates Singular equation}), the gradient estimate of $\vphi$ (\thmref{gradient estimate}), the $W^{2,p}$-estimate (\thmref{w2pestimates degenerate Singular equation}). 

The singular scalar curvature equation extends the singular Monge-Amp\`ere equation. Actually, the reference metric $\om_\theta$ in the singular scalar curvature equation \eqref{Singular cscK}, is defined to be a solution to the singular Monge-Amp\`ere equation
\begin{align}\label{Rictheta intro}
\om_\theta^n=(\om_{sr}+i\p\bar\p\vphi_{\theta})^n=|s|_h^{2\b-2}e^{h_{\theta}}\om^n.
\end{align} The precise definition is given in Definition \ref{Reference metric singular} in Section \ref{Reference metrics}. 
\begin{rem}
There is a large literature on the singular Monge-Amp\`ere equations \cite{MR2746347,MR944606,MR2647006,MR2505296,MR2233716,MR4319060} and the corresponding K\"ahler-Ricci flow \cite{MR3956691,MR2869020,MR3595934,MR4157554}, see also the references therein. However, more effort is required to tackle the new difficulties arising from deriving estimates for the scalar curvature equation \eqref{Singular cscK}, because of its 4th order nature. The tool we apply here is the integration method, instead of the maximum principle.
\end{rem}
\begin{rem}
We extend Chen-Cheng's estimates \cite{MR4301557} to the singular scalar curvature equation \eqref{Singular cscK}.
\end{rem}

\bigskip

If we focus on the degeneration, setting $\Om$ to be a K\"ahler class, then the singular equation \eqref{Singular cscK} is named the \textbf{degenerate scalar curvature equation}. The accurate definition is given in Definition \ref{Degenerate cscK defn} and the corresponding approximation is stated in Definition \ref{Degenerate cscK 1 approximation}.
The almost admissible estimates (Definition \ref{a priori estimates approximation}) for the approximate solution $\vphi_\eps$ are obtained from \thmref{intro almost admissible}, immediately.

In Section \ref{A priori estimates for approximate degenerate cscK equations}, we further show metric equivalence from the volume ratio bound, i.e. to prove the Laplacian estimate for the degenerate scalar curvature equation, \thmref{cscK Laplacian estimate}.

An almost admissible function is called \textbf{$\gamma$-admissible}, if it admits a weighted Laplacian estimate, which is defined in Definition \ref{admissible}.

\begin{thm}\label{admissible estimate}
The almost admissible solution to the approximate degenerate scalar curvature equation \eqref{Degenerate cscK approximation} with bounded $\|\p F_\eps\|_{\vphi_\eps}$ is
\begin{itemize}
\item
admissible, i.e. $\tr_\om\om_{\vphi_\eps}\leq C$, when $\beta> \frac{n+1}{2}$.
\item
$\gamma$-admissible for any $\gamma>0$, when $1<\beta\leq\frac{n+1}{2}$.
\end{itemize}
 \end{thm}
\begin{rem}
The Laplacian estimate holds for the smooth cscK metric $\beta=1$ in \cite[Theorem 1.7]{MR4301557}, the cscK cone metric when $0<\beta<1$ in \cite[Theorem 5.25]{arXiv:1803.09506} and see also \cite[Question 4.42]{arXiv:1803.09506} for further development. However, when $\beta>1$, the Laplacian estimate is quite different from the one when the cone angle is not larger than $1$. 
\end{rem}

For the degenerate scalar curvature equation, the Laplacian estimates is more involved. Our new approach is \textit{an integration method with weights}. The proof is divided into seven main steps. 

The first step is to transform the approximate degenerate scalar curvature equation \eqref{Degenerate cscK approximation} into an integral inequality with weights. The second step is to deal with the trouble term containing $\tri F$ by integration by parts. Unfortunately, this causes new difficulties, because of the loss of the weights, see Remark \ref{4th term}. 

The third step is to apply the Sobolev inequality to the integral inequality to conclude a rough iteration inequality. However, weights loss in this step again, see Remark \ref{Laplacian estimate: lose weight}. Both the fourth and fifth step are designed to overcome these difficulties. The fourth step aims to construct several inverse weighted inequalities, with the help of introducing two useful parameters $\sigma$ and $\gamma$. While, the fifth step is set for weighted inequalities, by carefully calculating the sub-critical exponent for weights. 

As last, with all these efforts, we arrive at an iteration inequality and proceed the iteration argument to conclude the Laplacian estimates.

\bigskip

After we establish the estimation, in Section \ref{Convergence and regularity}, we have two quick applications on the singular/degenerate cscK metrics, which were introduced in \cite{arXiv:1803.09506} as the $L^1$-limit of the approximate solutions, Definition \ref{Singular metric with prescribed scalar curvature}. 

One is to show regularities of the singular/degenerate cscK metrics. Precision statements are given in \thmref{existence singular} and \thmref{existence degenerate}. 
We see that that the potential function of the degenerate cscK metric is almost admissible and the convergence is smooth outside the divisor $D$. We also find that the volume form $\om^n_\vphi$ is shown to be prescribed with the given degenerate type along $D$ as
$
|s|_h^{2\b-2}\om^n.
$
The volume ratio bound could be further improved to be a global metric bound, under the assumption that the volume ratio of the approximate sequence has bounded gradient. Consequently, the degenerate cscK metric has weighted complex Hessian on the whole manifold $X$. In particular, when the cone angle is larger than half of $n+1$, the complex Hessian is globally bounded.


The other application is to revisit Yau's result on degenerate K\"ahler-Einstein (KE) metrics in the seminal work \cite{MR480350}. 
A special situation of the degenerate scalar curvature equation is the degenerate K\"ahler-Einstein equation
\begin{align}\label{critical pt Ding intro}
\om^n_\vphi=|s|_h^{2\b-2} e^{h_\om-\lambda\vphi +c}  \om^n.
\end{align}
In Section \ref{Reference metric and prescribed Ricci curvature problem}, the variational structure of the degenerate K\"ahler-Einstein metrics will be presented. We will show that they are all critical solutions of a logarithm variant of Ding's functional. Ding functional was introduced in \cite{MR967024}.  

A direct corollary from \thmref{intro almost admissible} and \thmref{admissible estimate} is
\begin{cor}[Yau \cite{MR480350}]\label{admissible estimate KE}
For any $\beta\geq1$, the solution $\vphi_\eps$ to the degenerate K\"ahler-Einstein equation \eqref{critical pt Ding intro} has $L^\infty$-estimate and bounded Hessian.
 \end{cor}
 \begin{rem}
The accurate statement is given in \thmref{KE Laplacian}.
\end{rem}
For the degenerate Monge-Amp\`ere problem, the difficult fourth order nonlinear term $\tri F$ in the scalar curvature problem becomes the second order term $\tri \vphi$, which could be controlled directly, comparing Proposition \ref{Laplacian estimate: integration inequality pro} in Section \ref{Step 2} and \eqref{KE tri F} in Section \ref{Log Kahler Einstein metric}. 

\begin{rem}
As a result, Yau's theorem for degenerate K\"ahler-Einstein metrics is recovered, that is when $\lambda\leq 0$, the degenerate Monge-Amp\`ere equation \eqref{critical pt Ding intro} admits a solution with bounded complex Hessian and being smooth outside the divisor $D$.
But, the method we apply to derive the Laplacian estimates is the integration method, which is different from Yau's maximum principle. 
\end{rem}

As a further matter, the integration method with weights is applied to derive the gradient estimates in Section \ref{Gradient estimate of vphi} as well.

\begin{rem}Our results also extend the gradient and the Laplacian estimate for the non-degenerate Monge-Amp\`ere equation, which were obtained by the integral method \cite{MR2993005}.
\end{rem}


At last, as a continuation of Question 1.14 in \cite{MR4020314} and Question 1.9 in \cite{arXiv:1803.09506}, we could propose the following uniqueness question.
\begin{ques}
whether the $\gamma$-admissible degenerate cscK metric constructed above is unique (up to automorphisms)?
\end{ques}
It is worth to compare this question to its counterpart for the cone metrics \cite{MR3496771,MR3405866,MR4020314}.

\begin{rem}
On Riemannian surfaces, there are intensive study on constant curvature metrics with large cone angles, see \cite{MR3990195,MR3340376} and references therein.
\end{rem}

\section{Degenerate scalar curvature problems}\label{cscK}
We denote the curvature form of $h$ to be
$
\Theta_D=-i\p\bar\p\log h,
$
which represents the class $C_1(L_D)$.
The Poincar\'e-Lelong equation gives us that
\begin{align}\label{PL}
2\pi[D]=i\p\bar\p\log|s|_h^2+\Theta_D=i\p\bar\p\log|s|^2.
\end{align}

In this section, we assume that the given K\"ahler class $\Om$ is proportional to $$C_1(X,D):=C_1(X)-(1-\beta)C_1(L_D).$$ W let $\lambda$ be a constant such that
\begin{align}\label{cohomology condition}
C_1(X,D)=\lambda \Om .
\end{align}
This cohomology condition \eqref{cohomology condition} implies the existence of a smooth function $h_\om$ such that
\begin{align}\label{h0}
Ric(\om)=\lambda\om+(1-\beta)\Theta_D+i\p\bar\p h_\om.
\end{align}



\subsection{Critical points of the log Ding functional}\label{Critical points of the log Ding functional}
The log Ding functional is a modification of the Ding functional \cite{MR967024} by adding the weight $|s|_h^{2\b-2}$. 
\begin{defn}For all $\vphi$ in $\mathcal H_\Om$, we define the \textbf{log Ding functional} to be
\begin{align*}
F_\beta(\vphi)=-D_{\om}(\vphi)-\frac{1}{\lambda}\log(\frac{1}{V}\int_X e^{h_\om-\lambda\vphi}|s|_h^{2\b-2}\om^n).
\end{align*}
Here
$
D_{\om}(\vphi)
:=\frac{1}{V}\frac{1}{n+1}\sum_{j=0}^{n}\int_{X}\vphi\om^{j}\wedge\om_{\varphi}^{n-j} .
$
\end{defn}

We compute the critical points of the log Ding functional.
\begin{prop}
The 1st variation of the log Ding functional at $\vphi$ is
\begin{align*}
\delta F_\beta(u)&=-\frac{1}{V}\int_X u \om_\vphi^n+\left(\int_Xe^{h_\om-\lambda\vphi}|s|_h^{2\b-2}\om^n\right)^{-1}  \int_X u e^{h_\om-\lambda\vphi}|s|_h^{2\b-2}\om^n.
\end{align*}
The critical point of the log Ding functional satisfies the following equation
\begin{align}\label{critical pt Ding}
\om^n_\vphi
&=|s|_h^{2\b-2} \frac{V\cdot e^{h_\om-\lambda\vphi }  \om^n}{\int_Xe^{h_\om-\lambda\vphi} |s|_h^{2\b-2} \om^n}.
\end{align}
\end{prop}

\begin{defn}[Log K\"ahler-Einstein metric]
We call the solution of \eqref{critical pt Ding}, \textbf{log K\"ahler-Einstein metric}. 
When $0<\beta<1$, it is called K\"ahler-Einstein cone metric, c.f. \cite{MR3761174,MR3911741}.
\end{defn}
Further computation shows that the log K\"ahler-Einstein metric satisfies the following identity
\begin{align}\label{Rictheta}
Ric(\om_\vphi)=\lambda\om_\vphi+2\pi (1-\b)[D].
\end{align}

\begin{prop}
The 2nd variation of the log Ding functional at $\vphi$ is
\begin{align*}
&\delta^2 F_\beta(u,v)=\frac{1}{V}\int_X (\p u,\p v)_\vphi\om_\vphi^n\\
&+\lambda\left\{\frac{\int_X u e^{h_\om-\lambda\vphi}|s|_h^{2\b-2}\om^n \int_Xv e^{h_\om-\lambda\vphi}|s|_h^{2\b-2}\om^n}{(\int_Xe^{h_\om-\lambda\vphi}|s|_h^{2\b-2}\om^n)^{2}}
-\frac{\int_X uv e^{h_\om-\lambda\vphi}|s|_h^{2\b-2}\om^n}{\int_Xe^{h_\om-\lambda\vphi}|s|_h^{2\b-2}\om^n}\right\}.
\end{align*}
\end{prop}
\begin{proof}
We continue the proceeding computation to see that
\begin{align*}
&\delta^2 F_\beta(u,v)=-\frac{1}{V}\int_X u \tri_\vphi v \om_\vphi^n\\
&-\left(\int_Xe^{h_\om-\lambda\vphi}|s|_h^{2\b-2}\om^n\right)^{-2} \int_X (-\lambda v) e^{h_\om-\lambda\vphi}|s|_h^{2\b-2}\om^n \int_X u e^{h_\om-\lambda\vphi}|s|_h^{2\b-2}\om^n\\
&+\left(\int_Xe^{h_\om-\lambda\vphi}|s|_h^{2\b-2}\om^n\right)^{-1}  \int_X u  (-\lambda v)  e^{h_\om-\lambda\vphi}|s|_h^{2\b-2}\om^n.
\end{align*}
So the conclusion follows directly.
\end{proof}

Then inserting the critical point equation \eqref{critical pt Ding} into the 2nd variation of the log Ding functional, we have the following corollaries.
\begin{cor}\label{the 2nd variation of the log Ding}
At a log K\"ahler-Einstein metric, the 2nd variation of the log Ding functional becomes
\begin{align*}
\delta^2 F_\beta(u,v)
&=\frac{1}{V}\int_X (\p u,\p v)_\vphi\om_\vphi^n+\lambda\left\{\frac{1}{V^2}\int_X u\om_\vphi^n\int_X v\om_\vphi^n-\frac{1}{V}\int_X  u v\om_\vphi^n\right\}.
\end{align*}
\end{cor}

\begin{cor}\label{log Ding convex}The log Ding functional is convex at its critical points, if one of the following condition holds
\begin{itemize}
\item $C_1(X,D)\leq 0$,\label{convex negative}
\item $X$ is a projective Fano manifold, $D$ stays in the linear system of $|K_X^{-1}|$ and $\Om=C_1(X)$, $0< \beta\leq 1$. \label{convex positive}
\end{itemize}
\end{cor}

\begin{proof}
Since $\lambda\Om=C_1(X,D)$ by \eqref{cohomology condition}, we have $\lambda\leq 0$. The lemma follows from applying the Cauchy-Schwarz inequality to \lemref{the 2nd variation of the log Ding}.

From the assumptions of the second statement, we have $C_1(X,D)=\beta C_1(X)=\beta\Om$ and then $\lambda=\beta$ by \eqref{cohomology condition}. Thus convexity follows from the lower bound of the first eigenvalue of the Laplacian for a K\"ahler-Einstein cone metric \cite{MR3761174}. 
\end{proof}


\subsection{Reference metric and prescribed Ricci curvature problem}\label{Reference metric and prescribed Ricci curvature problem}

We choose a smooth $(1,1)$-form $\theta\in C_1(X,D).$
Then there exists a smooth function $h_\theta$ such that
\begin{align}\label{h0}
Ric(\om)=\theta+(1-\beta)\Theta_D+i\p\bar\p h_\theta.
\end{align} 
\begin{defn}[Reference metric]\label{Reference metric}
A reference metric $\om_\theta$ is defined to be the solution of the degenerate Monge-Amp\`ere equation
\begin{align}\label{Rictheta}
\om^n_\theta=|s|_h^{2\b-2}e^{h_\theta} \om^n
\end{align}
 under the normalisation condition 
 \begin{align}\label{Rictheta normalisation}
\int_X |s|_h^{2\b-2}e^{h_\theta} \om^n=V.
\end{align}
\end{defn}
The reference metric $\om_\theta=\om+i\p\bar\p\vphi_\theta$ satisfies the equation for prescribing the Ricci curvature problem
\begin{align}\label{Rictheta ricci}
Ric(\om_\theta)=\theta+2\pi (1-\b)[D].
\end{align}
The equation \eqref{Rictheta} is a special case of the degenerate complex Monge-Amp\`ere equation \eqref{critical pt Ding}.

\begin{lem}With the reference metric $\om_\theta$ in \eqref{Rictheta}, the log Ding functional is rewritten as
\begin{align*}
F_\beta(\vphi)=-D_{\om}(\vphi)-\frac{1}{\lambda}\log(\frac{1}{V}\int_X e^{h_\om-h_\theta-\lambda\vphi}\om_\theta^n).
\end{align*}
The critical point satisfies the following equation
\begin{align}\label{critical pt Ding theta}
\om^n_\vphi
=\frac{V\cdot e^{h_\om-h_\theta-\lambda\vphi} \om^n_\theta}{\int_Xe^{h_\om-h_\theta-\lambda\vphi}\om^n_\theta}.
\end{align}
\end{lem}
From the log K\"ahler-Einstein equation \eqref{critical pt Ding} and the reference metric equation \eqref{Rictheta}, we have
\begin{align*}
 e^{h_\om-\lambda\vphi}\om^n_\theta= e^{h_\om-\lambda\vphi}|s|_h^{2\b-2}e^{h_\theta} \om^n=e^{h_\theta}\om^n_\vphi.
\end{align*}
The 2nd variation of the log Ding functional at $\vphi$ is
\begin{align*}
\delta^2 F_\beta(u,v)&=\frac{1}{V}\int_X (\p u,\p v)_\vphi\om_\vphi^n\\
&+\lambda\left\{\frac{\int_X u e^{h_\om-h_\theta-\lambda\vphi}\om^n_\theta\int_Xv e^{h_\om-h_\theta-\lambda\vphi}\om^n_\theta}{(\int_Xe^{h_\om-h_\theta-\lambda\vphi}\om^n_\theta)^{2}}
-\frac{\int_X uv e^{h_\om-h_\theta-\lambda\vphi}\om^n_\theta}{\int_Xe^{h_\om-h_\theta-\lambda\vphi}\om^n_\theta}\right\}.
\end{align*}

\subsection{Degenerate Monge-Amp\`ere equations}
The degenerate Monge-Amp\`ere equations \eqref{critical pt Ding} and \eqref{Rictheta} could be summarised as
\begin{align}\label{general}
\om^n_\vphi=|s|_h^{2\b-2} e^{F(x,\vphi)}\om^n.
\end{align}
When the cone angle $\beta>1$,  we define a solution of \eqref{general} as following.
\begin{defn}\label{admissible}
A function $\vphi$ is said to be $\gamma$-\textbf{admissible} for some $\gamma\geq 0$, if
\begin{itemize}
\item $\tr_\om\om_{\vphi}\leq C  |s|_h^{-2\gamma}$, on $X$,
\item $\vphi$ is smooth on $M$.
\end{itemize}
Moreover, we say a function $\vphi$ is a $\gamma$-admissible solution to the equation \eqref{general}, if it is a $\gamma$-admissible function and satisfies \eqref{general} on $M$.
When $\gamma=0$, the function $\vphi$ is said to be \textbf{admissible}.
\end{defn}

Yau initiated the study of these degenerate Monge-Amp\`ere equations in his seminal article \cite{MR480350} on Calabi conjecture.
For the K\"ahler-Einstein cone metrics with cone angle $0<\beta\leq1$, see \cite{MR3761174} and references therein.


\subsubsection{Approximation of the degenerate Monge-Amp\`ere equations}\label{Approximation}

In this section, we discuss some properties of the approximation of the reference metric.
We let
$$
\mathfrak h_\eps:=(\beta-1)\log S_\eps, \quad S_\eps:=|s|^2_h+\eps.
$$
We define $\om_{\theta_\eps}$ to be the approximation of the reference metric \eqref{Rictheta},
\begin{align}\label{Rictheta approximation}
\om^n_{\theta_\eps}=e^{h_\theta+\mathfrak h_\eps+c_\eps} \om^n
\text{ with }
\int_X e^{h_\theta+\mathfrak h_\eps+c_\eps} \om^n=V.
\end{align}
By \eqref{h0}, its Ricci curvature has the formula.
\begin{lem}
\begin{align}\label{Ricomthetaeps}
Ric(\om_{\theta_\eps})&=Ric(\om)-i\p\bar\p h_\theta-i\p\bar\p\mathfrak h_\eps\\
&=\theta+(1-\beta)\Theta_D+(1-\beta)i\p\bar\p\log S_\eps.\nonumber
\end{align}
\end{lem}
\begin{lem}\label{h eps} 
There exists a nonnegative constant $C$ such that
\begin{align}\label{h eps equation}
 CS_\eps^{-1}\om\geq  i\p\bar\p \log S_\eps\geq -\frac{|s|^2_h}{S_\eps}\Theta_D\geq -C\om.
\end{align}
\end{lem}
\begin{proof}
It follows from the calculation 
\begin{align*}
 i\p\bar\p \log S_\eps&=\frac{i\p\bar\p |s|^2_h}{S_\eps}-\frac{i\p|s|^2_h \bar\p|s|^2_h }{S_\eps^2}.
\end{align*}
The upper bound is direct to see, by using $|\p S_\eps|_{\om}\leq C S_\eps^{\frac{1}{2}}$, that
\begin{align*}
 i\p\bar\p \log S_\eps\leq CS_\eps^{-1}.
\end{align*}

While, making use of the identity
\begin{align*}
|s|^4_h i\p\bar\p \log|s|^2_h&=|s|^2_hi\p\bar\p |s|^2_h-i\p|s|^2_h \bar\p|s|^2_h,
\end{align*}
we have 
\begin{align*}
 i\p\bar\p \log S_\eps&=\frac{i\p\bar\p |s|^2_h}{S_\eps}
+\frac{|s|^4_h i\p\bar\p \log|s|^2_h - |s|^2_2i\p\bar\p |s|^2_h }{S_\eps^2}\\
 &=\frac{\eps i\p\bar\p |s|^2_h}{S_\eps^2}
+\frac{|s|^4_h i\p\bar\p \log|s|^2_h  }{S_\eps^2}.
\end{align*}
Replaced by the identity above again, it is further reduced to
\begin{align*}
&=\frac{\eps |s|^2_h i\p\bar\p \log|s|^2_h}{S_\eps^2}
+\frac{\eps i\p|s|^2_h \bar\p|s|^2_h}{|s|^2_hS_\eps^2}
+\frac{|s|^4_h i\p\bar\p \log|s|^2_h  }{S_\eps^2}\\
&=\frac{|s|^2_h i\p\bar\p \log|s|^2_h  }{S_\eps}+\frac{\eps i\p|s|^2_h \bar\p|s|^2_h}{|s|^2_hS_\eps^2}
 \geq -\frac{|s|^2_h}{S_\eps}\Theta_D.
\end{align*}
\end{proof}

\subsubsection{Approximation of the log KE metrics}
Similarly, the approximation for the log KE metric is defined to be
\begin{align}\label{critical pt Ding approximation}
\om^n_{\vphi_\eps}
=  e^{h_\om+\mathfrak h_\eps-\lambda\vphi_\eps +c_\eps}  \om^n,\quad \int_Xe^{h_\om+\mathfrak h_\eps-\lambda\vphi_\eps+c_\eps} \om^n=V.
\end{align}
They are critical points of the approximation of the log Ding functional
\begin{align*}
F^\eps_\beta(\vphi)=-D_{\om}(\vphi)-\frac{1}{\lambda}\log[\frac{1}{V}\int_X e^{h_\om-\lambda\vphi}S_\eps^{\b-1}\om^n].
\end{align*}

When $\beta>1$, we say an admissible solution to \eqref{critical pt Ding} is a degenerate KE metric. If $0<\beta< 1$, the H\"older space $C^{2,\a,\beta}$ was introduced in \cite{MR2975584}. A KE cone metric is a $C^{2,\a,\beta}$ solution to \eqref{critical pt Ding} and smooth on $M$.

Corresponding to Corollary \ref{log Ding convex}, we have
\begin{prop}\label{smooth approximation KE}
The following statements of existence of log KE metric and its smooth approximation hold.
\begin{itemize}
\item
When $\lambda>0$ and $0<\beta\leq 1$, we have $F^\eps_\beta \geq F_\beta$. Furthermore, if $F_\beta$ is proper, then there exists a KE cone metric of cone angle $\beta$, which has smooth approximation \eqref{critical pt Ding approximation}.
\item
When $\lambda<0$ and $\beta> 1$, we have $F^\eps_\beta \geq F_\beta$. Furthermore,  there exists a degenerate KE metric, which has smooth approximation \eqref{critical pt Ding approximation}.
\end{itemize}
\end{prop}
\subsection{Degenerate scalar curvature equation}\label{cscK pde}

In this section, we assume $n\geq 2$ and consider the degenerate case when $\beta>1$.
Recall that $\theta$ is a smooth $(1,1)$-form in $C_1(X,D)$ and $\om_\theta$ is the reference metric defined in \eqref{Rictheta}. We also set $\mathfrak h:=-(1-\b)\log |s|_h^2$ and define 
\begin{align}\label{S const}
\underline S_\b=\frac{C_1(X,D)[\om]^{n-1}}{[\om]^n}.
\end{align}
\begin{defn}\label{Degenerate cscK defn} 
The degenerate scalar curvature equation is defined as
\begin{equation}\label{Degenerate cscK}
\om_{\vphi}^n=e^F \om_{\theta}^n,\quad
\tri_{\vphi} F=\tr_{\vphi}\theta-R.
\end{equation}
Here, the reference metric $\om_\theta$ is introduced in Definition \eqref{Reference metric}.
When $R=\underline S_\b$, a solution to the degenerate scalar curvature equation is called a degenerate cscK metric.
\end{defn}
Direct computation shows that
\begin{lem}The scalar curvature of the degenerate scalar curvature equation satisfies that
\begin{align}
S(\om_\vphi)=R \text{ on }M.
\end{align}
\end{lem}

At last, we close this section by making use of the reference metric \eqref{Rictheta} then rewriting \eqref{Degenerate cscK} as below with respect to the smooth K\"ahler metric $\om$. We let 
\begin{align*}
 f=-h_\theta-\mathfrak h, \quad \tilde F=F-f.
\end{align*}

\begin{lem}The degenerate scalar curvature equation satisfies the following equations
\begin{equation}\label{Degenerate cscK 1}
\om_{\vphi}^n=e^{\tilde F} \om^n,\quad
\tri_{\vphi} \tilde F=\tr_{\vphi}(\theta-i\p\bar\p f)-R.
\end{equation} 
\end{lem}
\subsection{Log $K$-energy}\label{log K energy section}
Motivated from the log $K$-energy for the cone angle $\beta$ in $(0,1]$ in \cite{MR4020314}, we define the log $K$-energy for large cone angle $\beta>1$ to be
\begin{defn}\label{logKenergy}
The log $K$-energy is defined as
\begin{align}\label{log K energy}
\nu_\beta(\vphi)
&:=E_\beta(\vphi)
+J_{-\theta}(\vphi)+\frac{1}{V}\int_M (\mathfrak h+h_\theta)\om^n, \quad \forall \vphi\in \mathcal H_\Om.
\end{align}
Here $\om_\theta$ is the reference metric, which is the admissible solution to the degenerate complex Monge-Amp\`ere equation \eqref{Rictheta}. Also, the log entropy is defined in terms of the reference metric $\om_\theta$ as
\begin{align*}
E_\beta(\vphi)=\frac{1}{V}\int_M\log\frac{\om^n_\vphi}{\om_\theta^n}\om_\vphi^{n}
\end{align*}
and the $j_{\chi}$- and $J_{\chi}$-functionals are defined to be
\begin{align*}
j_{\chi}(\vphi):=\frac{1}{V}\int_{X}\vphi
\sum_{j=0}^{n-1}\om^{j}\wedge
\om_\vphi^{n-1-j}\wedge \chi,\quad
J_{\chi}(\vphi):=j_{\chi}(\vphi)-\underline{\chi} D_{\om}(\vphi).
\end{align*}
\end{defn}

\begin{cor}For all $\vphi$ in $\mathcal H_\Om$, we have
\begin{align*}
\nu_\beta(\vphi)&=\nu_1(\vphi)+(1-\beta)\frac{1}{V}\int_X\log|s|^2_h(\omega^n_\vphi-\om^n)
+(1-\b)J_{\Theta_D}(\vphi)\\
&=\nu_1(\vphi)+(1-\beta)\cdot[D_{\om,D}(\vphi)
-\frac{\mathrm{Vol}(D)}{V}\cdot D_{\om}(\vphi)].
\end{align*}
Here $\nu_1(\vphi)=E_1(\vphi)
+J_{-Ric(\om)}(\vphi)$ is the Mabuchi $K$ energy,  the entropy of $\om^n_\vphi$ is
	\begin{align*}
	&E_1(\vphi)=\frac{1}{V}\int_X \log\frac{\om^n_{\vphi}}{\om^n}\om^n_{\vphi}.
	\end{align*} and,  
the corresponding volume and the normalisation functional on the divisor $D$ are defined to be
\begin{align*}
\mathrm{Vol}(D)=\int_D \Omega^{n-1},\quad D_{\om,D}(\vphi)=\frac{n}{V}\int_0^1\int_D \p_t\vphi \om_\vphi^{n-1}dt.
\end{align*}
\end{cor}

\begin{cor}\label{K and log K}
Writing in terms of the smooth background metric $\om$, we have
\begin{align}\label{log K energy equation}
\nu_\beta(\vphi)=E_1(\vphi)
+J_{-\theta}(\vphi)+\frac{1}{V}\int_M (\mathfrak h+h_\theta)(\om^n-\om_\vphi^n),
\end{align}for all $\vphi\in \mathcal H_\Om$.
\end{cor}
We see that the last term has $\beta$ involved,
\begin{align}
\frac{-(1-\b)}{V}\int_M \log |s|_h^2 (\om^n-\om_\vphi^n).
\end{align}



\subsection{Approximate degenerate scalar curvature equation}
We first define the approximate degenerate cscK equation.
\begin{defn}
We say $\om_{\vphi_\eps}$ is an approximation of the degenerate scalar curvature equation, if it satisfies the following PDEs
\begin{equation}\label{Degenerate cscK approximation}
\om_{\vphi_\eps}^n=e^{F_\eps} \om_{\theta_\eps}^n,\quad
\tri_{\vphi_\eps} F_\eps=\tr_{\vphi_\eps}{\theta}-R.
\end{equation}
Particularly, when $R$ is a constant, we call it the approximate degenerate cscK equation.
\end{defn}

Then we write the approximate equation \eqref{Degenerate cscK approximation} in terms of the smooth background metric $\om$.
\begin{lem}The approximate degenerate scalar curvature equation satisfies the equations
\begin{equation}\label{Degenerate cscK 1 approximation}
\om_{\vphi_\eps}^n=e^{\tilde F_\eps} \om^n,\quad
\tri_{\vphi_\eps} \tilde F_\eps=\tr_{\vphi_\eps}(\theta-i\p\bar\p \tilde f_\eps)-R.
\end{equation} 
Here 
$
\tilde F_\eps=F_\eps-\tilde f_\eps,\quad 
\tilde f_\eps=-h_\theta-\mathfrak h_\eps-c_\eps.
$
\end{lem}

This section is devoted to prove the Laplacian estimates of the smooth solution $(\vphi_\eps,F_\eps)$ for the approximate degenerate scalar curvature equations \eqref{Degenerate cscK approximation}  or \eqref{Degenerate cscK 1 approximation}.

\subsection{Almost admissible solutions}
In order to clarify the idea of the Laplacian estimates. We introduce the following definition. 

\begin{defn}[Almost admissible solution for degenerate equations]\label{a priori estimates approximation}
We say $\vphi_\eps$ is an \text{almost admissible solution} to the approximate degenerate scalar curvature equation \eqref{Degenerate cscK 1 approximation}, if the following uniform estimates independent of $\eps$ hold
\begin{itemize}
\item $L^\infty$-estimates in \thmref{L infty estimates degenerate equation}: 
\begin{align}\label{almost admissible C0}
\|\vphi_\eps\|_{\infty},\quad \|F_\eps\|_{\infty}\leq C;
\end{align}
\item gradient estimate of $\vphi$ in \thmref{gradient estimate}: 
\begin{align}\label{almost admissible C1 sigmaD}
S_\eps^{\sigma^1_D}\|\p\vphi\|_{L^\infty(\om)}^2\leq C, \quad1>\sigma^1_D>\max\{1-\frac{2\beta}{n+2},0\};
\end{align}

\item $W^{2,p}$-estimate in \thmref{w2pestimates degenerate equation}: for any $p\geq 1$, 
\begin{align}\label{almost admissible w2p sigmaD}
\int_X   (\tr_{\om}\om_{\vphi_\eps})^{p}   S_\eps^{\sigma^2_D}\om^n\leq C(p),\quad\sigma^2_D:=(\beta-1)\frac{n-2}{n-1+p^{-1}}.
\end{align}
\end{itemize}

\end{defn}
\begin{rem}
These three estimates will be obtained for more general equation, that is the singular scalar curvature equation introduced in Section \ref{Singular cscK metrics}, see \thmref{L infty estimates Singular equation}, \thmref{gradient estimate} and \thmref{w2pestimates degenerate Singular equation}.
\end{rem}
This definition of the almost admissible solution further gives us the following estimates of $\tilde F_\eps$, 
\begin{lem}\label{Laplacian estimate: F}
Assume the $L^\infty$-estimate of $F_\eps$ holds. 
Then there exists a uniform nonnegative constant $C$ depending on $\theta$, $\|h_\theta+c_{\eps} \|_{C^2(\om)}$ and $\Theta_D$ such that
\begin{align*}
&\tilde F_\eps\leq C,\quad \tri \tilde F_\eps\geq \tri F_\eps-C,\\
&C[1+(\beta-1)S_\eps^{-1}]\tr_{\vphi_\eps}\om-R\geq \tri_{\vphi_\eps} \tilde F_\eps\geq -C \tr_{\vphi_\eps}\om-R.
\end{align*}
\end{lem}
\begin{proof}
The upper bound of $\tilde F_\eps$ follows from the $L^\infty$-estimate of $F_\eps$ and 
\begin{align*}
\tilde F_\eps=F_\eps+h_\theta+(\beta-1)\log S_\eps+c_\eps.
\end{align*}
The rest results are obtained from the following identities
\begin{align*}
\tri_{\vphi_\eps} \tilde F_\eps=\tr_{\vphi_\eps}[\theta+i\p\bar\p (h_\theta+\mathfrak h_\eps)]-R,\quad
i\p\bar\p \tilde F_\eps=i\p\bar\p (F_\eps + h_\theta+\mathfrak h_\eps)
\end{align*}
and the lower bound of $i\p\bar\p\mathfrak h_\eps=(\beta-1)i\p\bar\p\log S_\eps$ from \lemref{h eps}. 
\end{proof}

\begin{rem}
This lemma tells us that the directions of the inequalities above when $\beta>1$, are exactly opposite to their counterparts when $0<\beta<1$, where $\tilde F_\eps$ has lower bound and both $\tri_{\vphi_\eps} \tilde F_\eps$ and $\tri \tilde F_\eps$ have upper bound.
\end{rem}

\begin{lem}\label{Laplacian estimate: p F}
Assume the gradient estimate of $F_\eps$ holds i.e. 
\begin{align}\label{almost admissible C1 f sigmaD}
\|\p F_\eps\|_{L^\infty(\om_{\vphi_\eps})}\leq C.
\end{align} Then there exists a uniform constant $C$ depending on $\|h_\theta+c_{\eps} \|_{C^1(\om)}$ and $\sup_X  S_\eps^{-\frac{1}{2}}|\p S_\eps|_\om$ such that
\begin{align*}
 |\p F_\eps|^2\leq C \tr_\om\om_{\vphi_\eps},\quad
|\p\tilde F|^2\leq C[1+\tr_\om\om_{\vphi_\eps}+\frac{(\beta-1)^2}{S_\eps}].
\end{align*}
\end{lem}
\begin{proof}
We also make use of the inequality
\begin{align*}
|\p F_\eps|^2\leq |\p F_\eps|^2_{\vphi_\eps} \cdot \tr_\om\om_{\vphi_\eps}.
\end{align*}
We compute and see that 
\begin{align*}
|\p\tilde F_\eps|^2&=|\p (F_\eps+h_\theta+(\beta-1)\log S_\eps)|^2\\
&\leq C[1+|\p F_\eps|^2+(\beta-1)^2\frac{|\p |s|_h^2|^2}{S_\eps^2}]\\
&\leq C[1+\tr_\om\om_{\vphi_\eps}+\frac{(\beta-1)^2}{S_\eps}].
\end{align*}
In the last inequality, we use $|\p S_\eps|\leq  C_{2.4} |S_\eps|^{\frac{1}{2}}$ .
\end{proof}
\subsection{Approximation of log $K$-energy}In this subsection, we consider the degenerate cscK equation.
\begin{defn}
\begin{equation}\label{Degenerate cscK 1 approximation const}
\om_{\vphi_\eps}^n=e^{\tilde F_\eps} \om^n,\quad
\tri_{\vphi_\eps} \tilde F_\eps=\tr_{\vphi_\eps}(\theta-i\p\bar\p \tilde f_\eps)-\underline S_\beta.
\end{equation} 
\end{defn}

\begin{defn}[Approximate log $K$-energy]\label{log K energy approximation}
The approximate log $K$-energy is defined as
\begin{align}\label{log K energy approximation equation}
\nu^\eps_\beta(\vphi)=E_1(\vphi)
+J_{-\theta}(\vphi)-\frac{1}{V}\int_X \tilde f_\eps (\om^n-\om_\vphi^n), \quad \forall \vphi\in \mathcal H_\Om.
\end{align}
Here $\tilde f_\eps=-(\beta-1)\log S_\eps-h_\theta-c_\eps$.
\end{defn}

\begin{lem}[\cite{arXiv:1803.09506} Lemma 3.12]\label{1st derivative}
The first derivative of the approximate log $K$-energy is
\begin{align*}
\p_t\nu^\eps_\beta=\frac{1}{V}\int_X \vphi_t\left[\tri_\vphi \log\frac{\om^n_\vphi}{\om^n}-\tr_\vphi(\theta-i\p\bar\p \tilde f_\eps)+\underline S_\b\right]
\omega_\vphi^{n} .
\end{align*}
Its critical point satisfies the approximate degenerate cscK equation \eqref{Degenerate cscK 1 approximation const}.
\end{lem}

\begin{lem}\label{uniform energy bound}
Let $\vphi_\eps$ be the solution to the approximate degenerate cscK equation \eqref{Degenerate cscK 1 approximation const} with bounded entropy $E^\beta_{t,\eps}=\frac{1}{V}\int_X F_{\eps}\om_{\vphi_{\eps}}^n$. Then
\begin{align}
\nu^\eps_\beta(\vphi_\eps)\geq \nu_\beta(\vphi_\eps)-C.
\end{align}
\end{lem}
\begin{proof}
Comparing the log $K$-energy with its approximation \eqref{log K energy approximation equation}, we have
\begin{align*}
\nu_\beta^\eps-\nu_\beta=\frac{1}{V}\int_M [(\b-1)(\log S_\eps-\log |s|_h^2)+c_\eps](\om^n-\om_\vphi^n).
\end{align*} 
We see that
\begin{align*}
\frac{\beta-1}{V}\int_X( \log S_\eps - \log |s|_h^2)\om^n\geq 0.
\end{align*}
We use the volume ratio bound from Corollary \ref{Laplacian estimate: F} and the $L^\infty$-estimate of $F$ in \thmref{L infty estimates Singular equation} to show that
\begin{align*}
\frac{\beta-1}{V}\int_X(\log S_\eps-\log |s|_h^2 )\om_{\vphi_\eps}^n
&=\frac{\beta-1}{V}\int_X(\log S_\eps-\log |s|_h^2)e^{\tilde F_\eps} \om^n\\
&\leq C_1 \frac{\beta-1}{V}\int_X( \log S_\eps-\log |s|_h^2) \om^n\leq C_2.
\end{align*}
In which, the constants $C_1, C_2$ are independent of $\eps$. Thus the lemma is proved.
\end{proof}

We compute the second variation of the approximate log $K$-energy and obtain the upper bound of $\nu^\eps_\beta(\vphi_\eps)$.
\begin{prop}\label{upper bound of approximate log K energy prop}
The second derivative of the approximate log $K$-energy at the critical point is
\begin{align*}
\delta^2 \nu^\eps_\beta(u,v)
&=\frac{1}{V}\int_X (\p\p u, \p\p v)\omega_\vphi^{n}+\frac{1}{V}\int_X [Ric(\om)-\theta+i\p\bar\p\tilde f_\eps](\p u, \p v)\om^n_\vphi.
\end{align*}
\end{prop}
Note that, by \eqref{Ricomthetaeps}, we have $ i\p\bar\p \log S_\eps\geq- \frac{|s|^2_h}{S_\eps}\Theta_D$ and
\begin{align*}
Ric(\om)-\theta+i\p\bar\p\tilde f_\eps=Ric(\om_{\theta_\eps})-\theta=(1-\beta)\Theta_D+(1-\beta)i\p\bar\p\log S_\eps.
\end{align*} 
When $\beta\leq 1$, we have 
\begin{align*}
Ric(\om)-\theta+i\p\bar\p\tilde f_\eps\geq (1-\beta)\Theta_D-(1-\beta) \frac{|s|^2_h}{S_\eps}\Theta_D=(1-\beta)\frac{\eps}{S_\eps}\Theta_D.
\end{align*}
\begin{cor}
The approximate log $K$-energy is convex at its critical points, when $\beta=1$, or $0<\beta< 1$ and $\Theta_D\geq 0$.
\end{cor}


\section{Laplacian estimate for degenerate scalar curvature equation}\label{A priori estimates for approximate degenerate cscK equations}
\begin{thm}[Laplacian estimate]\label{cscK Laplacian estimate}
Suppose that $\vphi_\eps$ is an almost admissible solution to \eqref{Degenerate cscK 1 approximation} and the gradient estimate $\|\p F_\eps\|_{L^\infty(\om_{\vphi_\eps})}$ holds. 
Assume $\beta>1$ and the degenerate exponent satisfies that
\begin{equation}\label{Laplacian estimate: degenerate exponent}
\left\{
\begin{aligned}
 &\sigma_D=0,\text{ when }\beta>\frac{n+1}{2};\\
&\sigma_D>0,\text{ when } \beta\leq\frac{n+1}{2}.
   \end{aligned}
\right.
\end{equation} 
Then there exists a uniform constant $C$ such that
\begin{align}\label{Laplacian estimate gamma}
\tr_\om\om_{\vphi_\eps}\leq C \cdot S_\eps^{-\sigma_D}.
\end{align}
The uniform constant $C$ depends on the gradient estimate of $\vphi_\eps$, the $W^{2,p}$-estimate, $\beta, \sigma_D, c_\eps, n$ and
\begin{align*}
&\inf_{i\neq j}R_{i\bar i j\bar j}(\om), \quad C_S(\om), \quad \Theta_D,\quad \sup_X  S_\eps,\quad \sup_X  S_\eps^{-\frac{1}{2}}|\p S_\eps|_\om,\quad \|h_\theta+c_{\eps} \|_{C^2(\om)}.
\end{align*}

\end{thm}
\begin{proof}
In this proof, we let $C_1$ to be a constant determined in \eqref{Laplacian estimate: C1}. We denote the degenerate exponent $\sigma_D$ by $\gamma$ and also set
\begin{align*}
v:=\tr_\om\om_{\vphi_\eps},\quad w:=e^{-C_1\vphi_\eps} v, \quad u:= S_\eps^\gamma w=S_\eps^\gamma e^{-C_1\vphi_\eps} \tr_\om\om_{\vphi_\eps} .
\end{align*}
We will omit the lower index $\eps$ for convenience. We will apply the integration method and the proof of this theorem will be divided into several steps as following.
\end{proof}

\subsection{Step 1: integral inequality}
We now transform the approximate degenerate scalar curvature equations \eqref{Degenerate cscK 1 approximation} into an integral inequality.
\begin{prop}[Differential inequality]\label{Laplacian estimate: integral inequality pre prop}
Assume that $\gamma$ is a nonnegative real number.
Then there exist constants $C_{1.4}$ \eqref{Laplacian estimate: C1} and $C_{1.5}$ \eqref{Laplacian estimate: C15} such that the following differential inequality holds,
\begin{align*}
\tri_\vphi u
\geq C_{1.4}u  \tr_\vphi\om+e^{-C_1\vphi}\tri F\cdot S_\eps^\gamma-C_{1.5}(1+u).
\end{align*}
\end{prop}
\begin{proof}
According to Yau's computation,
\begin{align}\label{Yau computation}
\tri_\vphi \log (\tr_{\om}\om_\vphi) &\geq \frac{g_\vphi^{k\bar l}{R^{i\bar j}}_{k\bar l}(\om)g_{\vphi i\bar j}-S(\om)+\tri \tilde F}{\tr_{\om}\om_\vphi}\notag\\
		&\geq - C_{1.1}\tr_\vphi\om+\frac{\tri \tilde F}{\tr_{\om}\om_\vphi}.
\end{align} The constant $-C_{1.1}$ is the lower bound the bisectional curvature of $\om$, i.e. $\inf_{i\neq j}R_{i\bar i j\bar j}(\om)$.
Moreover, 
\begin{align}\label{Laplacian estimate: laplacian w}
\tri_\vphi \log w
		&\geq (C_1- C_{1.1})\tr_\vphi\om+\frac{\tri \tilde F}{v}-C_1 n.
\end{align}
Also, due to \lemref{h eps}, there exists $C_{1.2}$ depending on $\Theta_D$ such that
\begin{align}\label{Laplacian estimate: C12}
\tri_\vphi \log S_\eps\geq -C_{1.2}\tr_\vphi\om.
\end{align}
Adding together, we establish the differential inequality for $u= w S_\eps^\gamma$,
\begin{align*}
\tri_\vphi \log u\geq  (C_1- C_{1.1}-\gamma C_{1.2})\tr_\vphi\om+\frac{\tri \tilde F}{v}-C_1 n.
\end{align*}Note that we need the positivity of $\gamma$, i.e. $\gamma\geq 0$.

The lower bound of $\tri\tilde F\geq \tri F-C_{1.3}$ follows from Corollary \ref{Laplacian estimate: F}. The constant $C_{1.3}$ is 
\begin{align}\label{Laplacian estimate: C13}
-C_{1.3}&:=\inf_X [\tri(h_\theta+\mathfrak h_\eps)]
=\inf_X \tri h_\theta+(\beta-1)\inf_X\tri\log S_\eps\\
&= \inf_X \tri h_\theta-(\beta-1)C_{1.2}n\notag.
\end{align}
Choosing sufficiently large nonnegative $C_1$ such that 
\begin{align}\label{Laplacian estimate: C1}
C_{1.4}:=C_1-C_{1.1}-\gamma \cdot C_{1.2}\geq1,
\end{align}
we find
\begin{align*}
\tri_\vphi \log u
		\geq C_{1.4}\tr_\vphi\om+\frac{\tri F-C_{1.3}}{v}-C_1 n.
\end{align*}
Therefore, it is further written in the form
\begin{align}\label{Laplacian estimate: tri phi u}
\tri_\vphi u&= u\cdot[\tri_\vphi \log u+|\p\log u|_\vphi^2]\geq u \tri_\vphi \log u\notag\\
&\geq C_{1.4}u  \tr_\vphi\om+e^{-C_1\vphi}(\tri F-C_{1.3})S_\eps^\gamma-C_1 n u.
\end{align}

Setting
\begin{align}\label{Laplacian estimate: C15}
C_{1.5}:=2\max\{C_{1.3}e^{-C_1\inf_X\vphi}\sup_X S_\eps^\gamma,C_1n\},
\end{align} and inserting it to \eqref{Laplacian estimate: tri phi u}, we thus obtain the expected differential inequality.


\end{proof}

We introduce a notion $\tilde u=u+K$ for some nonnegative constant $K$ and set
\begin{align*}
&LHS_1:=\int_X|\p\tilde  u^p|^2S_\eps^\gamma\om^n_\vphi,\quad
LHS_2:=\int_X\tilde  u^{2p}u^{\frac{n}{n-1}}S_\eps^{-\frac{\gamma+\beta-1}{n-1}}\om_\vphi^n.
\end{align*} 
\begin{prop}[Integral inequality]\label{Laplacian estimate: integration cor}
There exists constants $C_{1.6}(C_1, \sup_X\vphi)$, and $C_{1.7}$ depending on $C_{1.4}$, $\inf_X\vphi$, $\sup_X F$, $\sup_X h_\theta$, $C_1$, $c_\eps$ such that
\begin{align}\label{Laplacian estimate: integration cor}
&\frac{2C_{1.6}}{p}LHS_1
+ C_{1.7}   LHS_2 
\leq RHS:=\mathcal N+C_{1.5}\int_X\tilde  u^{2p+1}\om_\vphi^n,\\
&\mathcal N:=-\int_Xe^{-C_1\vphi}\tilde u^{2p}\tri F S_\eps^{\gamma}e^{\tilde F}\om^n.\notag
\end{align}

\end{prop}
\begin{proof}
Multiplied with $\tilde u^{2p}, p\geq 1$ and integrated with respect to $\om^n_\vphi$ on $X$, the differential inequality in Proposition \ref{Laplacian estimate: integral inequality pre prop} yields
\begin{align*}
&\frac{2}{p}\int_X\tilde  u|\p\tilde  u^p|^2_\vphi\om^n_\vphi
=-\int_X\tilde  u^{2p}\tri_\vphi \tilde  u\om^n_\vphi\\
&\leq -\int_X\tilde u^{2p}[C_{1.4}u  \tr_\vphi\om+e^{-C_1\vphi}\tri F S_\eps^\gamma-C_{1.5} \tilde  u]\om^n_\vphi.
\end{align*}
Replacing $u$ by $e^{-C_1\vphi}vS_\eps^\gamma$ and using the relation between two norms $|\cdot|_\vphi$ and $|\cdot|$, i.e. $v |\cdot|_\vphi\geq |\cdot|$, we observe the lower bound of the gradient term
\begin{align}\label{Laplacian estimate: C16}
\frac{2}{p}\int_X\tilde  u|\p\tilde  u^p|^2_\vphi\om^n_\vphi
&\geq \frac{2}{p}\int_X  u|\p\tilde  u^p|^2_\vphi\om^n_\vphi
=\frac{2}{p}\int_Xe^{-C_1\vphi}v S_\eps^\gamma |\p\tilde  u^p|^2_\vphi\om^n_\vphi\notag\\
&\geq C_{1.6}(C_1,\sup_X\vphi) \frac{2}{p}\int_X |\p\tilde  u^p|^2 S_\eps^\gamma\om^n_\vphi.
\end{align}
Consequently, we regroup positive terms to the left hand side of the integral inequality to conclude that
\begin{align}\label{Laplacian estimate: integral inequality pre}
&\frac{2C_{1.6}}{p}\int_X|\p\tilde  u^p|^2S_\eps^\gamma\om^n_\vphi
+ C_{1.4} \int_X\tilde  u^{2p} u \tr_\vphi\om\om^n_\vphi
\leq RHS.
\end{align}

Substituting $\tilde F=F+h_\theta+(\beta-1)\log S_\eps+c_\eps$ into the fundamental inequality 
$\tr_\vphi\om\geq v^{\frac{1}{n-1}}e^{-\frac{\tilde F}{n-1}}$, we see that
\begin{align*}
 \tr_\vphi\om
\geq u^{\frac{1}{n-1}}e^{\frac{C_1\vphi}{n-1}}S_\eps^{\frac{-\gamma}{n-1}}e^{-\frac{F+h_\theta+(\beta-1)\log S_\eps+c_\eps}{n-1}}
\geq C_{1.7}u^{\frac{1}{n-1}}S_\eps^{-\frac{\gamma+\beta-1}{n-1}}.
\end{align*} 
Thus
\begin{align*}
\int_X\tilde  u^{2p} u \tr_\vphi\om\om^n_\vphi
\geq C_{1.7}\int_X\tilde  u^{2p}u^{\frac{n}{n-1}}S_\eps^{-\frac{\gamma+\beta-1}{n-1}}\om_\vphi^n.
\end{align*} 
The constant $C_{1.7}$ depends on $n,C_1,\inf\vphi, \sup F,\sup h_\theta,c_\eps$.
Therefore, \eqref{Laplacian estimate: integration cor} is obtained by inserting this inequality to \eqref{Laplacian estimate: integral inequality pre}.
\end{proof}
\begin{rem}
The constant $C_{1.3}$ depends on the cone angle $\beta\geq1$.
\end{rem}
\begin{rem}
Actually, we could choose $K=0$ and derive the integral inequality for $u$,
\begin{align*}
\frac{2C_{1.6}}{p}\int_X|\p  u^p|^2S_\eps^\gamma\om^n_\vphi
+ C_{1.4} \int_X  u^{2p+1}  \tr_\vphi\om\om^n_\vphi
\leq RHS.
\end{align*} 

We denote $k:=(2p+1)\gamma+(\beta-1)$.
With the help of the $L^\infty$-estimates of $\vphi$ and $F$, this inequality is simplified to be
\begin{align}\label{Laplacian estimate: inverse weighted inequalities general}
&\int_Xu^{2p+1}\tr_\vphi\om\om^n_\vphi
\geq C_{1.7}   \int_Xw^{2p+\frac{n}{n-1}}S_\eps^{k-\frac{\beta-1}{n-1}}\om^n.
\end{align} 
\end{rem}

\subsection{Step 2: nonlinear term containing $\tri F$}\label{Step 2}
The term $\mathcal N$ containing $\tri F$ in the $RHS$ of the integral inequality \eqref{Laplacian estimate: integration cor} requires further simplification by integration by parts.
\begin{prop}[Integral inequality]\label{Laplacian estimate: integration inequality pro} 
There exists a constant $C_{2.5}$ depending on $\inf_X\vphi$, $\sup_X F$, $\|\p\vphi\|_{L^{\infty}(\om)}$, $\|\p F\|_{L^{\infty}(\om_\vphi)}$, the constant in \lemref{Laplacian estimate: F} and \lemref{Laplacian estimate: p F}, and the constants $C_{1.5}$, $
C_{1.6}$, $C_{1.7}$, $\beta$
such that
\begin{align}
&p^{-1}LHS_1
+  LHS_2 \leq RHS
\leq p C_{2}\cdot RHS_1,\label{Laplacian estimate: integral inequality}\\
&RHS_1:=\int_X \tilde u^{2p}[ u+1+u^{\frac{1}{2}}S_\eps^{\frac{\gamma}{2}}+u^{\frac{1}{2}}S_\eps^{\frac{\gamma-\sigma_D^1}{2}}
+u^{\frac{1}{2}} S_\eps^{\frac{\gamma-1}{2}} 
 ] \om_\vphi^n.\label{Laplacian estimate: RHS1 defn}
\end{align}
\end{prop}
\begin{proof}
By integration by parts, we split the nonlinear term $\mathcal N$ into four sub-terms,
\begin{align}\label{Laplacian estimate: tri F}
&\mathcal N:=I+II+III+IV
=-C_1\int_X e^{-C_1\vphi}(\p \vphi,\p F)\tilde u^{2p}S_\eps^{\gamma} e^{\tilde F}\om^n\\
&+2p\int_Xe^{-C_1\vphi} \tilde u^{2p-1}(\p \tilde u, \p F) S_\eps^{\gamma} e^{\tilde F}\om^n\notag\\
&+\int_Xe^{-C_1\vphi}\tilde u^{2p}(\p F,\p S_\eps^{\gamma}) e^{\tilde F}\om^n\notag
+\int_Xe^{-C_1\vphi}\tilde u^{2p}(\p F,\p \tilde F)S_\eps^{\gamma} e^{\tilde F}\om^n\notag.
\end{align}
The assumption on the gradient estimate of the volume ratio $F$ and \lemref{Laplacian estimate: p F} give us that
$$
|\p F|\leq C e^{\frac{C_1}{2}\vphi}S_\eps^{-\frac{\gamma}{2}}u^{\frac{1}{2}}.
$$
Inserting it with the gradient estimate of $\vphi$ in \eqref{almost admissible C1 sigmaD}. i.e. $$|\p\vphi|\leq C S_\eps^\frac{-\sigma^1_D}{2},\quad \sigma^1_D<1$$ to the first term, we see that
\begin{align*}
I&\leq C_1\int_X e^{-C_1\vphi}|\p \vphi||\p F|\tilde u^{2p}S_\eps^{\gamma} \om_\vphi^n
\leq  C_{2.1}\int_X e^{-\frac{C_1}{2}\vphi}\tilde u^{2p}u^{\frac{1}{2}}S_\eps^{\frac{\gamma-\sigma_D^1}{2}} \om_\vphi^n,
\end{align*}
where, the constant $C_{2.1}$ depends on $\|S_\eps^\frac{\sigma^1_D}{2}\p\vphi\|_{L^{\infty}(\om)},\|\p F\|_{L^{\infty}(\om_\vphi)}$.

The second term is nothing but using H\"older's inequality and the bound of $|\p F|^2$,
\begin{align*}
II&=2\int_Xe^{-C_1\vphi}\tilde  u^{p}(\p\tilde  u^p, \p F) S_\eps^{\gamma} \om_\vphi^n\\
&\leq \frac{C_{1.6}}{p}\int_X|\p\tilde  u^p|^2S_\eps^\gamma\om^n_\vphi
+\frac{p}{C_{1.6}}\int_Xe^{-2C_1\vphi}\tilde u^{2p} |\p F|^2 S_\eps^{\gamma} \om^n_\vphi\\
& \leq \frac{C_{1.6}}{p}\int_X|\p\tilde  u^p|^2S_\eps^\gamma\om^n_\vphi
+p C_{2.2}\int_Xe^{-C_1\vphi}\tilde  u^{2p}u  \om^n_\vphi.
\end{align*}

In order to bound the third term, we use $|\p F|$ again and the fact that
\begin{align}\label{Laplacian estimate: C24}
|\p S^\gamma_\eps|\leq \gamma C_{2.4} |S_\eps|^{\gamma-\frac{1}{2}}.
\end{align}
As a result, we have
\begin{align*}
III\leq  \gamma C_{2.3}\int_Xe^{-\frac{C_1}{2}\vphi} \tilde u^{2p}u^{\frac{1}{2}} S_\eps^{\frac{\gamma-1}{2}}  \om^n_\vphi.
\end{align*}

Due to \lemref{Laplacian estimate: p F} again, we get
$$|\p\tilde F|^2\leq C[1+e^{\frac{C_1}{2}\vphi}S_\eps^{-\frac{\gamma}{2}}u^{\frac{1}{2}}+(\beta-1)S_\eps^{-\frac{1}{2}}].$$
Then the fourth term is bounded by
\begin{align*}
IV&\leq\int_Xe^{-\frac{C_1}{2}\vphi}\tilde u^{2p}u^{\frac{1}{2}}|\p \tilde F|S_\eps^{\frac{\gamma}{2}} \om_\vphi^n\\
&\leq C_{2.5}[\int_Xe^{-\frac{C_1}{2}\vphi}\tilde u^{2p}u^{\frac{1}{2}}S_\eps^{\frac{\gamma}{2}} \om_\vphi^n
+\int_X\tilde u^{2p}u\om_\vphi^n\\
&+(\beta-1)\int_Xe^{-\frac{C_1}{2}\vphi}\tilde u^{2p}u^{\frac{1}{2}}S_\eps^{\frac{\gamma-1}{2}} \om_\vphi^n].
\end{align*}

Inserting them back to \eqref{Laplacian estimate: tri F}, we have the bound of $\mathcal N$.
Substituting $\mathcal N$ in \eqref{Laplacian estimate: integral inequality pre} and note that $\sigma_D^1< 1$, 
we have thus proved \eqref{Laplacian estimate: integral inequality}. 
\end{proof}
\begin{rem}
When $\gamma=0$,  the third term $III=0$.
\end{rem}
\begin{rem}
When $\beta=1$, $\p\tilde F=\p F+\p h_\theta$. Then the fourth term $IV=\int_Xe^{-C_1\vphi}u^{2p}(\p F,\p \tilde F)S_\eps^{\gamma} e^{\tilde F}\om^n$.
\end{rem}
\begin{rem}\label{4th term}
In the third and the fourth term, the power of $S_\eps$ loses $\frac{1}{2}$, which cause troubles.
\end{rem}

\subsection{Step 3: rough iteration inequality}
We will apply the Sobolev inequality to the gradient term 
$$
LHS_1=\int_X|\p\tilde u^p|^2S_\eps^\gamma\om^n_\vphi
$$
in \eqref{Laplacian estimate: integral inequality}.
We set
\begin{align*}
 k_\gamma:=\gamma+\beta-1+\sigma, \quad \chi:=\frac{n}{n-1}, \quad \tilde\mu:=S_\eps^{k_\gamma \chi}\om^n.
\end{align*} 
It is direct to see that
$ k+\sigma=k_\gamma+2p\gamma$ and
\begin{align*}
\|\tilde u\|^{2p\chi}_{L^{2p\chi}(\tilde\mu)}=\int_X(\tilde u^{2p}S_\eps^{k_\gamma})^\chi\om^n=\int_X\tilde u^{2p\chi}\tilde \mu.
\end{align*}

\begin{prop}[Rough iteration inequality]\label{Rough iteration inequality prop}
There exists a constant $C_{3}$ depending on $C_{1.6}$, $C_{1.7}$, the dependence in $C_2$, $\inf_X F$, $\inf_X h_\theta$, $c_\eps$ and the Sobolev constant $C_S(\om)$ such that
\begin{align}\label{Laplacian estimate: Rough iteration inequality}
\|\tilde u\|^{2p}_{L^{2p\chi}(\tilde\mu)}+p  LHS_2
\leq C_{3}(p^2  RHS_1 +  RHS_2+1)
\end{align}
where $RHS_1$ is given in \eqref{Laplacian estimate: RHS1 defn} and 
\begin{align}\label{Laplacian estimate: RHS2 defn}
RHS_2:=\int_X (u^2 \tilde u^{2p-2}S_\eps^{\gamma+\sigma}
+C_{3.1}  u^2 \tilde u^{2p-2}S_\eps^{\gamma+\sigma-1} ) \om_\vphi^n.
\end{align}
\end{prop}

\begin{proof}We need to deal with the weights. 
Recall the equations of the volume form $\om^n_\vphi$ from \eqref{Degenerate cscK approximation} and the volume form of the approximate reference metric $\om^n_{\theta_\eps}$ by \eqref{Rictheta approximation}, 
\begin{align*}
\om^n_{\vphi_\eps}=e^{\tilde F_\eps}\om,\quad e^{\tilde F_\eps}=e^{F_\eps+h_\theta+c_\eps}S_\eps^{\beta-1}.
\end{align*}
We assume $S_\eps\leq 1$.
We see that there exists a constant $C_{3.0}$ depending on $\inf_X F$, $\inf_X h_\theta$ and $c_\eps$ such that
\begin{align*}
LHS_1
&\geq C_{3.0} \int_X|\p\tilde  u^p|^2S_\eps^{\gamma+\beta-1}\om^n
\geq C_{3.0} \int_X|\p \tilde u^p|^2S_\eps^{k_\gamma}\om^n,\quad \sigma\geq 0.
\end{align*}

Using \lemref{Laplacian estimate: key trick} with $p_1=1, p_2=p-1$, we have
\begin{align*}
LHS_1
\geq C_{3.0} \int_X|\p (u\tilde u^{p-1})|^2S_\eps^{k_\gamma}\om^n.
\end{align*}
Further calculation shows that 
\begin{align}\label{Laplacian estimate: main term of LHS}
=C_{3.0} [ \int_X|\p (u\tilde u^{p-1} S_\eps^{\frac{k_\gamma}{2}})|^2\om^n- \int_Xu^2\tilde  u^{2p-2} |\p S_\eps^{\frac{k_\gamma}{2}}|^2\om^n].
\end{align}

We now make use of the Sobolev inequality to the first term in \eqref{Laplacian estimate: main term of LHS} with $f=u\tilde u^{p-1} S_\eps^{\frac{k_\gamma}{2}}$, which states 
\begin{align*}
\|f\|_{L^{2\chi}(\om)}\leq C_S (\|\p f\|_{L^{2}(\om)}+\| f\|_{L^{2}(\om)}),
\end{align*} 
that is
\begin{align*}
 \int_X|\p (u\tilde u^{p-1} S_\eps^{\frac{k_\gamma}{2}})|^2\om^n
 &\geq  C_S^{-1} ( \int_X|u\tilde  u^{p-1} S_\eps^{\frac{k_\gamma}{2}}|^{2\chi}\om^n)^{\chi^{-1}}
 - \int_Xu^2\tilde  u^{2p-2} S_\eps^{k_\gamma}\om^n.
\end{align*}
Note that the power of the weight is increasing from $k_\gamma$ to $k_\gamma\chi$.

Substituting $u=\tilde u-K$ and $\tilde\mu=S_\eps^{k_\gamma \chi}\om^n$ in the main term, we get
\begin{align*}
& \int_X|u\tilde  u^{p-1} S_\eps^{\frac{k_\gamma}{2}}|^{2\chi}\om^n
= \int_X|\tilde u-K |^{2\chi}\tilde  u^{(p-1)2\chi}\tilde\mu\\
&\geq C(n)[\int_X\tilde  u^{2p\chi}\tilde\mu
-K^{2\chi}\int_X\tilde  u^{2(p-1)\chi}\tilde\mu].
\end{align*}
With the help of Young's inequality
\begin{align*}
\tilde  u^{2(p-1)\chi}\leq\frac{p-1}{p} \tilde  u^{2p\chi}+1\leq \tilde  u^{2p\chi}+1,
\end{align*}
choosing $0<K<1$ such that $1-K^{2\chi}\geq \frac{1}{2}$, we obtain
\begin{align*}
\int_X|u\tilde  u^{p-1} S_\eps^{\frac{k_\gamma}{2}}|^{2\chi}\om^n
\geq C(n)[ \frac{1}{2}\int_X\tilde  u^{2p\chi}\tilde\mu-K^{2\chi}]= \frac{C(n)}{2}[\int_X\tilde  u^{2p\chi}\tilde\mu-1].
\end{align*}

Using $|\p S_\eps|^2\leq C_{2.4}S_\eps$, we can estimate the second term in \eqref{Laplacian estimate: main term of LHS},
\begin{align}\label{Laplacian estimate: trouble term} 
\int_Xu^2 \tilde u^{2p-2} |\p S_\eps^{\frac{k_\gamma}{2}}|^2\om^n
&= \frac{k_\gamma^2}{4} \int_Xu^2 \tilde u^{2p-2} S_\eps^{k_\gamma-2}|\p S_\eps|^2\om^n\notag\\
&\leq C_{3.1} \int_X u^2 \tilde u^{2p-2} S_\eps^{k_\gamma-1}  \om^n.
\end{align}

At last, we add these inequalities together to see that 
\begin{align*}
LHS_1\geq C_{3.2} \{&C_S^{-1}  \left( \int_X\tilde  u^{2p\chi}\tilde\mu-1 \right)^{\chi^{-1}}
 - \int_Xu^2 \tilde u^{2p-2} S_\eps^{k_\gamma}\om^n\\
 &-C_{3.1} \int_Xu^2 \tilde u^{2p-2}S_\eps^{k_\gamma-1}  \om^n\}.
\end{align*}
Inserting this inequality to the integral inequality \eqref{Laplacian estimate: integral inequality}, we obtain the rough iteration inequality \eqref{Laplacian estimate: Rough iteration inequality}.

\end{proof}

\begin{rem}\label{Laplacian estimate: lose weight}
We observe that $1$ is subtracted from the weight of $S_\eps$ in the second term of $RHS_2$ \eqref{Laplacian estimate: RHS2 defn}, which causes difficulties presented in the weighted inequality, Proposition \ref{Weighted inequality}. We will solve this problem by making use of the inverse weighted inequalities, Proposition \ref{inverse weighted inequalities}.
\end{rem}
\begin{rem}
When $\beta=1$ and $\gamma=\sigma=0$, the trouble term \eqref{Laplacian estimate: trouble term} vanishes.
\end{rem}

We end this section by computing the auxiliary inequality.
\begin{lem}\label{Laplacian estimate: key trick}
We write $p=p_1+p_2$ with $p_1,p_2\geq 0$.
\begin{align*}
 \int_X|\p \tilde u^p|^2S_\eps^{k_\gamma}\om^n\geq \int_X|\p (u^{p_1}\tilde u^{p_2})|^2S_\eps^{k_\gamma}\om^n.
\end{align*}
\end{lem}
\begin{proof}
It is a direct computation
\begin{align*}
&\p (u^{p_1}\tilde u^{p_2})
=\p u^{p_1}\tilde u^{p_2}+u^{p_1}\p(\tilde u^{p_2})
=p_1 u^{p_1-1}\p u\tilde u^{p_2}+u^{p_1}p_2\tilde u^{p_2-1}\p \tilde u\\
&=\p\tilde u u ^{p_1-1} \tilde u^{p_2-1}[p_1\tilde u+p_2 u]
\leq p \p\tilde u \tilde u^{p-1}=\p(\tilde u^p),
\end{align*}
where we use $u\leq \tilde u$.
\end{proof}

\subsection{Step 4: weighted inequality}\label{Step 4}
We compare the left term $\|\tilde u\|^{2p}_{L^{2p\chi}(\tilde\mu)}$ of the rough iteration inequality, Proposition \ref{Rough iteration inequality prop}, with the right terms in ${RHS_1}$ and ${RHS_2}$, which are of the form
\begin{align*}
\int_X\tilde u^{2p} S_\eps^{\gamma+\sigma-k'}\om_\vphi^n.
\end{align*}

\begin{prop}[Weighted inequality]\label{Weighted inequality}
Assume that $n\geq 2$ and $ k'<1$.
Then there exists $1<a<\chi=\frac{n}{n-1}$ such that
\begin{align*}
\int_X\tilde u^{2p} S_\eps^{\gamma+\sigma-k'}\om_\vphi^n
\leq C\int_X\tilde u^{2p} S_\eps^{k_\gamma-k'}\om^n 
\leq C_{4.1} \|\tilde u\|^{2p}_{L^{2pa}(\tilde\mu)}
\end{align*}
where $C_{4.1}=\| S_\eps^{k_\gamma-k_\gamma\chi-k'}\|_{L^{c}(\tilde\mu)}$ is finite for some $c>n$.
\end{prop}
\begin{proof}
From $\tilde\mu=S_\eps^{k_\gamma \chi}\om^n=S_\eps^{(\gamma+\beta-1+\sigma) \chi}\om^n$, we compute
\begin{align*}
&\int_X\tilde u^{2p} S_\eps^{k_\gamma-k'}\om^n
=\int_X \tilde u^{2p}S_\eps^{k_\gamma-k_\gamma\chi-k'} \tilde\mu.
\end{align*}
By the generalisation of H\"older's inequality with $\frac{1}{a}+\frac{1}{c}=1$, this term is dominated by
\begin{align*}
 \|\tilde u\|^{2p}_{L^{2pa}(\tilde \mu)}(\int_X S_\eps^{(k_\gamma-k_\gamma\chi-k')c} \tilde\mu)^{\frac{1}{c}}.
\end{align*}
In order to make sure the last integral is finite, it is sufficient to ask $2(k_\gamma-k_\gamma\chi-k')c+2k_\gamma\chi+2n>0$, which is equivalent to
\begin{align*}
c<n\frac{k_\gamma+n-1}{k_\gamma+k'(n-1)}:=c_0.
\end{align*} 
Since $k'<1$, we have $c_0>n$. Then, we could choose $c$ between $n$ and $c_0$ such that $a<\frac{n}{n-1}$.

\end{proof}


\subsection{Step 5: inverse weighted inequality}\label{Step 3: inverse weighted inequalities}
Our tour to bound each term in $RHS_1$ \eqref{Laplacian estimate: RHS1 defn} and $RHS_2$ \eqref{Laplacian estimate: RHS2 defn}
\begin{align*}
 RHS_1&=\int_X \tilde u^{2p}[u+1+u^{\frac{1}{2}}S_\eps^{\frac{\gamma}{2}}+u^{\frac{1}{2}}S_\eps^{\frac{\gamma-\sigma_D^1}{2}}
+u^{\frac{1}{2}} S_\eps^{\frac{\gamma-1}{2}} 
 ] \om_\vphi^n,\\
RHS_2&=\int_X (u^2\tilde u^{2p-2}S_\eps^{\gamma+\sigma}
+C_{3.1}u^2  \tilde u^{2p-2}S_\eps^{\gamma+\sigma-1} ) \om_\vphi^n,
\end{align*}
is via applying Young's inequality repeatedly, with the help of the positive term
\begin{align*}
 LHS_2=p\int_X\tilde  u^{2p}u^{\frac{n}{n-1}}S_\eps^{-\frac{\gamma+\beta-1}{n-1}}\om_\vphi^n.
\end{align*}
According to Proposition \ref{Weighted inequality}, the second term in $RHS_2$ is the trouble term.

\begin{prop}[Inverse weighted inequality]\label{inverse weighted inequalities}
Assume that the parameters $\sigma$ and $\gamma$ satisfy $\sigma+\gamma<1$ and
\begin{equation}\label{Laplacian estimate: RHS2 parameters condition}
\left\{
\begin{aligned}
 &\sigma=\gamma=0,\text{ when }\beta>n;\\
&\frac{1}{2}\geq\sigma>\frac{n-\beta}{n-1},\quad \gamma=0,\text{ when } \frac{n+1}{2}< \beta\leq n;\\
&\sigma<\frac{1}{n+1},\quad \gamma>(1-\sigma)\frac{n-1}{n},\text{ when } 1\leq \beta\leq \frac{n+1}{2}.
   \end{aligned}
\right.
\end{equation}
Then there exists an exponent $k'<1$ such that
\begin{align*}
\|\tilde u\|^{2p}_{L^{2p\chi}(\tilde\mu)}
\leq C_5 [p^3  \int_X \tilde u^{2p} S_\eps^{\gamma+\sigma-k'}  \om_\vphi^n+1].
\end{align*}
\end{prop}
\begin{proof}
The proof of the inverse weighted inequality is divided into \lemref{Laplacian estimate: RHS1 good} for $RHS_1$, \lemref{Laplacian estimate: RHS2 bound} for $RHS_2$ and \lemref{Laplacian estimate: criteria} for examining the criteria. Adding the resulting inequalities of $RHS_1$ and $RHS_2$, we have
\begin{align*}
p^2  RHS_1 +  RHS_2&\leq 
 \tau pLHS_2+ p^3 C(\tau) \int_X \tilde u^{2p} S_\eps^{\gamma+\sigma-\max\{k_2',k_5'\}}  \om_\vphi^n\\
&+\tau pLHS_2+  \frac{C(\tau)}{p} \int_X \tilde u^{2p}S_\eps^{\gamma+\sigma-k'_7}  \om_\vphi^n.
\end{align*}
We set a new $k'$ to be $\max\{k_2',k_5',k_7'\}$.
Inserting this inequality to the rough iteration inequality \eqref{Laplacian estimate: Rough iteration inequality} and choosing sufficiently small $\tau$, we therefore obtain the asserted inequality.

\end{proof}
\begin{lem}\label{Laplacian estimate: RHS1 good}
Assume that 
\begin{align}\label{Laplacian estimate: RHS1 good condition}
 \sigma<\frac{\beta}{n+1},\quad \gamma+\sigma<1.
\end{align} Let $k'=\max\{1+\sigma-\frac{\beta}{n+1},\gamma+\sigma\}$. Then
\begin{align*}
 RHS_1
 \leq \frac{
 \tau}{p}LHS_2+ p C(\tau) \int_X \tilde u^{2p} S_\eps^{\gamma+\sigma-\max\{k_2',k_5'\}}  \om_\vphi^n.
\end{align*}
The exponents $k'_2,k_5'$ are given in the following proof, see \eqref{Laplacian estimate: RHS1 exponents k2} and \eqref{Laplacian estimate: RHS1 exponents}, respectively.
\end{lem}
\begin{proof}
We now establish the estimates for the five terms in $RHS_1$.
The 1st one is decomposed as
\begin{align*}
&\int_X \tilde u^{2p} u  \om_\vphi^n
=\int_X (\tilde  u^{2p}u^{\frac{n}{n-1}}S_\eps^{-\frac{\gamma+\beta-1}{n-1}})^{\frac{1}{a_1}}\tilde u^{2p\frac{1}{b_1}}u^{1-\frac{n}{n-1}\frac{1}{a_1}}S_\eps^{\frac{\gamma+\beta-1}{n-1}\frac{1}{a_1}} \om_\vphi^n.
\end{align*}
By Young's inequality with small $\tau$, it implies that
\begin{align*}
\int_X \tilde u^{2p} u  \om_\vphi^n
\leq  \frac{\tau}{ 4p }LHS_2+ p C \int_X \tilde u^{2p} u^{(1-\frac{n}{n-1}\frac{1}{a_1})b_1}S_\eps^{\frac{\gamma+\beta-1}{n-1}\frac{b_1}{a_1}}  \om_\vphi^n.
\end{align*}
The conjugate exponents we choose are
$a_1=\frac{n}{n-1}$ and $b_1=n.$
Accordingly, we have the exponent over $u$ is zero and
\begin{align*}
k_1:=\frac{\gamma+\beta-1}{n-1}\frac{b_1}{a_1}=\gamma+\beta-1.
\end{align*}

The estimate of the 3rd term is proceeded in the same way,
\begin{align*}
\int_X \tilde u^{2p} u^{\frac{1}{2}}S_\eps^{\frac{\gamma}{2}} \om_\vphi^n
&=\int_X (\tilde  u^{2p}u^{\frac{n}{n-1}}S_\eps^{-\frac{\gamma+\beta-1}{n-1}})^{\frac{1}{a_2}}\tilde u^{2p\frac{1}{b_2}}u^{\frac{1}{2}-\frac{n}{n-1}\frac{1}{a_2}}S_\eps^{\frac{\gamma}{2}+\frac{\gamma+\beta-1}{n-1}\frac{1}{a_2}} \om_\vphi^n\\
&\leq \frac{\tau}{ 4p}LHS_2+ p C \int_X \tilde u^{2p} u^{(\frac{1}{2}-\frac{n}{n-1}\frac{1}{a_2})b_2}S_\eps^{(\frac{\gamma}{2}+\frac{\gamma+\beta-1}{n-1}\frac{1}{a_2})b_2}  \om_\vphi^n.
\end{align*}
The exponent $a_2$ is set to be $\frac{2n}{n-1}$. Hence, the exponent above $u$ vanishes. Moreover, $b_2=\frac{2n}{n+1}$ and
\begin{align*}
k_2:=(\frac{\gamma}{2}+\frac{\gamma+\beta-1}{n-1}\frac{1}{a_2})b_2
= \gamma+\frac{\beta-1}{n+1}.
\end{align*}

The 4th and 5th terms are treated by Young's inequality with small $\tau$, as well. The estimate for the 4th term is  
\begin{align*}
&\int_X \tilde u^{2p} u^{\frac{1}{2}}S_\eps^{\frac{\gamma-\sigma_D^1}{2}} \om_\vphi^n
\leq \frac{\tau}{4 p}LHS_2+ p C \int_X \tilde u^{2p} u^{(\frac{1}{2}-\frac{n}{n-1}\frac{1}{a_2})b_2}S_\eps^{k_4}  \om^n
\end{align*}
and the exponent is
\begin{align*}
k_4:=(\frac{\gamma-\sigma_D^1}{2}+\frac{\gamma+\beta-1}{n-1}\frac{1}{a_2})b_2
= \gamma+\frac{\beta-1-n\sigma_D^1}{n+1}.
\end{align*}

While, the estimate for the 5th term is 
\begin{align*}
&\int_X \tilde u^{2p} u^{\frac{1}{2}}S_\eps^{\frac{\gamma-1}{2}} \om_\vphi^n
\leq \frac{\tau}{ 4p}LHS_2+ p C \int_X \tilde u^{2p} u^{(\frac{1}{2}-\frac{n}{n-1}\frac{1}{a_2})b_2}S_\eps^{k_5}  \om^n,
\end{align*}
with the exponent satisfying
\begin{align*}
k_5:=(\frac{\gamma-1}{2}+\frac{\gamma+\beta-1}{n-1}\frac{1}{a_2})b_2= \gamma+\frac{\beta-1-n}{n+1}.
\end{align*}

We let $k_i'$ satisfy $k_i=\gamma+\sigma-k_i'$. Then we summary the exponents from the above estimates to see that
\begin{align}\label{Laplacian estimate: RHS1 exponents}
&k_1'=\gamma+\sigma-(\gamma+\beta-1)=\sigma-(\beta-1),\notag\\
&k_3'=\gamma+\sigma-( \gamma+\frac{\beta-1}{n+1})=\sigma-\frac{\beta-1}{n+1},\notag\\
&k_4'=\gamma+\sigma-(\gamma+\frac{\beta-1-n\sigma_D^1}{n+1})
=\sigma-\frac{\beta-1-n\sigma_D^1}{n+1},\notag\\
&k_5'=\gamma+\sigma-(\gamma+\frac{\beta-1-n}{n+1})=\sigma-\frac{\beta-1-n}{n+1}.
\end{align}
Since $\sigma_D^1\leq 1$, we observe that $k_5'$ is the largest one among these four exponents.
We further compute that
\begin{align*}
k_5'-1=\sigma-\frac{\beta}{n+1},
\end{align*}
which is negative under the hypothesis of our lemma, i.e. $\sigma<\frac{\beta}{n+1}$.

The integrand of the 2nd term is decomposed as,
\begin{align*}
\int_X\tilde u^{2p}\om_\vphi^n=
\int_X\tilde u^{2p}S_\eps^{\gamma+\sigma-k_2'}\om_\vphi^n, 
\end{align*}
where we choose the exponent 
\begin{align}\label{Laplacian estimate: RHS1 exponents k2}
k_2':=\gamma+\sigma<1.
\end{align}
Therefore, we set the exponent to be $\max\{k_2',k_5'\}$ and obtain the required inequality for $RHS_1$.

\end{proof}

We then derive the estimates for the terms in $RHS_2$.
\begin{lem}\label{Laplacian estimate: RHS2 bound}
Assume the following condition holds
\begin{align}\label{Laplacian estimate: RHS2 condition}
\gamma>1-\frac{\beta}{n}-\sigma(1-\frac{1}{n}):=\gamma_0.
\end{align} Then 
\begin{align*}
RHS_2
 \leq \tau pLHS_2+  \frac{C(\tau)}{p} \int_X \tilde u^{2p}S_\eps^{\gamma+\sigma-k'_7}  \om_\vphi^n,
\end{align*}
where $k'_7<1$ is given in \eqref{Laplacian estimate: k7}.
\end{lem}
\begin{proof}The first term in $RHS_2$ is a good term. On the other hand, the trouble term is treated with the help of $LHS_2$,
\begin{align*}
&RHS_2^2:=\int_X u^2 \tilde u^{2p-2}S_\eps^{\gamma+\sigma-1}  \om_\vphi^n\\
&=\int_X (\tilde  u^{2p}u^{\frac{n}{n-1}}S_\eps^{-\frac{\gamma+\beta-1}{n-1}})^{\frac{1}{a_7}}
\tilde u^{2p\frac{1}{b_7}-2}  u^{2-\frac{n}{n-1}\frac{1}{a_7}}S_\eps^{\gamma+\sigma-1+\frac{\gamma+\beta-1}{n-1}\frac{1}{a_7}} \om_\vphi^n,
\end{align*}
by applying an argument analogous to the proof of $RHS_1$.
By Young's inequality, we get
\begin{align*}
RHS_2^2
\leq \tau LHS_2+  C \int_X \tilde u^{2p-2b_7}  u^{2b_7-\frac{n}{n-1}\frac{b_7}{a_7}}S_\eps^{k_7}  \om_\vphi^n.
\end{align*}
Using $\tilde u\geq K$ and choosing $a_7=\frac{n}{2(n-1)}$,
\begin{align*}
RHS_2^2\leq \tau LHS_2+  C \int_X \tilde u^{2p}S_\eps^{k_7}  \om_\vphi^n.
\end{align*}
The exponent
\begin{align*}
k_7:=[\gamma+\sigma-1+\frac{\gamma+\beta-1}{n-1}\frac{1}{a_7}]b_7.
\end{align*}
Direct computation shows that
\begin{align}\label{Laplacian estimate: k7}
k'_7:=\gamma+\sigma-k_7=\frac{-\gamma n-\sigma(n-1)+a_7(n-1)-\beta+1}{(a_7-1)(n-1)}.
\end{align}
Thus the conclusion $k'_7<1$ holds, under the condition \eqref{Laplacian estimate: RHS2 condition}.

\end{proof}

At last, we examine the conditions \eqref{Laplacian estimate: RHS1 good condition} and \eqref{Laplacian estimate: RHS2 condition}.
\begin{lem}\label{Laplacian estimate: criteria}
Assume that the parameters $\sigma$ and $\gamma$ satisfy $\sigma+\gamma<1$ and \eqref{Laplacian estimate: RHS2 parameters condition}.
Then the condition \eqref{Laplacian estimate: RHS1 good condition} and \eqref{Laplacian estimate: RHS2 condition} hold.
\end{lem}
\begin{proof}
We check that, when $\beta>n$ and $\sigma=\gamma=0$, 
the condition \eqref{Laplacian estimate: RHS1 good condition} is satisfied, i.e.
\begin{align*}
 \sigma=0<\frac{n}{n+1}<\frac{\beta}{n+1}.
\end{align*}and
Meanwhile, the condition \eqref{Laplacian estimate: RHS2 condition} is satisfies, too. That is
\begin{align*}
\gamma_0=1-\frac{\beta}{n}-\sigma(1-\frac{1}{n})=1-\frac{\beta}{n}<0=\gamma.
\end{align*}

Similarly, the second criterion implies that
\begin{align*}
\frac{\beta}{n+1}>\frac{1}{2}\geq \sigma.
\end{align*}and
\begin{align*}
\gamma_0=1-\frac{\beta}{n}-\sigma(1-\frac{1}{n})<1-\frac{\beta}{n}-\frac{n-\beta}{n-1}(1-\frac{1}{n})=0=\gamma.
\end{align*}

The third criterion also deduces that
\begin{align*}
\frac{\beta}{n+1}\geq\frac{1}{n+1}> \sigma.
\end{align*}and
\begin{align*}
\gamma_0=1-\frac{\beta}{n}-\sigma(1-\frac{1}{n})\leq (1-\sigma)\frac{n-1}{n}<\gamma.
\end{align*}

Therefore, any one of the three criteria in \eqref{Laplacian estimate: RHS2 condition} guarantees both the conditions \eqref{Laplacian estimate: RHS1 good condition} and \eqref{Laplacian estimate: RHS2 condition}.
\end{proof}

\begin{rem}\label{rough W2p estimate}
Using Young's inequality similar to the proof above, we could also obtain 
a $W^{2,p}$ estimate for $\tr_\om\om_\vphi$. However, this bound relies on the bound of $\p F$. Alternatively, we will obtain an accurate $W^{2,p}$ estimate in \thmref{w2pestimates degenerate Singular equation} in Section \ref{W2p estimate}, without any condition on $\p F$.
\end{rem}



\subsection{Step 6: iteration}\label{Step 6: iteration}
Combining the weighted inequality Proposition \ref{Weighted inequality} with the weighted inequality Proposition \ref{inverse weighted inequalities}, we obtain that




\begin{prop}[Iteration inequality]\label{Iteration inequality 1}
Assume that the parameters $\sigma$ and $\gamma$ satisfy $\sigma+\gamma<1$ and \eqref{Laplacian estimate: RHS2 parameters condition}. Then it holds
\begin{align}
\|\tilde u\|^{2p}_{L^{2p\chi}(\tilde\mu)}
\leq C_6 [p^3 \|\tilde u\|^{2p}_{L^{2pa}(\tilde\mu)}+1].
\end{align}

\end{prop}

We normalise the measure to be one. The norm $\|\tilde u\|_{L^{2p\chi}(\tilde\mu)}$ is increasing in $p$.
We assume $\|\tilde u\|_{L^{2p_0\chi}(\tilde\mu)}\geq 1$ for some $p_0\geq 1$, otherwise it is done.

We develop the iterating process from the iteration inequality
\begin{align}\label{Laplacian estimate: iteration}
&\|\tilde u\|_{L^{2p\frac{n}{n-1}}(\tilde\mu)}
\leq p^{\frac{3}{2}p^{-1}} C_7^{\frac{1}{2}p^{-1}} \|\tilde u\|_{L^{2pa}(\tilde\mu)}.
\end{align}
Setting
\begin{align*}
\chi_a:=\frac{\frac{n}{n-1}}{a}>1,\quad p=\chi_a^i,\quad i=0,1,2,\cdots
\end{align*}
and iterating \eqref{Laplacian estimate: iteration} with $p=\chi_a^m$,
\begin{align*}
&\|\tilde u\|_{L^{2\chi_a^m\frac{n}{n-1}}(\tilde\mu)}
\leq \chi_a^{\frac{3}{2}m\chi_a^{-m}} C_{7}^{\frac{1}{2}\chi_a^{-m}} \|\tilde u\|_{L^{2\chi_a^ma}(\tilde\mu)},
\end{align*}
which is
\begin{align*}
=\chi_a^{\frac{3}{2}m\chi_a^{-m}} C^{\frac{1}{2}\chi_a^{-m}} _{7} \|\tilde u\|_{L^{2\chi_a^{m-1}\frac{n}{n-1}}(\tilde\mu)}.
\end{align*}
We next apply \eqref{Laplacian estimate: iteration} again with $p=\chi^{m-1}$,
\begin{align*}
&\leq \chi_a^{\frac{3}{2}m\chi_a^{-m}+\frac{3}{2}(m-1)\chi_a^{-(m-1)}}C_{7}^{\frac{1}{2}[\chi_a^{-m}+\chi_a^{-(m-1)}]}  \|\tilde u\|_{L^{2\chi_a^{m-1}a}(\tilde\mu)}.
\end{align*}

We choose $i_0$ such that $\tilde p_0=\chi_a^{i_0}\geq p_0$.
Repeating the argument above, we arrive at
\begin{align*}
&\leq \chi_a^{\frac{3}{2}\sum_{i=i_0}^m i\chi_a^{-i}}C_{7}^{\frac{1}{2}\sum_{i=i_0}^m{\chi_a^{-i}}}\|\tilde u\|_{L^{2a\tilde p_0}(\tilde\mu)}.
\end{align*}
Since these two series $\sum_{i=i_0}^\infty i\chi_a^{-i}$ and $\sum_{i=i_0}^\infty{\chi_a^{-i}}$ are convergent, we take $m\rightarrow \infty$ and conclude that
\begin{align*}
&\|\tilde u\|_{L^{\infty}}\leq C \|\tilde u\|_{L^{2a\tilde p_0}(\tilde\mu)}\leq  C[ \|u\|_{L^{2a\tilde p_0}(\tilde\mu)}+\|K\|_{L^{2a\tilde p_0}(\tilde\mu)}].
\end{align*}

In order to obtain the uniform bound of $\tr_\om\om_{\vphi}\cdot S_\eps^{\gamma}$, we apply the $L^\infty$ bound of $\vphi$ and $L^{2a\tilde p_0}$ bound of $\tr_\om\om_\vphi$ from Definition \ref{a priori estimates approximation}, it is left to compare the exponent of the weight in the integral 
\begin{align*}
\int_X e^{-C_12a\tilde p_0\vphi}(\tr_\om\om_{\vphi})^{2a\tilde p_0}S_\eps^{2a\tilde p_0\gamma+(\gamma+\beta-1+\sigma)\frac{n}{n-1}}\om^n
\end{align*} with the exponent  $\sigma^2_D$ in \eqref{almost admissible w2p sigmaD}, 
\begin{align*}
[(2a\tilde p_0+1)\gamma+\beta-1+\sigma]\frac{n}{n-1} \geq \sigma^2_D=(\beta-1)\frac{n-2}{n-1+(2a\tilde p_0)^{-1}}.
\end{align*}

In conclusion, we obtain the Laplacian estimate
\begin{align*}
v_\eps:=\tr_\om\om_{\vphi_\eps}\leq C S_\eps^{-\gamma}.
\end{align*}
Moreover, $\gamma=0$, when $\beta>\frac{n+1}{2}$.

\subsection{Laplacian estimate for degenerate KE equation}\label{Log Kahler Einstein metric}
We apply the weighted integration method developed for the degenerate scalar curvature equation to the degenerate KE problem, that provides an alternative proof of Yau's Laplacian estimate for the approximate degenerate KE equation \eqref{critical pt Ding approximation}, 
\begin{align}\label{critical pt Ding approximation estimates}
\om^n_{\vphi_\eps}
=  e^{h_\om+\mathfrak h_\eps-\lambda\vphi_\eps +c_\eps}  \om^n.
\end{align}
Comparing with the approximate degenerate scalar curvature equation \eqref{Degenerate cscK 1 approximation}, we have $\theta=\lambda\om$, $R=\underline S_\b=\lambda n$ and 
\begin{align*}
F_\eps=-\lambda\vphi_\eps,\quad \tilde f_\eps=-h_\om-\mathfrak h_\eps -c_\eps, \quad \tilde F_\eps=F_\eps-\tilde f_\eps.
\end{align*}
We also have 
\begin{align}\label{KE tri F}
\tri F_\eps=-\lambda\tri \vphi_\eps=-\lambda(\tr_\om\om_{\vphi_\eps}-n).
\end{align}

\begin{thm}\label{KE Laplacian}Suppose $\vphi_\eps$ is a solution to the approximate KE equation \eqref{critical pt Ding approximation estimates}.
Then there exists a uniform constant $C$ such that
\begin{align}\label{Laplacian estimate}
\tr_\om\om_{\vphi_\eps}\leq C \text{ on } X
\end{align}
where the uniform constant $C$ depends on
\begin{align*}
\|\vphi_\eps\|_\infty,\quad \|h_\om\|_\infty,\quad\inf_X\tri h_\om, \quad
 \inf_{i\neq j}R_{i\bar i j\bar j}(\om), \quad C_S(\om),\quad \Theta_D, 
\end{align*}
and $\lambda,\quad\beta, \quad c_\eps, \quad n$.
\end{thm}
\begin{proof}
Substituting \eqref{KE tri F} into the nonlinear term \eqref{Laplacian estimate: integration cor}, we have 
\begin{align*}
\mathcal N=\lambda\int_Xe^{-C_1\vphi}\tilde u^{2p}(v-n) S_\eps^{\gamma}\om_\vphi^n
 =\lambda\int_X\tilde u^{2p}u\om_\vphi^n-n\lambda\int_Xe^{-C_1\vphi}\tilde u^{2p} S_\eps^{\gamma}\om_\vphi^n.
\end{align*}
The integral inequality is reduced immediately to
\begin{align}\label{Laplacian estimate: integral inequality pre prop KE}
&\frac{2C_{1.6}}{p}LHS_1
+ C_{1.7}   LHS_2 
\leq \mathcal (\lambda+C_{1.5})\int_X\tilde  u^{2p+1}\om_\vphi^n.
\end{align}
Modified the constants, it becomes
\begin{align}\label{Laplacian estimate: integral inequality KE}
&p^{-1}LHS_1+   LHS_2 
\leq C_2 RHS_1:=C_2 \int_X\tilde  u^{2p+1}\om_\vphi^n.
\end{align}

Following the argument in Proposition \ref{Rough iteration inequality prop}, we have the rough iteration inequality
\begin{align}\label{Laplacian estimate: RHSLHS KE}
\|\tilde u\|^{2p}_{L^{2p\chi}(\tilde\mu)}+p  LHS_2
\leq C_{3}(p  RHS_1 +  RHS_2+1).
\end{align}

Examining the estimates from \lemref{Laplacian estimate: RHS1 good} for $RHS_1$, \lemref{Laplacian estimate: RHS2 bound} for $RHS_2$, we deduce the inverse weighted inequalities 
\begin{align*}
\|\tilde u\|^{2p}_{L^{2p\chi}(\tilde\mu)}
\leq C_5 [p  \int_X \tilde u^{2p} S_\eps^{\gamma+\sigma-k'}  \om_\vphi^n+1]\text{ for some }k'<1,
\end{align*}
under the conditions
\begin{equation*}
\left\{
\begin{aligned}
 &\sigma-(\beta-1)<1;\\
&\sigma+\gamma<1;\\
&\gamma>1-\frac{\beta}{n}-\sigma(1-\frac{1}{n}),
   \end{aligned}
\right.
\end{equation*}
which are actually automatically satisfied from the criteria
\begin{equation}\label{Laplacian estimate: RHS2 parameters condition ke}
\gamma=0,\quad 1>\sigma>\frac{n-\beta}{n-1}.
\end{equation}


Therefore, we could apply the weighted inequality, Proposition \ref{Weighted inequality}, to derive the iteration inequality. Then we further employ the iteration techniques as Section \ref{Step 6: iteration} to conclude the $L^\infty$ norm of $$\tilde u=e^{-C_1\vphi} \tr_\om\om_{\vphi_\eps} +K$$ in terms of the $W^{2,p}$-estimate, which could be obtained from \thmref{w2pestimates degenerate Singular equation} in a similar way, or from the argument in Remark \ref{rough W2p estimate}.
\end{proof}


\section{Singular cscK metrics}\label{Singular cscK metrics}
In \cite{MR4020314}, we introduced the singular cscK metrics and proved serval existence results, where we focus on the case $0<\beta<1$. In this section, we consider $\beta>1$ and obtain the $L^\infty$ estimate, the gradient estimate and the $W^{2,p}$ estimate for the singular cscK metrics.

We start from the basic setup of the singular scalar curvature equation we introduced in \cite{MR4020314}.
\begin{defn}
A real $(1,1)$-cohomology class $\Om$ is called
\textit{big}, if it contains a \textit{K\"ahler current}, which is a closed positive $(1,1)$-current $T$ satisfying $T\geq t\om_K$ for some $t>0$. A big class $\Om$ is defined to be \textit{semi-positive}, if it admits a smooth closed $(1,1)$-form representative.
\end{defn}


Recall that $\om $ is a K\"ahler metric on $X$. We let $\om_{sr}$ be a smooth representative in the big and semi-positive class $\Om$.


\subsection{Perturbed K\"ahler metrics}\label{Perturbed Kahler metrics}
In order to modify $\om_{sr}$ to be a K\"ahler metric,
one way is to perturb $\om_{sr}$ by adding another K\"ahler metric,
\begin{align*}
\om_t :=\om_{sr}+t\cdot \om\in \Om_t:=\Om+t[\om],\quad\text{for all }t>0.
\end{align*}
The other way is to apply Kodaira's Lemma, namely there exists a sufficiently small number $a_0$ and an effective divisor $E$ such that $\Om-a_0 [E]$ is ample and
\begin{align}
\om_{K}:=\om_{sr}+ i\p\bar\p \phi_E>0,\quad \phi_E:=a_0\log h_E
\end{align}
is a K\"ahler metric. In which, $h_E$ is a smooth Hermitian metric on the associated line bundle of $E$ and
$s_E$ is the defining section of $E$. 

We also write
\begin{align*}
\tilde\om_t:=\om_K+ t\om.
\end{align*}
\begin{lem}\label{metrics equivalence}
The three K\"ahler metrics $\om,\om_K, \tilde\om_t$ are all equivalent,
\begin{align*}
\om_K\leq\tilde\om_t\leq\om_K+\om,\quad C_K^{-1}\om\leq\om_K\leq C_K\om.
\end{align*}
\end{lem}

With the help of these three K\"ahler metrics, we are able to measure the $(1,1)$-form $\theta$.
As Lemma 5.6 in \cite{arXiv:1803.09506}, we define the bound of $\theta$, which is independent of $t$, by using $\tilde\om_t$ as the background metric.
\begin{defn}\label{L infty estimates theta defn}
We write
\begin{align}\label{L infty estimates theta}
C_l\cdot \tilde\om_t\leq \theta\leq  C_u\cdot\tilde\om_t.
\end{align}
where
$
C_l:=\min\{0,\inf_{(X,\tilde\om_t)}\theta\},\quad C_u:=\max\{0,\sup_{(X,\tilde\om_t)}\theta\}.
$
\end{defn}

\begin{lem}
The given $(1,1)$-form $\theta$ has the lower bound
\begin{align}\label{Cl}
\theta\geq C_l \om_{\theta_{t,\eps}}-i\p\bar\p\phi_l,\quad \phi_l:=C_l(\vphi_{\theta_{t,\eps}}-\phi_E)
\end{align}
and the upper bound
\begin{align}\label{Cu}
\theta\leq C_u \om_{\theta_{t,\eps}}-i\p\bar\p\phi_u,\quad \phi_u:=C_u(\vphi_{\theta_{t,\eps}}-\phi_E).
\end{align}
\end{lem}

\begin{rem}
In particular, if the big and semi-positive class $\Om$ is proposional to $C_1(X,D)$, i.e. $\Om=\lambda C_1(X,D),  \lambda=-1,1,$ then we have
\begin{align*}
C_l=0, \text{ when  } \lambda=1, \text{ and } C_u=0, \text{ when }\lambda=-1.
\end{align*}
Actually, it even holds $\theta=\lambda \om_{sr}$. 
\end{rem}
\subsection{Reference metrics}\label{Reference metrics}
\begin{defn}[Reference metric]\label{Reference metric singular}
The reference metric $\om_\theta$ is defined to be the solution of the following singular Monge-Amp\`ere equation
\begin{align}\label{lift reference metric log Fano app limit}
\om_\theta^n:=(\om_{sr}+i\p\bar\p\vphi_{\theta})^n=|s|_h^{2\b-2}e^{h_{\theta}}\om^n
\end{align}
where $h_\theta$ is defined in \eqref{h0}.
\end{defn}
\subsubsection{$t$-purterbation}
Replacing $\om_{sr}$ by the K\"ahler metric $\om_t$, we introduce the following perturbed metric. 
\begin{defn}The $t$-purterbated reference metric $\om_{\theta_t}\in \Om_t$ is defined to be a solution to the following degenerate Monge-Amp\`ere equation
\begin{align}\label{lift reference metric log Fano app both 2}
\om_{\theta_t}^n:=(\om_{t}+i\p\bar\p\vphi_{\theta_t})^n=|s|_h^{2\b-2} e^{h_\theta+c_t}\om^n,
\end{align}
under the condition 
\begin{align*}
 Vol(\Om_t)=\int_X|s|_h^{2\b-2} e^{h_\theta+c_{t}}\om^n.
\end{align*}
\end{defn}

Using the Poincar\'e-Lelong formula \eqref{PL} and the Ricci curvature of $\om$ from \eqref{h0}, we compute the Ricci curvature of $\om_{\theta_{t}}$.
\begin{lem}\label{Ric om theta t}
The Ricci curvature equation for \eqref{lift reference metric log Fano app both 2} reads
\begin{align*}
Ric(\om_{\theta_{t}})=\theta+2\pi(1-\beta)[D].
\end{align*}
\end{lem}

\subsubsection{$(t,\eps)$-approximation}
Similar to Section \ref{Approximation}, we further approximate \eqref{lift reference metric log Fano app both 2} by a family of smooth equation for $\om_{\theta_{t,\eps}}\in \Om_t$.
\begin{defn}We define the $(t,\eps)$-approximate reference metric $$\om_{\theta_{t,\eps}}:=\om_{t}+i\p\bar\p\vphi_{\theta_{t,\eps}}$$ satisfy the equation
\begin{align}\label{lift reference metric log Fano app both 3}
\om_{\theta_{t,\eps}}^n=S_\eps^{\b-1} e^{h_{\theta}+c_{t,\eps}}\om^n.
\end{align}
The constant $c_{t,\eps}$ is determined by the normalised volume condition
\begin{align*}
Vol(\Om_t)=\int_XS_\eps^{\beta-1}e^{h_{\theta}+c_{t,\eps}}\om^n.
\end{align*}
\end{defn}
The volume is uniformly bounded independent of $t, \eps$. In the following, when we say a constant or an estimate is uniform, it means it is independent of $t$ or $\eps$.

\begin{lem}\label{Ric om theta t eps}
The Ricci curvature equation for \eqref{lift reference metric log Fano app both 3} reads
\begin{align*}
Ric(\om_{\theta_{t,\eps}})=\theta+(1-\beta)\{2\pi[D]-i\p\bar\p(\log|s|_h^2-\log S_\eps)\}.
\end{align*}
\end{lem}
\begin{proof}
Taking $-i\p\bar\p\log$ on \eqref{lift reference metric log Fano app both 3}, we get
\begin{align*}
Ric(\om_{\theta_{t,\eps}})=-(\beta-1)i\p\bar\p\log S_\eps-i\p\bar\p h_\theta+Ric(\om).
\end{align*}
Then we have by \eqref{h0} that
\begin{align*}
Ric(\om_{\theta_{t,\eps}})=-(\beta-1)i\p\bar\p\log S_\eps+\theta+(1-\beta)\Theta_D.
\end{align*}
The conclusion thus follows from \eqref{PL}.
\end{proof}
\subsubsection{Estimation of the reference metric}

We write 
\begin{align*}
 \tilde\vphi_{\theta_{t,\eps}}:=\vphi_{\theta_{t,\eps}}-\phi_E.
\end{align*}
Accordingly, we see that  
\begin{align*}
\tilde\om_t=\om_t+i\p\bar\p\phi_E,\quad \om_{\theta_{t,\eps}}=\tilde \om_t+i\p\bar\p\tilde\vphi_{\theta_{t,\eps}}.
\end{align*}
By substitution into \eqref{lift reference metric log Fano app both 3}, we rewrite the $(t,\eps)$-approximation \eqref{lift reference metric log Fano app both 3} of the reference metric as
\begin{lem}
\begin{align}\label{lift reference metric log Fano app both 2 Phi}
\om_{\theta_{t,\eps}}^n=(\tilde \om_t+i\p\bar\p \tilde\vphi_{\theta_{t,\eps}})^n=e^{-f_{t,\eps}} \tilde \om_t^n,
\end{align}
where
\begin{align}\label{Singular cscK metrics: f}
-f_{t,\eps}:=(\beta-1)\log S_\eps+h_\theta+c_{t,\eps}+\log\frac{\om^n}{\tilde\om_t^n}.
\end{align}
\end{lem}

\begin{lem}\label{nef tilde f}
There exists a uniform constant $C$ such that
\begin{align*}
 &f_{t,\eps}\geq C, \quad e^{-  f_{t,\eps}}\leq C,\quad 
 |\p f|^2\leq C[1+(\beta-1)^2S^{-1}_\eps]\\
 & -C[(\beta-1)S_\eps^{-1}+1]\om\leq i\p\bar\p  f_{t,\eps}\leq C \om.
\end{align*}
\end{lem}
\begin{proof}
Since the expression of $e^{-  f_{t,\eps}}$ is
\begin{align*}
e^{-  f_{t,\eps}}=S_\eps^{\beta-1} e^{h_\theta+c_{t,\eps}}\frac{\om^n}{\tilde\om_t^n},
\end{align*}
we obtain its upper bound from $\beta\geq 1$.
Moreover, its derivatives are estimated by 
\begin{align*}
|\p f|^2&=|\p [(\beta-1)\log S_\eps+h_\theta+c_{t,\eps}+\log\frac{\om^n}{\tilde\om_t^n}]|^2\\
&\leq C[(\beta-1)^2|\p |s|_h^2|^2S_\eps^{-2}+1]
\leq C[(\beta-1)^2S_\eps^{-1}+1].
\end{align*}
and
\begin{align*}
i\p\bar\p f_{t,\eps}=-i\p\bar\p [(\beta-1)\log S_\eps+h_\theta+c_{t,\eps}+\log\frac{\om^n}{\tilde\om_t^n}].
\end{align*}
Thus this lemma is a consequence of \lemref{h eps}.
\end{proof}

\begin{lem}\label{big nef approximate reference metric bound}Suppose that $\varphi_{\theta_{t,\eps}}$ is a solution to the approximate equation \eqref{lift reference metric log Fano app both 3}.
Then there exists a uniform constant $C$ independent of $t,\eps$ such that
\begin{align*}
\|\varphi_{\theta_{t,\eps}}\|_\infty\leq C, \quad 
C^{-1} S_\eps^{\beta-1} |s_E|^{a(n-1)}_{h_E}\cdot\om
\leq \om_{\theta_{t,\eps}}
\leq C |s_E|^{-a}_{h_E}\cdot\om.
\end{align*}
Moreover, the sequence $\varphi_{\theta_{t,\eps}}$ $L^1$-converges to $\vphi_\theta$, which is the unique solution to the equation \eqref{lift reference metric log Fano app limit}.
The solution $\vphi_\theta\in C^0(X)\cap C^{\infty}(X\setminus (E\cup D))$.
\end{lem}
\begin{proof}
We outline the proof for readers' convenience.
The $L^\infty$ estimate is obtained, by applying Theorem 2.1 and Proposition 3.1 in \cite{MR2505296}, since $$(|s|_h^2+\eps)^{\b-1} e^{h_{\theta}+c_{t,\eps}}$$ is uniformly in $L^p(\om^n_{sr})$ for $p>1$. 

For the Laplacian estimate, we denote
\begin{align*}
v=\tr_{\tilde \om_t}\om_{\theta_{t,\eps}},\quad w=e^{-C_1 \tilde\vphi_{\theta_{t,\eps}}} v, 
\end{align*}
and obtain from \eqref{Laplacian estimate: laplacian w} that
\begin{align}\label{Yau computation C1 semi}
e^{C_1 \tilde\vphi_{\theta_{t,\eps}}}\tri_\vphi w
		&\geq v^{1+\frac{1}{n-1}} e^{-\frac{\tilde  F_{\theta_{t,\eps}}}{n-1}}+\tri \tilde  F_{\theta_{t,\eps}}-C_1 n v.
\end{align}
The constant $C_1$ is taken to be $C_{1.1}+1$ and $C_{1.1}$ is the lower bound the bisectional curvature of $\tilde \om_t$. 

We write $\tilde F_{\theta_{t,\eps}}=-f_{t,\eps}$ and use \lemref{nef tilde f} to show that there exists a uniform constant $C$, independent of $t$ and $\eps$ such that
\begin{align*}
\tilde F_{\theta_{t,\eps}}\leq C,\quad  \tri \tilde F_{\theta_{t,\eps}}\geq -C.
\end{align*} 

Then the argument of maximum principle applies. At the maximum point $p$ of $w$, $v(p)$ is bounded above from \eqref{Yau computation C1 semi}. While, at any pint $x\in X$, $w(x)\leq w(p)$, which means
\begin{align*}
v(x)\leq e^{C_1( \tilde\vphi_{\theta_{t,\eps}}(x)- \tilde\vphi_{\theta_{t,\eps}}(p))} v(p)
\leq e^{C_1(\vphi_{\theta_{t,\eps}}(x)- \tilde\vphi_{\theta_{t,\eps}}(p))} |s_E|^{-C_1 a_0}_{h_E}(x).
\end{align*}
Moreover, $\tri_{\tilde \om_t}\phi_E$ is uniformly bounded independent of $t,\eps$.
inserting them into
\begin{align*}
\tri_{\tilde \om_t}\vphi_{\theta_{t,\eps}}=\tr_{\tilde \om_t}\om_{\theta_{t,\eps}}-n+\tri_{\tilde \om_t}\phi_E =v-n+\tri_{\tilde \om_t}\phi_E,
\end{align*}
we obtain the Laplacian estimate 
\begin{align*}
\tri_{\tilde \om_t}\vphi_{\theta_{t,\eps}} \leq C |s_E|^{-a}_{h_E}
\end{align*}
by taking $a=C_1 a_0.$ The metric bound is thus obtained from the following \lemref{Laplacian to metric bound}, namely
\begin{align*}
C^{-1} |s_E|^{a(n-1)}_{h_E}\tilde \om_t^n\leq\om_{\theta_{t,\eps}} \leq C|s_E|^{-a}_{h_E} \tilde \om_t^n
\end{align*}
together with the equivalence of the reference metrics \lemref{metrics equivalence}.
\end{proof}
\begin{lem}\label{Laplacian to metric bound}
Assume we have the Laplacian estimate 
\begin{align*}
\tri_{\tilde \om_t}\vphi_{\theta_{t,\eps}} \leq A.
\end{align*}
Then the metric bound holds
\begin{align*}
 C S_\eps^{\beta-1} (n+A)^{-(n-1)}\tilde \om_t \leq \om_{\theta_{t,\eps}}\leq (n+A) \tilde \om_t.
\end{align*}
\end{lem}
\begin{proof}
The upper bound of $ \om_{\theta_{t,\eps}}$ is obtained from
\begin{align*}
\tr_{\tilde \om_t}\om_{\theta_{t,\eps}}=n+\tri_{\tilde \om_t}\vphi_{\theta_{t,\eps}}\leq n+A.
\end{align*}
The lower bound follows directly from the fundamental inequality and the equation \eqref{lift reference metric log Fano app both 3} of the volume ratio,
\begin{align*}
\tr_{\om_{\theta_{t,\eps}}}{\tilde \om_t}\leq(\frac{\om_{\theta_{t,\eps}}^n}{\tilde \om_t^n})^{-1}(\tr_{\tilde \om_t}\om_{\theta_{t,\eps}})^{n-1} \leq CS_\eps^{1-\beta} (n+A)^{n-1}.
\end{align*}
\end{proof}

\subsection{Singular scalar curvature equation}
\begin{defn}\label{Singular cscK eps defn}
Let $\Om$ be a big and semi-positive class and $\om_{sr}$ is a smooth representative in $\Om$. The singular scalar curvature equation in $\Om$ is defined to be
\begin{equation}\label{Singular cscK eps}
\om_\vphi^n=(\om_{sr}+i\p\bar\p\vphi)^n=e^F \om_{\theta}^n,\quad
\tri_{\vphi} F=\tr_{\vphi}\theta-R.
\end{equation}
The reference metric $\om_\theta$ is introduced in \eqref{lift reference metric log Fano app limit} and $R$ is a real-valued function. 
In particular, when considering the singular cscK equation, we have
\begin{align*}
R=\ul S_\beta=\frac{nC_1(X,D)\Om^{n-1}}{\Om^{n}}.
\end{align*}
\end{defn}
Inserting the expression of $\om_\theta$, i.e. $\om_\theta^n=e^{-f}\om^n$, into \eqref{Singular cscK eps}, we get
\begin{lem}
\begin{equation*}
(\om_{sr}+i\p\bar\p\vphi)^n=e^{\tilde F} \om^n,\quad
\tri_{\vphi} \tilde F=\tr_{\vphi}(\theta-i\p\bar\p f)-R,
\end{equation*}
where 
$\tilde F=F-f,\quad f=-(\beta-1)\log|s|_h^2-h_\theta.$
\end{lem}

\subsection{Approximate singular scalar curvature equation}

We define an analogue of the $t$-perturbation and the $(t,\eps)$-approximation of the singular scalar curvature equation.
In general, we consider the perturbed equations.
\begin{defn}\label{Singular cscK t eps defn}The perturbed singular scalar curvature equation is defined as
\begin{equation*}
\om_{\vphi_t}^n=(\om_{t}+i\p\bar\p\vphi_t)^n=e^{F_t} \om_{\theta_t}^n,\quad
\tri_{\vphi_t} F_t=\tr_{\vphi_t}\theta-R.
\end{equation*}
While, the approximate singular scalar curvature equation is set to be
\begin{equation}\label{Singular cscK t eps}
\om_{\vphi_{t,\eps}}^n=(\om_t+i\p\bar\p\vphi_{t,\eps})^n=e^{F_{t,\eps}} \om_{\theta_{t,\eps}}^n,\quad
\tri_{\vphi_{t,\eps}} F_{t,\eps}=\tr_{\vphi_{t,\eps}}\theta-R.
\end{equation}
\end{defn}
\begin{rem}
The function $R$ is also needed to be perturbed. Since their perturbation are bounded and does not effect the estimates we will derive, we still use the same notation in the following sections for convince.
\end{rem}

\begin{defn}
We approximate the singular cscK equation by
\begin{equation*}
\om_{\vphi_t}^n=(\om_{t}+i\p\bar\p\vphi_t)^n=e^{F_t} \om_{\theta_t}^n,\quad
\tri_{\vphi_t} F_t=\tr_{\vphi_t}\theta-R_{t}
\end{equation*}
with the constant to be
\begin{equation}\label{approximate average scalar}
R_{t}=\frac{nC_1(X,D)(\Om+t[\om])^{n-1}}{(\Om+t[\om])^{n}}.
\end{equation}
\end{defn}

By the formulas of $Ric( \om_{\theta_t})$ from \lemref{Ric om theta t} and $Ric(\om_{\theta_{t,\eps}})$ by \lemref{Ric om theta t eps}, the scalar curvature of both approximations are given as below.
\begin{lem}On the regular part $M=X\setminus D$, we have
\begin{align*}
S(\om_{\vphi_{t}})=R,\quad
S(\om_{\vphi_{t,\eps}})=R+(1-\beta)\tri_{\vphi_{t,\eps}}(\log S_\eps-\log|s|_h^2).
\end{align*}
\end{lem}
\begin{proof}
They are direct obtained by inserting the Ricci curvature equations $Ric(\om_{\theta_{t}})=\theta+2\pi(1-\beta)[D]
$ and
\begin{align*}
Ric(\om_{\theta_{t,\eps}})=\theta+(1-\beta)\{2\pi[D]-i\p\bar\p(\log|s|_h^2-\log S_\eps)\}.
\end{align*}
into the scalar curvature equations
\begin{align*}
S(\om_{\vphi_{t}})=-\tri_{\vphi_t} F_t+Ric(\om_{\theta_t}),\quad
S(\om_{\vphi_{t,\eps}})=-\tri_{\vphi_{t,\eps}}F_{t,\eps}+Ric(\om_{\theta_{t,\eps}}).
\end{align*}
\end{proof}



We will derive a priori estimates for the smooth solutions to the approximate equation \eqref{Singular cscK t eps}. We set
\begin{align*}
&  \tilde\vphi_{{t,\eps}}:=\vphi_{{t,\eps}}-\phi_E, \quad\tilde F_{t,\eps}:=F_{t,\eps}-f_{t,\eps}.\end{align*}
Then it follows $\om_{\vphi_{t,\eps}}=\om_t+i\p\bar\p \vphi_{t,\eps}=\tilde\om_t+i\p\bar\p\tilde\vphi_{t,\eps}.
$ Hence, \eqref{Singular cscK t eps} is rewritten for the couple $(\Phi_{t,\eps}, \tilde F_{t,\eps})$ as following equations regarding to the smooth K\"ahler metric $\tilde\om_t$.
\begin{lem}
\begin{equation}\label{Singular cscK t eps tilde}
(\tilde\om_t+i\p\bar\p \tilde\vphi_{t,\eps})^n=e^{\tilde F_{t,\eps}} \tilde\om_t^n,\quad
\tri_{\vphi_{t,\eps}} \tilde F{_{t,\eps}}=\tr_{\vphi_{t,\eps}}(\theta-i\p\bar\p f_{t,\eps})-R.
\end{equation}
\end{lem}

According to \lemref{nef tilde f}, we see that $f_{t,\eps}$ have uniform lower bound and $i\p\bar\p f_{t,\eps}$ has uniform upper bound, when $\beta>1$. It is different from the the cone metric $0<\beta<1$ or the smooth metric $\beta=1$, whose $f_{t,\eps}$ have uniform upper bound and $i\p\bar\p f_{t,\eps}$ have uniform lower bound.

The a priori estimates in \cite{MR4301557}  were extended to the cone metrics in the big and semi-positive class in \cite{arXiv:1803.09506}. In the following sections, including Section \ref{Linfty estimate}, Section \ref{Gradient estimate of vphi} and Section \ref{W2p estimate}, we obtain estimation for the approximate singular scalar curvature equation \eqref{Singular cscK t eps} with $\beta>1$. 

\begin{defn}[Almost admissible solution for singular equations]\label{a priori estimates approximation singular}
We say $\vphi_\eps$ is an \text{almost admissible solution} to the approximate singular scalar curvature equation \eqref{Singular cscK t eps}, if there are uniform constants independent of $t,\eps$ such that the following estimates hold
\begin{itemize}
\item $L^\infty$-estimates in \thmref{L infty estimates Singular equation}: 
\begin{align*}
&\|\vphi_{t,\eps}\|_\infty\leq C,\quad \|e^{F_{t,\eps}}\|_{p;\tilde\om_t^n},\quad \|e^{\tilde F_{t,\eps}}\|_{p;\tilde\om_t^n} \leq C(p),\quad p\geq 1 ;\\
&\sup_X(F_{t,\eps}-\sigma_s\phi_E),\quad-\inf_{X}[F_{t,\eps}-\sigma_i\phi_E]\leq C;\\
& \sigma_s:=C_l-\tau,\quad \sigma_i:=C_u+\tau,\quad \forall\tau>0. 
\end{align*}
\item gradient estimate of $\vphi$ in \thmref{gradient estimate}: 
\begin{align*}
|s_E|^{2a_0 \sigma^1_E } S_\eps^{\sigma^1_D}|\p\tilde\vphi_{t,\eps}|^2_{\tilde\om_t}\leq C, \quad \sigma^1_D\geq 1
\end{align*}
where the singular exponent $\sigma^1_E$ satisfies \eqref{Gradient estimate: sigmaE} and \eqref{Gradient estimate: sigmaE 2};

\item $W^{2,p}$-estimate in \thmref{w2pestimates degenerate Singular equation}: 
\begin{align*}
\int_X   (\tr_{\tilde\om_t}\om_\vphi)^{p}   |s_E|^{\sigma^2_E}_{h_E} S_\eps^{\sigma^2_D}  \tilde\om_t^n\leq  C(p),\quad\forall p\geq 1
\end{align*}
where $\sigma^2_D>(\beta-1)\frac{n-2-2np^{-1}}{n-1+p^{-1}}$ and $\sigma^2_E$ is given in \eqref{w2pestimates sigma E}.
\end{itemize}

\end{defn}

\begin{thm}\label{almost admissible singular}
The family of solutions to the approximate singular scalar curvature equation \eqref{Singular cscK t eps} with bounded entropy is almost admissible. 
\end{thm}
Suppose $\Om$ is K\"ahler, the singular equation is reduced to the degenerate equation \eqref{Degenerate cscK}.
\begin{thm}\label{almost admissible degenerate}
Assume $\{\vphi_{\eps}\}$ is a family of solutions to the approximate degenerate scalar curvature equation \eqref{Degenerate cscK approximation} with bounded entropy. Then $\vphi_{\eps}$ is almost admissible, Definition \ref{a priori estimates approximation}.
\end{thm}

\subsection{Singular metrics with prescribed scalar curvature}

In this section, we will prove an existence theorem for singular metrics with prescribed scalar curvature. The idea is to construct solution to the singular equation by taking limit of the solutions to the approximate equation. The proof is an adaption of Theorem 4.40 in our previous article \cite{arXiv:1803.09506}.

We recall the definition of singular canonical metric satisfying particular scalar curvature, see \cite[Definition 4.39]{arXiv:1803.09506}, which is motivated from the study of canonical metrics on connected normal complex projective variety.
\begin{defn}[Singular metrics with prescribed scalar curvature]\label{Singular metric with prescribed scalar curvature}
In a big cohomology class $[\om_{sr}]$, we say $\om:=\om_{sr}+i\p\bar\p\vphi$ is a singular metric with prescribed scalar curvature, if $\vphi\in PSH(\om_{sr})$ is a $L^1$-limit of a sequence of K\"ahler potentials solving the approximate scalar curvature equation \eqref{Singular cscK t eps}.

In addition, we say the singular metric $\om_\vphi$ is bounded, if $\vphi$ is $\mathcal E^1(\om_{sr})\cap L^\infty$.

When $\Om$ is K\"ahler, we name it the degenerate metric instead.
\end{defn}

\subsection{Estimation of approximate solutions}
We need to consider the existence of the solution to the approximate singular scalar curvature equation \eqref{Singular cscK t eps}.
Restricted to the case for log KE metrics, we have seen the smooth approximation in Proposition \ref{smooth approximation KE}. For the cscK cone metrics, $0<\beta\leq 1$, the smooth approximation is shown in  \cite[Proposition 4.37]{arXiv:1803.09506}. We collect the results as below. We say a real $(1,1)$-cohomology class $\Om$ is
\textit{nef}, if the cohomology class $\Om_t:=\Om+t[\om]$ is K\"ahler for all $t>0$.
\begin{prop}Let $\Om$ be a big and nef class on a K\"ahler manifold $(X,\om)$, whose automorphism group $Aut(M)$ is trivial. Suppose that $\Om$ satisfies the cohomology condition 
\begin{align*}
\left\{
\begin{array}{lcl}
	&&0\leq \eta<\frac{n+1}{n}\alpha_\beta,\quad
	 C_1(X,D)<\eta\Omega,\\
	&& (-n\frac{C_1(X,D)\cdot\Omega^{n-1}}{\Omega^{n}}+\eta)\Om+(n-1)C_1(X,D)>0.
\end{array}
\right.
\end{align*}

Then $\Om$ has the cscK approximation property, which precisely asserts that, for small positive $t$, we have
	\begin{enumerate}
		\item
	the log $K$-energy is $J$-proper in $\Om_t$;
	\item there exists a cscK cone metric $\om_{t}$ in $\Om_t$;
	\item the cscK cone metric $\om_{t}$ a smooth approximation $\om_{t,\eps}$ in $\Om_{t}$.
	\end{enumerate}
\end{prop}

\begin{ques}
We may ask the question whether the approximate degenerate cscK equation \eqref{Degenerate cscK approximation} has a smooth solution, if the log $K$-energy \eqref{log K energy} is proper. 
\end{ques}
It is different from the case $0<\beta\leq 1$, when the approximate log $K$-energy $\nu^\eps_\beta$ dominates the log $K$-energy $\nu_\beta$. So the properness of $\nu^\eps_\beta$ follows from the properness of $\nu_\beta$ and then implies existence of the approximate solutions by \cite{MR4301558}.


\subsection{Convergence and regularity}\label{Convergence and regularity}

Now we further consider the convergence problem of the family $\{\vphi_{t,\eps}\}$ is consisting of smooth approximation solutions for \eqref{Singular cscK t eps} and regularity of the convergent limit.
We normalise the family to satisfy $\sup_X\vphi_{t,\eps}=0$. According to the Hartogs's lemma, a sequence of $\om_{sr}$-psh functions has $L^1$-weak convergent subsequence.
\begin{lem}
The family $\{\vphi_{t,\eps}\}$ converges to $\vphi$ in $L^1$, as $t,\eps\rightarrow 0$. Moreover, $\vphi$ is an $\om_{sr}$-psh function with $\sup_X\vphi=0$.
\end{lem}

\begin{lem}
Assumed the uniform bound of $\|\vphi_{t,\eps}\|_\infty$ and $\|e^{F_{t,\eps}}\|_{p;\tilde\om_t^n}$,  the limit $\vphi$ belongs to $\mathcal E^1(X,\om_{sr})\cap L^\infty(X)$.
\end{lem}
\begin{proof}
The proof is included in the proof of \cite[Proposition 4.41]{arXiv:1803.09506}.
\end{proof}

In conclusion, we apply the gradient and $W^{2,p}$ estimates in \thmref{almost admissible singular} to obtain
\begin{thm}\label{existence singular}
Assume that the family $\{\vphi_{t,\eps}\}$ is the almost admissible solution to the approximate singular scalar curvature equation \eqref{Singular cscK t eps}. Then there exists a bounded singular metric $\om_\vphi$ with prescribed scalar curvature, moreover, $\om_\vphi$ has the gradient estimate and the $W^{2,p}$-estimate, as stated in \thmref{almost admissible singular}.
\end{thm}

The convergence behaviour is improved for the degenerate metrics.

\begin{thm}\label{existence degenerate}
Assume that $\Om$ is K\"ahler and the family $\{\vphi_{\eps}\}$, consisting of the approximate degenerate scalar curvature solutions \eqref{Degenerate cscK approximation}, has bounded entropy. Then there exists degenerate metric $\om_\vphi$ with prescribed scalar curvature, which is bounded, and has the gradient estimate and the $W^{2,p}$-estimate as stated in \thmref{almost admissible degenerate}. Moreover, $\om_\vphi$ is smooth and satisfies the degenerate scalar curvature equation \eqref{Degenerate cscK defn}  outside $D$.

Furthermore, if $\{\vphi_{\eps}\}$ has bounded gradients of the volume ratios $\|\p F_\eps\|_{\vphi_\eps}$, then the degenerate metric $\om_\vphi$ has Laplacian estimate, namely it is admissible when $\beta>\frac{n+1}{2}$, and $\gamma$-admissible for any $\gamma>0$ when $1<\beta< \frac{n+1}{2}$. 
\end{thm}
\begin{proof}
The first part of the theorem is a direct corollary of \thmref{existence singular}. The smooth convergence outside $D$ is obtained, by using the local Laplacian estimate in Section 6 of the arXiv version of \cite{MR4301557}. By applying the Evans-Krylov estimates and the bootstrap method, the solution is smooth on the regular part $X\setminus D$ and satisfies the equation there. 
The global Laplacian estimates for degenerate metrics are obtained from \thmref{cscK Laplacian estimate}. 
\end{proof}

In Section \ref{A priori estimates for approximate degenerate cscK equations}, we develop an integration method with weights for the general degenerate scalar curvature equation \eqref{Degenerate cscK approximation}. 

In particular, we utilise our theorems to the degenerate K\"ahler-Einstein metrics. 
We see that the integration method we develop here provide an alternative method to obtain the Laplacian estimate for the approximate degenerate K\"ahler-Einstein equation \eqref{critical pt Ding approximation}, as shown in Section \ref{Log Kahler Einstein metric}. While, Yau's proof of the Laplacian estimate applies the maximum principle.
\begin{cor}\label{existence log KE}
When $\lambda\leq 0$ and $\beta> 1$, there exists a family of smooth approximate K\"ahler-Einstein metrics \eqref{critical pt Ding approximation}, which converges to an admissible degenerate K\"ahler-Einstein metric.
\end{cor}





\section{$L^\infty$-estimate}\label{Linfty estimate}
This section is devoted to obtain the $L^\infty$-estimate for singular scalar curvature equation \eqref{Singular cscK t eps} for both $\vphi$ and $F$. 


\begin{thm}[$L^\infty$-estimate for singular equation]\label{L infty estimates Singular equation}
Assume $\vphi_{t,\eps}$ is a solution of the approximate singular scalar curvature equation \eqref{Singular cscK t eps}. Then we have the following estimates. 
\begin{enumerate}
\item For any $p\geq 1$, there exists a constant $A_1$ such that
	\begin{align*}
	  \|\vphi_{t,\eps}\|_\infty,\quad \sup_X[F_{t,\eps}-\sigma_s \phi_E],\quad \|e^{F_{t,\eps}}\|_{p;\tilde\om_t^n},\quad \|e^{\tilde F_{t,\eps}}\|_{p;\tilde\om_t^n}\leq A_1.
	\end{align*}
The constant $\sigma_s=C_l-\tau$ and $A_1$ depends on the entropy
$
	E^\beta_{t,\eps}=\frac{1}{V}\int_X F_{t,\eps}\om_{\vphi_{t,\eps}}^n,
$
the alpha invariant $\alpha(\Om_1)$, $\|e^{C_l\phi_E}\|_{L^{p_0}(\tilde \om_t^n)}$ for some $p_0\geq 1$ and  
\begin{align*}
\inf_X f_{t,\eps},\quad C_l=\inf_{(X,\tilde\om_t)}\theta,\quad\sup_X R,\quad n,\quad p.
\end{align*}

 \item There also holds the lower bound the volume ratio
	\begin{align*}
	\inf_{X}[F_{t,\eps}-\sigma_i\phi_E]\geq A_2, \quad \forall\tau>0.
	\end{align*}
	The constant $\sigma_i=C_u+\tau$ and $A_2$ depends on $\sup_XF_{t,\eps}$ and
	\begin{align*}
\inf_X f_{t,\eps}, \quad C_u=\sup_{(X,\tilde\om_{t})}\theta,\quad \inf_X R,\quad n.
	\end{align*}
\end{enumerate}
\end{thm}
Furthermore, since $\tilde F_{t,\eps}$ is dominated by $F+(\beta-1)\log S_\eps$ up to a smooth function, the $L^\infty$ estimate of $F_{t,\eps}$ gives the volume ratio bound.
\begin{cor}[Volume ratio]
\begin{align}\label{Gradient estimate: volume ratio bound}
C^{-1} e^{\sigma_i\phi_E}  S_\eps^{\beta-1}\tilde\om^n_t \leq \om^n_{\vphi_{t,\eps}}=e^{\tilde F_{t,\eps}}\tilde\om^n_t\leq C e^{\sigma_s\phi_E} S_\eps^{\beta-1}\tilde\om^n_t.
\end{align} 
\end{cor}

Before we start the proof, we state the corresponding estimate for the degenerate equation.
If we further assume that $\Om=[\om_K]$ is K\"ahler, then $\sigma_E=0$.
\begin{thm}[$L^\infty$-estimate for degenerate equation]\label{L infty estimates degenerate equation}
Assume $\vphi_{\eps}$ is a solution of the approximate degenerate scalar curvature equation \eqref{Degenerate cscK approximation}. 
Then for any $p\geq 1$, there exists a constant $A_0$ such that
	\begin{align*}
	  \|\vphi_{\eps}\|_\infty,\quad \|F_{\eps}\|_\infty,\quad \|e^{F_{\eps}}\|_{p;\om^n},\quad \|e^{\tilde F_{\eps}}\|_{p;\om^n}\leq A_0.
	\end{align*}

\end{thm}

Now we start the proof of \thmref{L infty estimates Singular equation}.
To proceed further, we omit the indexes $(t,\eps)$ of \eqref{Singular cscK t eps} for convenience, that is
\begin{equation}\label{Singular cscK t eps 3 short}
\om^n_\vphi=(\om_t+i\p\bar\p\vphi)^n=e^{F} \om_{\theta}^n,\quad
\tri_{\vphi} F=\tr_{\vphi}\theta-R.
\end{equation}


\subsection{General estimation}
In this section, we extend Chen-Cheng \cite{MR4301557} to the singular setting. We will first summarise a machinery to obtain the $L^\infty$-estimates in Proposition \ref{Linfty estimate v prop} and Proposition \ref{Estimation of I prop}. Then we apply this robotic method to conclude $L^\infty$-estimates under various conditions on $\theta$ in Section \ref{Applcations}.

We let $\vphi_a$ be an auxiliary function defined later in \eqref{vphi a} and set 
\begin{align}\label{Linfty estimate u}
w:=b_0F+b_1\vphi+b_2\vphi_a.
\end{align}
The singular scalar curvature equation \eqref{Singular cscK t eps 3 short} gives the identity.
\begin{lem}\label{Linfty estimate tri w}
We write $A_R:=-b_0R+b_1n+b_2\tr_\vphi\om_{\vphi_a}$. Then
\begin{align*}
\tri_\vphi w&=b_0(\tr_{\vphi}\theta-R)+b_1(n-\tr_\vphi\om_t)+b_2(\tr_\vphi\om_{\vphi_a}-\tr_\vphi\om_t)\\
&=b_0\tr_{\vphi}\theta-(b_1+b_2)\tr_\vphi\om_t+A_R.
\end{align*}
\end{lem}

\begin{rem}\label{Linfty estimate b1}
There are two ways to deal with $\om_t$, one is $\om_t=\tilde\om_t-i\p\bar\p\phi_E$, and the other one is $\om_t=\om_{sr}+t\om$. Thus, we have
\begin{align*}
\tri_\vphi w&=b_0\tr_{\vphi}\theta-(b_1+b_2)\tr_\vphi\tilde\om_t+(b_1+b_2)\tri_\vphi\phi_E+A_R\\
&=b_0\tr_{\vphi}\theta-(b_1+b_2)\tr_\vphi(\om_{sr}+t\om)+A_R.
\end{align*}
The constant $b_1$ will be chosen to be negative such that $-(b_1+b_2)$ is positive. Accordingly, we see that
$\tri_\vphi w\geq b_0\tr_{\vphi}\theta-t(b_1+b_2)\tr_\vphi\om+A_R.$
\end{rem}

We introduce a weight function $H$ and add it to $u$
\begin{align*}
u:=w-H 
\end{align*} and utilise various conditions on the $(1,1)$-form $\theta$, aiming to obtain a differential inequality from the identity in \lemref{Linfty estimate tri u}
\begin{align}\label{Linfty estimate tri u}
\tri_\vphi u\geq A_\theta\tr_\vphi\tilde\om_t+A_R.
\end{align} Here, we choose $b_1$ such that $A_\theta>0$.
Then the maximum principle is applied near the maximum of $u$ to conclude the estimate of $u$. 

Given a point $z\in X$ and a ball $B_d(z)\subset X$, we let $\eta$ be the local cutoff function $B_d(z)$ regarding to the metric $\tilde\om_t$ such that $\eta(z)=1$ and $\eta=1-b_3$ outside the half ball $B_{\frac{d}{2}}(z)$.
Then we have the standard estimate of the local cutoff function as 
\begin{align}\label{cutoff}
\tri_\vphi\log\eta\geq -A_{b_3} \tr_\vphi\tilde\om_t, \quad A_{b_3}:=[\frac{2b_3}{d(1-b_3)}]^2+\frac{4b_3}{d^2(1-b_3)}.
\end{align}

\begin{prop}\label{Linfty estimate v prop}
Assume that $u$ satisfy \eqref{Linfty estimate tri u}.
We define 
\begin{align}\label{Linfty estimate v}
v:=b_4  u=b_4 (b_0 F-H+b_1\vphi+b_2\vphi_a)
\end{align} and set $\tilde d=8d^{-2}$ and $0<b_3<<1$ satisfy
\begin{align}\label{Linfty estimate Atheta b3}
b_3= \frac{A_\theta}{\tilde d \cdot b_4^{-1}+A_\theta}\text{ such that }
A_{b_3}\leq \tilde d \frac{b_3}{1-b_3}=  b_4 A_\theta.
\end{align}  Then the following estimates hold.
\begin{enumerate}[(i)]
\item\label{Linfty estimate v prop 1} 
The upper bound of $v$ is $b_4^{-1}\sup_X v\leq C(n)d I^{\frac{1}{2n}}$ where
\begin{align}\label{I}
I:=b_3^{-2n} \int_{A_R\leq 0}  (A_R)_-^{2n} e^{2n v+2(F-f)}\tilde \om_t^n .
\end{align}

\item\label{Linfty estimate v prop 3} The upper bound of $b_0F-H$ is
\begin{align}\label{Linfty estimate F H}
b_0F-H\leq C(n)d  I^{\frac{1}{2n}}-b_2\vphi_a.
\end{align}

\item\label{Linfty estimate v prop 4}
Assume the exponent $p\geq 1$ and the constant $b_2$ satisfies
\begin{align}\label{Linfty estimate b2}
  b_2 p\leq \alpha(\Om_1).
\end{align}
Then it holds
\begin{align}\label{integral F auxiliary}
\|e^{b_0F-H}\|^p_{L^p(\om^n)}
\leq e^{p \cdot b_4^{-1}\sup_X v}\int_X e^{-\alpha\vphi_a}\om_1^n.
\end{align}

\item \label{Linfty estimate eF}
When $b_0=1$,
assume $e^H\in L^{p_H^+}(\tilde\om_t^n)$ for some $p_H^+\geq 1$. Then for any $\tilde p\leq \frac{pp_H^+}{p+p_H^+}$, we have $\|e^{\tilde F}\|_{L^{\tilde p}(\tilde\om_t^n)}\leq 
e^{-\inf_X f}\|e^{F}\|_{L^{\tilde p}(\tilde\om_t^n)}$ and
\begin{align*}
\|e^{F}\|_{L^{\tilde p}(\tilde\om_t^n)}\leq 
\|e^{F-H}\|_{L^p(\tilde\om_t^n)}\|e^{H}\|_{L^{p_H^+}(\tilde\om_t^n)},
\end{align*}
which implies $\|\vphi\|_\infty,\|\vphi_a\|_\infty\leq C$. Consequently, \eqref{Linfty estimate F H} becomes
\begin{align*}
F\leq H+C(n)d  I^{\frac{1}{2n}}-b_2\|\vphi_a\|_\infty.
\end{align*}

\item\label{Linfty estimate v prop 2}
When $b_0=1$, we assume $e^{-H}\in L^{p_H}(\tilde \om_t^n)
$ and $b_4$ satisfies 
\begin{align}\label{Linfty estimate b4}
	0<b_4\leq \min\{(-4nb_1)^{-1}\alpha(\Om_1),(4n)^{-1}p_H\}.
\end{align}
Then the estimate of $I$ is given in Proposition \ref{Estimation of I prop}.
\end{enumerate}
\end{prop}
\begin{proof}
Now we combine the inequality for $u$ and the cutoff function $\eta$. 
Inserting \eqref{Linfty estimate tri u} and \eqref{cutoff} to $\tri_\vphi (v+\log\eta)$, we get 
\begin{align*}
&\tri_\vphi (v+\log\eta)\geq (b_4 A_\theta-A_{b_3})\tr_\vphi\tilde\om_t+b_4 A_R \geq b_4 A_R.
\end{align*}
Furthermore, it implies 
\begin{align*}
\tri_\vphi (e^{v}\eta)\geq b_4 A_R e^{v}\eta.
\end{align*}
Therefore, applying the Aleksandrov maximum principle yo this differential inequality, we obtain the estimate of $v$ in \eqref{Linfty estimate v prop 1}.

Before we obtain the estimate of $I$ in \eqref{Linfty estimate v prop 2}, which will be given in Proposition \ref{Estimation of I prop}, we derive the upper bound of $b_0F-H$ and an $L^p$ bound of $e^{b_0F-H}$.

Since $\vphi, \vphi_a$ are $\om_t$-psh functions, we could modify a constant such that $\vphi,\vphi_a\leq 0$. Hence, $b_1\vphi\geq 0$ and the formula \eqref{Linfty estimate v} of $v$ gives 
\begin{align*}
b_0F-H\leq b_4^{-1}\sup_Xv-b_2\vphi_a.
\end{align*}
Thus the upper bound \eqref{Linfty estimate F H} of $b_0F-H$ in \eqref{Linfty estimate v prop 3} is obtained by inserting \eqref{I} to the inequality above.

Note that $\om_1=\om_{sr}+\om$ and $\omega_{sr}\geq0$, we have $\om\leq \om_1$ and $\vphi_a$ is also psh with respect to $\om_1=\om_{sr}+\om$.
Accordingly, we apply the $\alpha$-invariant of $\Om_1=[\om_1]$ to conclude the $L^p(\om^n)$ bound of $e^{b_0F-H}$ in \eqref{Linfty estimate v prop 4}.

The proof of \eqref{Linfty estimate eF} is given as following.
While, $\om$ is equivalent to $\tilde\om_t$, i.e. $\tilde\om_t\leq (C_K+t)\omega$, we could replace $L^p(\om^n)$ norm by $L^p(\tilde\om_t)$ norm, 
$$\|e^{b_0F-H}\|_{L^p(\tilde\om_t^n)}\leq (C_K+t)\|e^{b_0F-H}\|_{L^p(\om^n)}.$$
Under the hypothesis $e^H\in L^{p_0}(\tilde\om_t^n)$, the estimate of
$\|e^{F}\|_{ L^p(\tilde\om_t^n)}$ is obtained by applying the H\"older inequality to \eqref{integral F auxiliary}. Moreover, the upper bound of $e^{-f}$ from \lemref{nef tilde f} implies the $L^p(\tilde\om_t^n)$ bound of $e^{\tilde F}$. By the equation 
$
\om_\vphi^n=e^{\tilde F}\tilde\om_t^n,
$ the bound $\|e^{\tilde F}\|_{L^{\tilde p}(\tilde\om_t^n)}$ further
implies the uniform bound of $\vphi$, due to \lemref{big nef approximate reference metric bound}. 

The auxiliary function $\vphi_a$ is defined to be a solution to the following approximation of the singular Monge-Amp\`ere equation
\begin{align}\label{vphi a}
\om_{\vphi_a}^n=E^{-1 }\om_\theta^n e^F\sqrt{F^2+1},\quad E=V^{-1}\int_Xe^F\sqrt{F^2+1}\om_\theta^n.
\end{align}
Similarly, the volume element of the auxiliary function $\om_{\vphi_a}$ is also $L^p$, by applying $\|e^{F}\|_{L^{\tilde p}(\tilde\om_t^n)}$. For the same reason, $\vphi_a$ is also uniformly bounded, too.

Inserting the resulting estimate of $\vphi_a$ into \eqref{Linfty estimate F H}, we have the upper bound of $F-H$. 
\end{proof}



The rest of the proof is to estimate $I$, \eqref{Linfty estimate v prop 2} in Proposition \ref{Linfty estimate v prop}.
\begin{prop}[Estimation of $I$]\label{Estimation of I prop}
In general, we have
\begin{align}\label{Linfty estimate rough I bound}
I\leq e^{-2\inf_X f} \int_{A_R\leq 0}  |b_3^{-1}(b_0R-b_1n)|^{2n} e^{-2nb_4 H}e^{2nb_4b_1\vphi}e^{(2nb_4b_0+2)F} \tilde \om_t^n
\end{align} 
and in the integral domain $\{x\in X\vert A_R\leq 0\}$ of $I$,
\begin{align}\label{F and entropy}
F\leq (\frac{|b_0R-b_1n|}{b_2n})^n (E^\beta_{t,\eps}+2e^{-1}+1).
\end{align}

If we further assume that $|b_0R-b_1n|$ is bounded,
$
e^{-H}\in L^{p_H}(\tilde \om_t^n)
$
and $b_4$ satisfies \eqref{Linfty estimate b4}.
Then
\begin{align}\label{Linfty estimate b4 cor}
I^2\leq  C^2(C_K+1)^n \|e^{-H}\|^{p_H}_{L^{p_H}(\tilde\om_t^n)} \int_{X} e^{-\alpha(\Om_1)\vphi} \om_1^n,
\end{align} 
where $C=e^{-2\inf_X f}  e^{(2nb_4+2)\sup_{A_R\leq0}F} \|b_3^{-1}(b_0R-b_1n)\|^{2n}_\infty$. 
\end{prop}
\begin{proof}
We use the lower bound of $f$ from \lemref{nef tilde f} in \eqref{I},
\begin{align*}
I\leq e^{-2\inf_X f} \int_{A_R\leq 0}  (A_R)_-^{2n} e^{2n v}e^{2F}\tilde \om_t^n.
\end{align*}
Then we deal with each factors in the integrand.
Since $\vphi_a$ is an $\om_t$-psh function, we have $\vphi_a\leq 0$. Thus \eqref{Linfty estimate v} tells us
\begin{align*}
e^{2nv}\leq e^{2nb_4 b_0F} e^{-2nb_4 H}e^{2nb_4b_1\vphi}.
\end{align*}

Using the expression of $A_R$ in \lemref{Linfty estimate tri w}, we get 
$0\geq A_R\geq -b_0R+b_1n.$
So, we estimate 
\begin{align*}
(A_R)_-^{2n}\leq| b_0R-b_1n|^{2n}.
\end{align*}
Inserting these estimates into the expression of $I$, we obtain \eqref{Linfty estimate rough I bound}.

The bound of $F$ is obtained in terms of the entropy $E^\beta_{t,\eps}$, by using the auxiliary function $\vphi_a$.
 Since $A_R\leq 0$, applying the geometric mean inequality to the expression of $A_R$ in \lemref{Linfty estimate tri w}, we have
\begin{align*}
b_2 n( \frac{\om^n_\vphi}{\om^n_{\vphi_a}})^{\frac{-1}{n}}\leq b_2\tr_\vphi\om_{\vphi_a}\leq b_0R-b_1n.
\end{align*}
Substituting \eqref{vphi a} into the inequality above, we have
\begin{align*}
F\leq \sqrt{F^2+1}\leq (\frac{b_0R-b_1n}{b_2n})^n E.
\end{align*}
At last, we use that fact that $E$ is bounded by the entropy $E^\beta_{t,\eps}$ from Lemma 5.4 in \cite{arXiv:1803.09506}, to conclude \eqref{F and entropy}.

When $2nb_4b_0+2\geq 0$, we let $C=e^{-2\inf_X f} b_3^{-1} e^{(2nb_4b_0+2)\sup_X F}$ and get
\begin{align*}
I\leq  C  \int_{A_R\leq 0} (b_0R-b_1n)^{2n}  e^{-2nb_4 H}e^{2nb_4b_1\vphi}\tilde \om_t^n.
\end{align*} 
By H\"older inequality, it is further bounded by
\begin{align*}
 C\|b_0R-b_1n\|_\infty^{2n} \int_X e^{-4nb_4 H}\tilde \om_t^n\int_{X}  e^{4nb_4b_1\vphi}\tilde \om_t^n.
\end{align*} with $p_R> 2n$.
The integral $ \int_{X}   e^{-4nb_4 H}\tilde \om_t^n$ is finite, when $4nb_4\leq p_H$.
As \lemref{integral F auxiliary}, we use $\tilde\om_t\leq \om_K+\om\leq (C_K+1)\om$ by \lemref{metrics equivalence} and $\om\leq \om_1=\om_{sr}+\om$ by semi-positivity of $\om_{sr}$.
So, the integral 
\begin{align*}
\int_Xe^{4nb_4b_1\vphi}\tilde \om_t^n\leq (C_K+1)^n\int_Xe^{4nb_4b_1\vphi} \om_1^n
\end{align*} is also bounded, once $-4nb_4b_1\leq \alpha(\Om_1)$. 
\end{proof}
\begin{rem}
In the proof, we could use 
$ \int_X(b_0R-b_1n)^{p_R}\tilde \om_t^n $, $p_R>2n$ instead.
\end{rem}

\subsection{Applications}\label{Applcations}
Now we are ready to make use of different properties on $\theta$ to estimate $\vphi$ and $F$, with the help of Proposition \ref{Linfty estimate v prop} and its corollaries. 
We clarify the steps in practice. Firstly, we derive \eqref{Linfty estimate tri u} to write down the formulas of $A_\theta$, $A_R$ and $H$. Secondly, we ask $A_\theta$ to be strictly positive to determine the value of $b_1$. While, $b_2$ is chosen as \eqref{Linfty estimate b2} depending on the auxiliary function $\vphi_a$. Thirdly, we use the expression of $H$ to verify both conditions including $\|e^H\|_{L^{p_H^+}(\tilde\om_t^n)}$ and $\|e^{-H}\|_{L^{p_H}(\tilde \om_t^n)}$ in \eqref{Linfty estimate eF} of Proposition \ref{Linfty estimate v prop} and Proposition \ref{Estimation of I prop}, respectively. Consequently, $b_4$ is determined form $b_1$, $p_H$ in \eqref{Linfty estimate b4}. At last, we could obtain the value of $b_3$ by \eqref{Linfty estimate Atheta b3} and compute $I$ in Proposition \ref{Estimation of I prop}.

Now we start our applications.
As we observe in \cite{arXiv:1803.09506} that the particular property of the given $(1,1)$-form $\theta\in C_1(X,D)$ leads to various differential inequalities. The most general bound on $\theta$ is Definition \ref{L infty estimates theta defn}
\begin{align*}
C_l\cdot \tilde\om_t\leq \theta\leq  C_u\cdot\tilde\om_t.
\end{align*}
\begin{prop}\label{L infty estimates f sup vphi prop}
Suppose that $e^{C_l\phi_E}\in L^{p_0}(\tilde \om_t^n)$ for some $p_0\geq 1$.
Then there exists a constant $C$ such that for all $p>p_0$, $\tau>0$,
\begin{align}\label{L infty estimates f sup vphi}
\|e^{F}\|_{L^{p}(\tilde\om_t^n)}, \quad
\|e^{\tilde F}\|_{L^{p}(\tilde\om_t^n)},\quad
 \|\vphi\|_\infty, \quad \sup_X(F-\sigma_s\phi_E)\leq C,
\end{align}
where $\sigma_s=C_l-\tau$.
The constant $C$ depends on $E^\beta_{t,\eps}$, $C_l$, $\alpha(\Om_1)$, $\sup_X R$, $\inf_X f$, $\sup_X\phi_E,n,p$.
\end{prop}
\begin{proof}
We let $b_0=1$ and insert the lower bound \eqref{Cl} of $\theta$, i.e. $\theta\geq  C_l\tilde\om_t$, together with $\om_t=\tilde\om_t-i\p\bar\p\phi_E$ to \lemref{Linfty estimate tri w},
\begin{align*}
\tri_\vphi w
\geq C_l\tr_{\vphi}\tilde\om_t-(b_1+b_2)(\tr_\vphi\tilde\om_t-i\p\bar\p\phi_E)+A_R.
\end{align*}
Comparing with \eqref{Linfty estimate tri u}, we read from this inequality that 
\begin{align*}
H=(b_1+b_2)\phi_E,\quad A_\theta=C_l-(b_1+b_2),\quad A_R=-R+b_1n+b_2\tr_\vphi\om_{\vphi_a}.
\end{align*}  

Given a fixed $p\geq 1$, by \eqref{Linfty estimate b2}, we further take
$b_2=p^{-1}\alpha(\Om_1).$ Letting $A_\theta=t_0>0$, we have $b_1=C_l-b_2-t_0$. Also, from \eqref{Linfty estimate Atheta b3}, we get $b_3= \frac{t_0}{8d^{-2}b_4^{-1}+t_0}$.

As a result, we see  that $H=(C_l-t_0)\phi_E.$
Clearly, $-H$ is bounded above. Moreover, since $e^{C_l\phi_E}\in L^{p_0}(\tilde \om_t^n)$ and $e^{-t_0\phi_E}\in L^{p_1}(\tilde \om_t^n)$ as long as $t_0$ is small enough, we have 
\begin{align*}
e^{H}=e^{(C_l-t_0)\phi_E}\in L^{p^+_H}(\tilde \om_t^n)\text{ for all }p^+_H\leq \frac{p_0p_1}{p_0+p_1}.
\end{align*}
We also choose $b_4\leq (-4nb_1)^{-1}\alpha$ by \eqref{Linfty estimate b4}.
According to Proposition \ref{Estimation of I prop}, we could examine that the integral $I$ is finite and all estimates are independent of $t$ and $\eps$.

Therefore, the hypotheses of $H$ in \eqref{Linfty estimate v prop 2} in Proposition \ref{Linfty estimate v prop} and \eqref{Linfty estimate eF} of Proposition \ref{Linfty estimate v prop} are satisfied and we conclude the estimates of
$\|e^{F}\|_{L^{\tilde p}(\tilde\om_t^n)}$,  
$\|e^{\tilde F}\|_{L^{\tilde p}(\tilde\om_t^n)}$,
$\|\vphi\|_\infty$ and $\sup_X(F-H)\leq C$. 
\end{proof}

\begin{prop}\label{L infty estimates f inf prop}
Under the assumption in Proposition \ref{L infty estimates f sup vphi prop}, it holds for any $\tau>0$, 
\begin{align}\label{L infty estimates f inf}
\inf_X[F-\sigma_i\phi_E]\geq -C,\quad \sigma_i=C_u+\tau.
\end{align}
The constant $C$ depends on $\|\vphi\|_\infty$, $C_u$, $\inf_X R$, $\inf_X f$, $\sup_X\phi_E$, $n$.
\end{prop}
\begin{proof}
We apply Proposition \ref{Linfty estimate v prop} again, taking 
\begin{align*}
w=-F+b_1\vphi,\quad b_0=-1,\quad b_2=0.
\end{align*}
Substituting $\theta\leq C_u\tilde\om_t$ into \lemref{Linfty estimate tri w}, we obtain
\begin{align*}
\tri_\vphi w\geq -C_u\tr_\vphi\tilde\om_t-b_1\tr_\vphi\om_t+A_R, \quad A_R=R+b_1n.
\end{align*}
By $\tilde\om_t=\om_t+i\p\bar\p\phi_E$, it is reduced to  
\begin{align*}
\tri_\vphi w\geq -C_u\tr_\vphi\tilde\om_t-b_1\tr_\vphi(\tilde\om_t-i\p\bar\p\phi_E)+A_R.
\end{align*}
Accordingly, $H=b_1\phi_E$ and $A_\theta=-C_u-b_1$.

We choose $b_1=-C_u-t_0$ such that 
\begin{align*}
A_\theta=t_0,\quad H=-(C_u+t_0)\phi_E.
\end{align*}
We see that $e^{-H}$ is bounded above.

From \eqref{Linfty estimate F H}, $-F-H\leq C(n)d I^{\frac{1}{2n}}$.
We verify the estimate of $I$ in Proposition \ref{Estimation of I prop}. We let $b_4=\frac{1}{n}$. Due to \eqref{Linfty estimate Atheta b3}, we have 
$b_3=\frac{t_0}{8d^{-2}b_4^{-1}+t_0}$. Then we get
\begin{align}\label{Linfty estimate inf F I}
I\leq e^{-2\inf_X f} \int_{A_R\leq 0}  |b_3^{-1}(-R-b_1n)|^{2n} e^{-2 H}e^{-2(C_u+t_0)\vphi} \tilde \om_t^n,
\end{align}
which is finite, by using the upper bound of $-H=(C_u+t_0)\phi_E$ and $\inf_X\vphi$. Therefore, the upper bound of $-F+(C_u+t_0)\phi_E$ is derived.
\end{proof}

We observe that we could remove $\tau$.
\begin{cor}\label{L infty estimates f sup vphi prop cor}
Assume that $|R-C_ln|\leq A_s t$ for some constant $A_s$. Then we have \eqref{L infty estimates f sup vphi} with $\sigma_s=C_l$.
\end{cor}
\begin{proof}
The proof is identical to Proposition \ref{L infty estimates f sup vphi prop}, but choosing $b_0=1$, $A_\theta=t$ and $H=(C_l-t)\phi_E$. We also take $b_2=t$. Since $t\rightarrow 0$, we have
$ pb_2=pt\leq \alpha(\Om_1)$, which is the condition \eqref{Linfty estimate b2}.
Then $b_1=C_l-b_2-t=C_l-2t$ and $b_4=\min\{\frac{1}{n},\frac{\alpha}{-4nb_1}\}$. Also,
\begin{align*} 
\frac{t}{8d^{-2}b_4^{-1}+1} \leq b_3= \frac{t}{8d^{-2}b_4^{-1}+t}\leq \frac{1}{8d^{-2}b_4^{-1}} .
\end{align*} 
The scalar curvature assumption gives
\begin{align*}
|R-b_1n|=|R-(C_l-2t)n|\leq |R-C_ln|+2tn\leq (A_s+2n) t.
\end{align*} 
We have $F\leq (\frac{|R-b_1n|}{b_2n})^n (E^\beta_{t,\eps}+2e^{-1}+1)
$ is bounded in the domain $A_R\leq 0$.
We insert these values to \eqref{Linfty estimate rough I bound} in Proposition \ref{Estimation of I prop} to estimate 
\begin{align*}
I\leq e^{-2\inf_X f}  \int_{A_R\leq 0}  |b_3^{-1}(R-b_1n)|^{2n} e^{-2nb_4 H} e^{2nb_4b_1\vphi}e^{(2nb_4+2)F} \tilde \om_t^n.
\end{align*}  By further using the bound of $e^{-H}$ and $|b_3^{-1}(R-b_1n)|$, we have $I$ is bounded, if $b_4$ is smaller than $\min\{\frac{1}{n},\frac{\alpha}{-2nb_1q}\}$. Consequently, the estimates \eqref{L infty estimates f sup vphi} follow from Proposition \ref{Linfty estimate v prop} and \eqref{Linfty estimate eF} of Proposition \ref{Linfty estimate v prop}. We find that all constants are independent of $t$, so we could further take $t\rightarrow 0$ such that the weight $\sigma_s=C_l$.
\end{proof}

\begin{cor}\label{L infty estimates f inf prop cor}
Assume $|R-C_un|\leq A_s t$ for some constant $A_s$. Then \eqref{L infty estimates f inf} holds with $\sigma_i=C_u$.
\end{cor}
\begin{proof}
We learn from Proposition \ref{L infty estimates f inf prop} that $A_\theta=t,  H=-(C_u+t)\phi_E$. We have $b_2=0$ and $b_4=\frac{1}{n}$. Also, $b_1=-C_u-t$ is bounded when $0\leq t\leq 1$ and
\begin{align*}
b_3=\frac{t}{8d^{-2}b_4^{-1}+t}\geq Ct.
\end{align*}
Inserting the upper bound of $e^{-H}$ and $|b_3^{-1}(-R-b_1n)|$ in \eqref{Linfty estimate inf F I}, we get the estimate of $I$ is independent of $t$. Therefore, we conclude from Proposition \ref{L infty estimates f inf prop} with $t\rightarrow 0$.
\end{proof}
\begin{rem}
For the cscK problem, the averaged scalar curvature $R_t$ \eqref{approximate average scalar} of the approximate singular cscK metric is close to $\ul S_\beta=\frac{C_1(X,D)\Om^{n-1}}{\Om^n}$, namely
\begin{align*}
R_t-\ul S_\beta\leq C t.
\end{align*} So, the assumption in both corollaries means some sort of pinching of the eigenvalues of the representative $\theta\in C_1(X,D)$.
\end{rem}



\section{Gradient estimate of $\vphi$}\label{Gradient estimate of vphi}
In this section, we obtain the gradient estimate for the singular cscK metric, extending the results for non-degenerate cscK metrics \cite{MR4301557}. We will use the singular exponent $\sigma_E$ to measure the singularity of the given big and semi-positive class $\Om$. Meanwhile, we will use the degenerate exponent $\sigma_D$ to reflect how the degeneracy of the singular cscK equation \eqref{Singular cscK t eps} could effect the gradient estimate. It is surprising to see that, when the cone angle $\beta> \frac{n+2}{2}$, the degenerate exponent does not appear in the gradient estimate. That means the gradient estimates remains exactly the same to the one for the non-degenerate metric.

\begin{thm}[{Gradient estimate of $\vphi$}]\label{gradient estimate}
Suppose that $\vphi_\eps$ is a solution to the approximate singular scalar curvature equation \eqref{Singular cscK t eps tilde}. 
\begin{enumerate}
\item
Assume that $\Om$ is big and semi-positive. Then there exists a constant $C$ such that
\begin{align*}
 |s_E|^{2a_0 \sigma_E } S_\eps^{\sigma_D}|\p\psi|^2_{\tilde\om_t}\leq A_3, \quad \sigma_D\geq 1.
\end{align*}
The singular exponent $\sigma_E$ is sufficiently large, and determined in \eqref{Gradient estimate: sigmaE} and \eqref{Gradient estimate: sigmaE 2}.

\item
If we further assume that $\Om=[\om_K]$ is K\"ahler, then $\sigma_E=0$ and the degenerate exponent $\sigma_D$ is weaken to satisfy
\begin{equation}\label{gradient estimate: degenerate exponent}
\left\{
\begin{aligned}
 &\sigma_D=0,\text{ when }\beta>\frac{n+2}{2},\\
&\sigma_D>1-\frac{2\beta}{n+2}\geq 0,\text{ when } \beta\leq\frac{n+2}{2}.
   \end{aligned}
\right.
\end{equation} 
Then there holds the gradient estimate
\begin{align*}
S_\eps^{\sigma_D}|\p\vphi_\eps|_{\om_K}^2\leq A_4.
\end{align*}
\end{enumerate}
\end{thm}
The precise statements will be given in \thmref{gradient estimate sigma 1} and \thmref{gradient estimate sigma 2}, respectively.
\begin{rem}
In the second conclusion, we see that  $\sigma_D$ could be chosen to be zero, when $\beta> \frac{n+2}{2}$.
\end{rem}
\begin{proof}
We denote $\tilde\vphi$ by $\psi$ and all the norms are taken with respect to the K\"ahler metric $\tilde\om_t$ in this proof. We will use the approximate singular scalar curvature equation \eqref{Singular cscK t eps tilde} and omit the lower index for convenience. 
\begin{equation}\label{Singular cscK t eps tilde gradient}
(\tilde\om_t+i\p\bar\p \psi)^n=e^{\tilde F} \tilde\om_t^n,\quad
\tri_{\psi} \tilde F=\tr_{\psi}(\theta-i\p\bar\p f)-R,
\end{equation}
where, $
\tilde F=F-f, f=-(\beta-1)\log S_\eps-h_\theta-c_{t,\eps}-\log\frac{\om^n}{\tilde\om_t^n}.
$
We will divide the proof into several steps in this section.
\end{proof}

\subsection{Differential inequality}
Let $K$ be a positive constant determined later and $H$ be an auxiliary function on $\tilde F$ and $\psi$. 
We set 
\begin{align*}
&v:=|\p\psi|^2+K,\quad u:=e^{H} v.
\end{align*}
Then we calculate that
\begin{lem}\label{Gradient estimate: Differential identity lem}
\begin{align}\label{Gradient estimate: Differential identity}
u^{-1}\tri_\psi u&= \tri_\psi H+|\p H|^2_\psi
+2\frac{H_i\psi_{i\bar i}\psi_{\bar i}+H_i\psi_{\bar j\bar i}\psi_{j}}{(1+\psi_{i\bar i})v}\notag \\
&+\frac{R_{i\bar i j\bar j}(\tilde\om_t)\psi_{j}\psi_{\bar j}+|\psi_{ij}|^2+|\psi_{i\bar i}|^2}{(1+\psi_{i\bar i})v}+\frac{2Re(\tilde F_i\psi_{\bar i})}{v}.
\end{align}
\end{lem}
\begin{proof}
We compute the Laplacian of $\log v$ under the normal coordinates with respect to the metric $\tilde\om_t$, 
\begin{align*}
\tri_\psi \log v=-|\p\log v|^2_\vphi+\frac{R_{i\bar i j\bar j}(\tilde\om_t)\psi_{j}\psi_{\bar j}+|\psi_{ij}|^2+|\psi_{i\bar i}|^2}{(1+\psi_{i\bar i}) v}+\frac{2\Re(\tilde F_i\psi_{\bar i})}{v}.
\end{align*}
Applying 
$
\tri_\psi \log u=\tri_\psi H+\tri_\psi\log v,
$
we get
\begin{align*}
u^{-1}\tri_\psi u&=| \p\log u|^2_\psi+\tri_\psi \log u\\
&=|\p\log u|^2_\psi+\tri_\psi H-|\p\log v|^2_\psi\\
&+\frac{R_{i\bar i j\bar j}(\tilde\om_t)\psi_{j}\psi_{\bar j}+|\psi_{ij}|^2+|\psi_{i\bar i}|^2}{(1+\psi_{i\bar i}) v}+\frac{2\Re(\tilde F_i\psi_{\bar i})}{v}.
\end{align*}
We further calculate that
\begin{align*}
|\p\log u|^2_\psi-|\p\log v|^2_\psi&=(\p H,2\p\log v+\p H)_\psi=|\p H|^2_\psi+2(\p H,\p v)_\psi v^{-1}\\
&=|\p H|^2_\psi+2\frac{H_i\psi_{i\bar i}\psi_{\bar i}+H_i\psi_{\bar j\bar i}\psi_{j}}{(1+\psi_{i\bar i})v}.
\end{align*}
In summary, the desired identity is obtained by adding them together.
\end{proof}

The differential inequality is obtained from the identity \eqref{Gradient estimate: Differential identity}, after dropping off the positive terms and carefully choosing the weight function $H$.

Firstly, we remove the positive terms in \eqref{Gradient estimate: Differential identity}.

\begin{lem}\label{Gradient estimate: Differential inequality lem}
Let $K\geq 1$ and $-C_{1.1}:=\inf_{X}R_{i\bar ij\bar j}(\tilde\om_t)$. Then it holds
\begin{align}\label{tri Z}
u^{-1}\tri_\psi u&\geq  \tri_\psi H+2\Re[(H_i+\tilde F_i)\psi_{\bar i}]v^{-1}
\\
&+[(-C_{1.1}-1)+v^{-1}]\tr_{\psi}\tilde\om_t
+(\tr_{\tilde\om_t}\om_\psi-2n)v^{-1}.\notag
\end{align}
\end{lem}
\begin{proof}
By removing the positive terms
\begin{align*}
|\p H|^2_\psi|\p\psi|^2+2\frac{H_i\psi_{\bar j\bar i}\psi_{j}}{1+\psi_{i\bar i}}+\frac{|\psi_{ij}|^2}{1+\psi_{i\bar i}}\geq 0,
\end{align*}
and inserting the lower bound of the bisectional curvature 
$R_{i\bar i j\bar j}(\tilde\om_t)\psi_{j}\psi_{\bar j}\geq -C_{1.1} v,$
into \eqref{Gradient estimate: Differential identity}, we get
\begin{align*}
e^{-H}\tri_\psi u&\geq v\tri_\psi H+K|\p H|^2_\psi
+2\frac{H_i\psi_{i\bar i}\psi_{\bar i}}{1+\psi_{i\bar i}}\notag \\
&-C_{1.1}v\tr_\psi\tilde\om_t+\frac{|\psi_{i\bar i}|^2}{1+\psi_{i\bar i}}+2\Re(\tilde F_i\psi_{\bar i}).
\end{align*}
While, we compute
\begin{align*}
2\frac{H_i\psi_{i\bar i}\psi_{\bar i}}{1+\psi_{i\bar i}} 
=2H_i\psi_{\bar i}-2\frac{H_i\psi_{\bar i}}{1+\psi_{i\bar i}} .
\end{align*}
By Young's inequality, it follows
\begin{align*}
 -\frac{2H_i\psi_{\bar i}}{1+\psi_{i\bar i}}\geq -|\p H|^2_\psi -|\p \psi|^2\tr_\psi\tilde\om_t\geq -|\p H|^2_\psi -v\tr_\psi\tilde\om_t.
\end{align*}
By substitution into the inequality above, we have
\begin{align*}
e^{-H}\tri_\psi u&\geq v\tri_\psi H+(K-1)|\p H|^2_\psi
+(-C_{1.1}-1)v\tr_\psi\tilde\om_t\\
&+\frac{|\psi_{i\bar i}|^2}{1+\psi_{i\bar i}}+2\Re[(H_i+\tilde F_i)\psi_{\bar i}].
\end{align*}

Rewriting the positive term
\begin{align*}
&\frac{|\psi_{i\bar i}|^2}{1+\psi_{i\bar i}}=\frac{|\psi_{i\bar i}|^2-1+1}{1+\psi_{i\bar i}}
=\sum_i[\psi_{i\bar i}-1+\frac{1}{1+\psi_{i\bar i}}]\\
&=\tr_{\tilde\om_t}\om_\psi-2n+\tr_{\psi}\tilde\om_t
\end{align*} and inserting it back to the inequality, we have obtained \eqref{tri Z}.
\end{proof}

Then we continue the proof of Proposition \eqref{Gradient estimate: Differential inequality prop}.
In order to deal with the gradient term in \eqref{tri Z}, we further choose 
\begin{align*}
B=-\sigma_E\psi+e^{-\psi}, \quad H:=-F+B+\sigma_D\log S_\eps.
\end{align*}
In which, we see that
\begin{align*}
e^{-\psi}=e^{-\vphi+\phi_E}=e^{-\vphi} |s_E|^{2a_0}_{h_E},
\end{align*}
which is bounded by using the $L^\infty$-estimate of $\vphi$.

\begin{lem}
\begin{align}\label{Gradient estimate: Differential inequality gradient A}
&2Re[(H_i+\tilde F_i)\psi_{\bar i}]v^{-1}
\geq 2(-\sigma_E-e^{-\psi})(1-Kv^{-1})+A_w v^{-\frac{1}{2}}  
\end{align}
with the weight 
\begin{align}\label{Gradient estimate: Aw}
A_w:=
-2C_{2}-2(\beta-1+\sigma_D)|\p\log S_\eps|.
\end{align}
\end{lem}
\begin{proof} 
We compute 
\begin{align*}
B_i=(-\sigma_E-e^{-\psi})\psi_i\text{ and } H+\tilde F=B-f+\sigma_D\log S_\eps
\end{align*} to get
\begin{align*}
2Re[(H_i+\tilde F_i)\psi_{\bar i}]
= 2(-\sigma_E-e^{-\psi})|\p\psi|^2+2\Re[-f_i\psi_{\bar i}+\sigma_D(\log S_\eps)_i\psi_{\bar i}].
\end{align*}
Inserting $f=-(\beta-1)\log S_\eps-h_\theta-c_{t,\eps}-\log\frac{\om^n}{\tilde\om_t^n}$ into the identity above, we have that
\begin{align*}
-f_i\psi_{\bar i}+\sigma_D(\log S_\eps)_i\psi_{\bar i}
&=(h_\theta+\log\frac{\om^n}{\tilde\om_t^n})_i\psi_{\bar i}
+(\beta-1+\sigma_D)(\log S_\eps)_i\psi_{\bar i}\\
&\geq - [C_{2}+(\beta-1+\sigma_D)|\p\log S_\eps|]\cdot |\p\psi| ,
\end{align*}
where $C_2$ depends on $\|\p(h_\theta+\log\frac{\om^n}{\tilde\om_t^n}) \|_{L^\infty(\tilde\om_t)}$.
Consequently, we have proved \eqref{Gradient estimate: Differential inequality gradient A} by using $|\p\psi|^2=v-K$.
\end{proof}


Meanwhile, we bound the Laplacian term $\tri_\psi H$ in \eqref{tri Z}.
\begin{lem}
\begin{align}\label{Gradient estimate: Differential inequality laplacian A}
\tri_\psi H
\geq \tilde A_\theta\tr_\psi\tilde\om_t+e^{-\psi}|\p\psi|^2_\psi+(\inf R-\sigma_En-e^{-\psi}n).
\end{align}
where, $C_{1.4}:=\inf_{(X,\tilde\om_t)} i\p\bar\p\log S_\eps$ and
\begin{align*}
\tilde A_\theta=-\sup_{(X,\tilde\om_t)}\theta+\sigma_E+e^{-\psi}+\sigma_D C_{1.4}.
\end{align*}
\end{lem}
\begin{proof}
We take the Laplacian of $B=-\sigma_E\psi+e^{-\psi}$,
\begin{align*}
\tri_\psi B=-\sigma_E\tri_\psi\psi+e^{-\psi}|\p\psi|^2_\psi-e^{-\psi}\tri_\psi\psi,
\end{align*} Then we calculate the Laplacian of the auxiliary function $H$ with the singular scalar curvature equation \eqref{Singular cscK t eps tilde gradient},
\begin{align*}
&\tri_\psi H=-\tri_\psi  F+\tri_\psi B+ \sigma_D\tri_\psi \log S_\eps\\
&=-\tr_{\psi}\theta+R+e^{-\psi}|\p\psi|^2_\psi+(-\sigma_E-e^{-\psi})(n-\tr_\psi\tilde\om_t)
+ \sigma_D\tri_\psi \log S_\eps,
\end{align*}
which implies the lower bound
\begin{align*}
\tri_\psi H&\geq[-\sup_{(X,\tilde\om_t)}\theta+\sigma_E+e^{-\psi}]\tr_\psi\tilde\om_t+\inf R+e^{-\psi}|\p\psi|^2_\psi-(\sigma_E+e^{-\psi})n\\
&+ \sigma_D\tri_\psi \log S_\eps.
\end{align*}
Since $\sigma_D\geq 0$, the asserted inequality follows from the lower bound of $ i\p\bar\p \log S_\eps$ from \lemref{h eps}, namely 
\begin{align*}
\tri_\psi\log S_\eps\geq C_{1.4}\tr_\psi\tilde\om_t.
\end{align*}
\end{proof}

We set three bounded quantities to be
\begin{align*}
A_\theta:=\tilde A_\theta-C_{1.1}-1=-\sup_{(X,\tilde\om_t)}\theta+\sigma_E+e^{-\psi}+\sigma_D C_{1.4}-C_{1.1}-1,
\end{align*}
\begin{align}\label{Gradient estimate: constants A}
& A_s:=\inf R-(n+2)(\sigma_E+e^{-\psi}),\quad A_c:=2K(\sigma_E+e^{-\psi})-2n.
\end{align}
As a result, inserting the functions $A_s,A_c,A_w$, \eqref{Gradient estimate: Differential inequality gradient A} and \eqref{Gradient estimate: Differential inequality laplacian A} back to \eqref{tri Z}, we obtain that
\begin{align*}
u^{-1}\tri_\psi u&\geq  A_\theta  \tr_\psi\tilde\om_t
+e^{-\psi}|\p\psi|^2_\psi +A_s
+(\tr_{\tilde\om_t}\om_\psi +A_c)v^{-1} +A_wv^{-\frac{1}{2}} .
\end{align*} 
In order to apply \lemref{Gradient estimate: Differential inequality G} to deduce the lower bound of the first two terms, we need to verify the assumptions
\begin{align*}
A_\theta\geq C_K e^{-\psi},\quad C_K:=1+\frac{K}{n^2(n-1)}.
\end{align*}
They are satisfied, as long as we choose large $\sigma_E$, i.e.
\begin{align}\label{Gradient estimate: sigmaE}
\sigma_E\geq \frac{K}{n^2(n-1)} e^{-\psi}+\sup_{(X,\tilde\om_t)}\theta-\sigma_D C_{1.4}+C_{1.1}+1.
\end{align}
Consequently, we obtain the desired differential inequality asserted in the following proposition.
\begin{prop}\label{Gradient estimate: Differential inequality prop}
It holds
\begin{align}\label{Gradient estimate: Differential inequality}
\tri_\psi u&\geq  \frac{1}{n-1}e^{-\frac{H+\tilde F}{n}-\psi} u^{1+\frac{1}{n}}\notag\\
&+A_s u
+e^H(\tr_{\tilde\om_t}\om_\psi+A_c) +A_{w}e^{\frac{H}{2}} u^{\frac{1}{2}}.
\end{align}  
\end{prop}

At last, we close this section by proving the required lemma.
\begin{lem}\label{Gradient estimate: Differential inequality G}
For any given function $K$, there holds
 \begin{align*}
G:=C_K\tr_\psi\tilde\om_t+|\p\psi|^2_\psi
 \geq \frac{1}{n-1}e^{-\frac{\tilde F}{n}}  v^{\frac{1}{n}}.
 \end{align*}  
\end{lem}
\begin{proof}
 By the inequality of arithmetic and geometric means and the equation \eqref{Singular cscK t eps tilde}, we have 
 \begin{align*}
\tr_\psi\tilde\om_t\geq (\tr_{\tilde\om_t}\om_\psi\cdot e^{-\tilde F})^{\frac{1}{n-1}}.
\end{align*} 
Jensen inequality for the concave function $(\cdot)^{\frac{1}{n-1}}$ further implies that
\begin{align*}
(\tr_{\tilde\om_t}\om_\psi)^{\frac{1}{n-1}}\geq \frac{1}{n}\sum_i(1+\psi_{i\bar i})^{\frac{1}{n-1}} n^{\frac{1}{n-1}}.
\end{align*}
Inserting them into the expression of $(n-1)G$, we get
 \begin{align*}
(n-1)G
&\geq \sum_i[\frac{n-1}{n}(1+\psi_{i\bar i})^{\frac{1}{n-1}}e^{-\frac{\tilde F}{n-1}}n^{\frac{1}{n-1}}+\frac{K}{n^2(1+\psi_{i\bar i})}+\frac{(n-1)|\psi_i|^2}{1+\psi_{i\bar i}}].
\end{align*} 
By $n\geq 2$ and $n-1\geq \frac{1}{n}$, it follows
 \begin{align*}
\geq \sum_i[\frac{n-1}{n}(1+\psi_{i\bar i})^{\frac{1}{n-1}}e^{-\frac{\tilde F}{n-1}}+\frac{1}{n}\frac{|\psi_i|^2+n^{-1}K}{1+\psi_{i\bar i}}].
\end{align*}
Using Young's inequality, we thus obtain the lower bound of $(n-1)G$
 \begin{align*}
\geq\sum_i [(1+\psi_{i\bar i})^{\frac{1}{n-1}}e^{-\frac{\tilde F}{n-1}}]^{\frac{n-1}{n}} \cdot [\frac{|\psi_i|^2+n^{-1}K}{1+\psi_{i\bar i}}]^{\frac{1}{n}}
= e^{-\frac{\tilde F}{n}} (|\p\psi|^2+K)^{\frac{1}{n}}.
\end{align*}
 \end{proof}
 
\subsection{Computing weights}
We are examining the coefficients of the inequality \eqref{Gradient estimate: Differential inequality}.
\begin{prop}\label{Gradient estimate: Differential inequality c prop}
Assume that 
\begin{align}\label{Gradient estimate: sigmaE 2}
\sigma_E\geq \sigma_i=C_u+\tau.
\end{align} Then there exists nonnegative constants $C,C_1,C_2$ such that
\begin{align}
&\tri_\psi u\geq  C_1[A_m u^{1+\frac{1}{n}}
+e^H\tr_{\tilde\om_t}\om_\psi]-C_2[u+1-A_w e^{\frac{H}{2}} u^{\frac{1}{2}}],\label{Gradient estimate: Differential inequality c}\\
&A_m:=(n-1)^{-1}e^{-\frac{H+\tilde F}{n}-\psi}\geq C|s_E|_{h_E}^{-2a_0(\frac{\sigma_E}{n}+1)}  S_\eps^{-\frac{\beta-1+\sigma_D}{n}  },\label{Gradient estimate: Differential inequality main}\\
&A_w e^{\frac{H}{2}}\geq - C[1+(\beta-1+\sigma_D)S_\eps^{-\frac{1}{2}}] S_\eps^{\frac{\sigma_D}{2}}  |s_E|^{a_0(\sigma_E-\sigma_i)} .\label{Gradient estimate: Differential inequality Aw}
\end{align}
\end{prop}
\begin{proof}
We compute the weights and coefficients in \eqref{Gradient estimate: Differential inequality prop}, including $$A_s, \quad A_c,\quad  e^H, \quad e^{-\frac{H+\tilde F}{n}-\psi},\quad A_w.$$
Since $e^{-\psi}$ is bounded, we see that both $A_s$ and $A_c$ in \eqref{Gradient estimate: constants A} are bounded.

Applying the uniform bounds of $\vphi$ and $F$ from \thmref{L infty estimates Singular equation}, namely $$A_1+\sigma_s\phi_E\geq F\geq \sigma_i\phi_E+A_2$$ to $H=-F-\sigma_E\vphi+\sigma_E\phi_E+e^{-\psi}+\sigma_D\log S_\eps$, we get
\begin{align}\label{Gradient estimate: Differential inequality eH}
C^{-1} e^{(\sigma_E-\sigma_s)\phi_E} S_\eps^{\sigma_D}\leq e^H\leq C e^{(\sigma_E-\sigma_i)\phi_E}S_\eps^{\sigma_D}.
\end{align}
Since $\sigma_E\geq \sigma_i$, we know that $e^{(\sigma_E-\sigma_i)\phi_E}$ is bounded above. 

Inserting $-F+\tilde F=-f$ and $f=-(\beta-1)\log S_\eps-h_\theta-c_{t,\eps}-\log\frac{\om^n}{\tilde\om_t^n}$ into $-\frac{H+\tilde F}{n}-\psi$, we have
\begin{align*}
&-\frac{H+\tilde F}{n}-\psi
=-\frac{-f-\sigma_E\vphi+\sigma_E\phi_E+e^{-\psi}+\sigma_D\log S_\eps}{n}-\vphi+\phi_E\\
&=-\frac{h_\theta+c_{t,\eps}+\log\frac{\om^n}{\tilde\om_t^n}-(\sigma_E+n)(\vphi-\phi_E)+e^{-\psi}+(\beta-1+\sigma_D)\log S_\eps}{n},
\end{align*}
which has the lower bound 
\begin{align*}
-\frac{H+\tilde F}{n}-\psi
\geq -C-(\frac{\sigma_E}{n}+1)\phi_E-\frac{\beta-1+\sigma_D}{n}\log S_\eps.
\end{align*}
As a result, we obtain the strictly positive lower bound of $A_m$ in \eqref{Gradient estimate: Differential inequality main},
which depends on the upper bound of $|s_E|^2_{h_E}$ and $S_\eps$.

By the same proof of \lemref{nef tilde f}, we get
\begin{align*}
A_w=
-2C_{2}-2(\beta-1+\sigma_D)|\p\log S_\eps|
\geq -C[1+(\beta-1+\sigma_D)S_\eps^{-\frac{1}{2}}]   .
\end{align*}
Thus \eqref{Gradient estimate: Differential inequality Aw} in obtained by using \eqref{Gradient estimate: Differential inequality eH}.
\end{proof}
\subsection{Maximum principle}

We continue the proof of \thmref{gradient estimate}. We have proved that, when $\sigma_D\geq 1$ and $\sigma_E\geq \sigma_i$, it follows according to Proposition \ref{Gradient estimate: Differential inequality c prop} that
$$A_we^{\frac{H}{2}}\geq -C S_\eps^{\frac{\sigma_D-1}{2}}   |s_E|^{a_0(\sigma_E-\sigma_i)},$$
which is bounded below, near $D$ and $E$.
Therefore, there exists a non-negative constant $C_1,C_2$ such that
\begin{align}\label{Gradient estimate: Differential inequality 1}
\tri_\psi u&\geq  C_1 u^{1+\frac{1}{n}}-C_2[u+ u^{\frac{1}{2}}+1].
\end{align}

We are ready to apply the maximum principle to prove the gradient estimate, \thmref{gradient estimate}, when $\sigma_D\geq1$.
\begin{thm}\label{gradient estimate sigma 1}
Assume that $\Om$ is big and semi-positive.
Suppose that $\vphi_\eps$ is a solution to the approximate singular scalar curvature equation \eqref{Singular cscK t eps tilde}.
Then there exists a constant $A_3$ such that
\begin{align*}
 |s_E|^{2a_0 (\sigma_E-\sigma_s )} S_\eps|\p\psi|^2_{\tilde\om_t}\leq A_3.
\end{align*}
and $\sigma_E$ is determined in \eqref{Gradient estimate: sigmaE} and \eqref{Gradient estimate: sigmaE 2}.
The constant $A_3$ depends on 
\begin{align*}
&\|\vphi_\eps\|_\infty,\quad \sup_X(F_{t,\eps}-\sigma_s\phi_E),\quad\inf_{X}[F_{t,\eps}-\sigma_i\phi_E], \\
&\inf_{i\neq j}R_{i\bar i j\bar j}(\tilde\om_t),\quad\inf_X R,\quad \sup_{(X,\tilde\om_t)}\theta,\quad \|h_\theta+c_{t,\eps}+\log\frac{\om^n}{\tilde\om_t^n} \|_{C^1(\tilde\om_t)},\\
&\sup_X  S_\eps,\quad \sup_X  S_\eps^{-\frac{1}{2}}|\p S_\eps|_\om,\quad \Theta_D,\quad \sup_X \phi_E,\quad \beta,\quad n.
\end{align*}
\end{thm}
\begin{proof}
We proceed with the argument of the maximum principle.
We assume that $p$ is at the maximum point of $Z$. Then the inequality \eqref{Gradient estimate: Differential inequality 1} at $p$ implies that $u(p)$ is uniformly bounded above. But at any point $x\in X$, $u(x)$ is bounded above by the maximum value $u(p),$ which means 
\begin{align*}
|\p\psi|^2(x)+K
\leq   e^{-H(x)}u(p).
\end{align*}
Therefore, using $H=-F-\sigma_E\psi+e^{-\psi}+\sigma_D\log S_\eps$ and the upper bound of $F$, $\|\vphi\|_\infty$, $\|e^{-\psi}\|_\infty$, we obtain the gradient estimate
\begin{align*}
|\p\psi|^2(x)+K&\leq   e^{F+\sigma_E\vphi-\sigma_E\phi_E-e^{-\psi}-\sigma_D\log S_\eps}u(p)\\
&\leq C  e^{\sigma_s\phi_E-\sigma_E\phi_E-\sigma_D\log S_\eps}\\
&= C  |s_E|^{2(\sigma_s-\sigma_E)a_0} S_\eps^{-\sigma_D}.
\end{align*}
At last, the desired weighted gradient estimate is obtained.
\end{proof}

\subsection{Integration method}
In this section, we aims to improve $\sigma_D$ to be less than 1. The weighted integral method is applied to obtain the gradient estimate. In order to explain the ideas well, we restrict ourselves in the K\"ahler class $\Om$.
\begin{thm}\label{gradient estimate sigma 2}
Assume that $\Om=[\om_K]$ is K\"ahler.
Suppose that $\vphi_\eps$ is a solution to the approximate singular scalar curvature equation \eqref{Singular cscK t eps tilde}.  Assume that the degenerate exponent satisfies the condition \eqref{gradient estimate: degenerate exponent}, that is
\begin{align*}
\sigma_D>\max\{1-\frac{2\beta}{n+2},0\}.
\end{align*}
Then there holds the gradient estimate
\begin{align*}
S_\eps^{\sigma_D}|\p\vphi_\eps|_{\om_K}^2\leq A_4.
\end{align*}
The constant $A_4$ depends on the same dependence of $A_3$ in \thmref{gradient estimate sigma 1}, and in additional, the Sobolev constant of $\tilde\om_t$.
\end{thm}
\begin{rem}
Note $\tilde\om_t$ is equivalent to $\om_K$.
\end{rem}
We first obtain the general integral inequality. Then we apply the weighted analysis similar to the proof for the Laplacian estimate in Section \ref{A priori estimates for approximate degenerate cscK equations} to proceed the iteration argument.

\begin{prop}[Integral inequality]\label{Gradient estimate: Integral inequality}
Assume that $\Om$ is big and semi-positive and $\vphi_\eps$ is a solution to the approximate singular scalar curvature equation \eqref{Singular cscK t eps}. We set $\tilde u:=u+K_0$. Then we have
\begin{align*}
& \int_X   \tilde u^{p-1} u^{1+\frac{1}{n}}e^{[\sigma_i-(\frac{\sigma_E}{n}+1)]\phi_E}  S_\eps^{\beta-1-\frac{\beta-1+\sigma_D}{n}}\tilde\om^n_t\\
&+\int_X\tilde u^{p-2}[  (p-1) |\p u|_\psi^2+\tilde ue^{(\sigma_E-\sigma_i)\phi_E} S_\eps^{\sigma_D} \tr_{\tilde\om_t}\om_\psi ]e^{\sigma_i\phi_E}  S_\eps^{\beta-1}\tilde\om^n_t\notag\\
&\leq C  \int_X \tilde u^{p-1} [ u + e^{-(\frac{\sigma_E}{n}+1)\phi_E } 
+u^{\frac{1}{2}} e^{(\sigma_E-\sigma_i)\phi_E} S_\eps^{\frac{\sigma_D-1}{2}} ]e^{\sigma_s\phi_E}S_\eps^{\beta-1}\tilde\om^n_t.\notag
\end{align*} 
The dependence of $C$ is the same to what of $A_3$.
\end{prop}
\begin{proof}
Substituting the differential inequality \eqref{Gradient estimate: Differential inequality c} into the identity
\begin{align*}
&(p-1)\int_X \tilde u^{p-2}|\p u|_\psi^2\om_\psi^n=\int_X \tilde u^{p-1}(-\tri_\psi u)\om_\psi^n,\quad p\geq 2,
\end{align*} 
we have the integral inequality
\begin{align*}
& \int_X \tilde u^{p-1} u^{1+\frac{1}{n}}e^{-(\frac{\sigma_E}{n}+1)\phi_E } S_\eps^{-\frac{\beta-1+\sigma_D}{n}  } \om_\psi^n\\
&+\int_X\tilde u^{p-2}[(p-1) |\p u|_\psi^2+\tilde ue^H\tr_{\tilde\om_t}\om_\psi ]\om_\psi^n\\
&\leq C  \int_X \tilde u^{p-1} [ u
+e^{-(\frac{\sigma_E}{n}+1)\phi_E } 
+ u^{\frac{1}{2}} e^{(\sigma_E-\sigma_i)\phi_E}  S_\eps^{\frac{\sigma_D-1}{2}}]\om_\psi^n.
\end{align*} 
Inserting the volume ratio bound \eqref{Gradient estimate: volume ratio bound},
$
\frac{1}{C} e^{\sigma_i\phi_E}  S_\eps^{\beta-1}\ \leq \frac{\om^n_{\psi}}{\tilde\om^n_t}\leq C e^{\sigma_s\phi_E} S_\eps^{\beta-1}
$ 
into the inequality above, we obtain the integral inequality.
\end{proof}

In the following sections, we focus on the degenerate scale curvature equation in a given K\"ahler class, where $a_0=t=0$ and $\phi_E=0$. Notions introduced in Section \ref{Perturbed Kahler metrics} become $$\om_{sr}=\tilde\om_{t=0}=\om_{t=0}=\om_K,\quad \om_\psi=\om_\vphi.$$ The following corollary is obtained immediately from inserting the bound of $e^H$ in \eqref{Gradient estimate: Differential inequality eH} into Proposition \ref{Gradient estimate: Integral inequality}.
\begin{cor}[Integral inequality: degenerate equation]\label{Gradient estimate: Integral inequality degenerate}
Assume that $\Om$ is K\"ahler. Then the integral inequality in Proposition \ref{Gradient estimate: Integral inequality} reduces to
\begin{align*}
LHS_1+LHS_2\leq C  \cdot RHS_1,
\end{align*} 
where, we set 
\begin{align*}
LHS_1&:=\int_X\tilde u^{p-2}[(p-1)|\p u|_\psi^2 
+\tilde u S_\eps^{\sigma_D}\tr_{\tilde\om_t}\om_\psi] \om_\psi^n,\notag\\
LHS_2&:=\int_X  \tilde u^{p-1} u^{1+\frac{1}{n}}   S_\eps^{-\frac{\beta-1+\sigma_D}{n}} \om_\psi^n,\\
RHS_1&:=  \int_X  \tilde u^{p-1}  [ u 
+1 
+u^{\frac{1}{2}} S_\eps^{\frac{\sigma_D-1}{2}} ] \om_\psi^n.
\end{align*} 
\end{cor}

\subsection{Rough iteration inequality}
We now deal with the lower bound of $LHS_1$ by applying the Sobolev inequality.
We introduce some notions for convenience, $q:=p-\frac{1}{2}$, $\chi:=\frac{2n}{2n-1}$ and
\begin{align*}
k:=p\sigma_D+\beta-1,\quad
k_\sigma:=\frac{\sigma_D}{2}+\beta-1+\sigma,\quad \tilde\mu:=S^{k_\sigma\chi}_\eps\tilde\om_t^n.
\end{align*} 
Clearly, $k_\sigma=k+\sigma-q\sigma_D$ and 
\begin{align*}
\|\tilde u\|^{q}_{L^{q\chi}(\tilde\mu)}=[\int_X \tilde u^{q\chi}\tilde\mu ]^{\chi^{-1}}=[\int_X (\tilde u^{p-\frac{1}{2}}S^{k_\sigma}_\eps)^{\frac{2n}{2n-1}}\tilde\om_t^n ]^{\frac{2n-1}{2n}}.
\end{align*} 
\begin{prop}[Rough iteration inequality]\label{Gradient estimate: Rough iteration inequality}
It holds
\begin{align*}
LHS_0+\sqrt{q} LHS_2
\leq C (\sqrt{q}RHS_1+RHS_2),
\end{align*} 
where, we denote $LHS_0:=[\int_X (u\tilde u^{q-1})^{\chi}\tilde\mu]^{\chi^{-1}}$ and
\begin{align}\label{Gradient estimate: RHS2}
RHS_2:=\int_X u\tilde u^{q-1}S^{\frac{\sigma_D}{2}+\sigma}_\eps\om_\psi^n+k_\sigma\int_X u\tilde u^{q-1} S^{\frac{\sigma_D-1}{2}+\sigma}_\eps \om_\psi^n.
\end{align} 
\end{prop}
\begin{proof}
Applying Young's inequality to the $LHS_1$ of the integral inequality, \corref{Gradient estimate: Integral inequality degenerate}, we get
\begin{align*}
LHS_1\geq \sqrt{p-1}\int_X \tilde u^{p-\frac{3}{2}}S_\eps^{\frac{\sigma_D}{2}} |\p u|\om_\psi^n
=q^{-1}\sqrt{p-1}\int_X |\p \tilde u^{q}|  S_\eps^{\frac{\sigma_D}{2}} \om_\psi^n.
\end{align*} 
Making use of the key \lemref{Laplacian estimate: key trick}, we get
\begin{align*}
LHS_1\geq q^{-1}\sqrt{p-1}\int_X |\p (u\tilde u^{q-1})|  S_\eps^{\frac{\sigma_D}{2}} \om_\psi^n.
\end{align*} 
The lower bound of the volume ratio $F$ in \eqref{Gradient estimate: volume ratio bound}  further gives
\begin{align*}
LHS_1&\geq Cq^{-1}\sqrt{\frac{1}{2}(p-\frac{1}{2})}\int_X |\p (u\tilde u^{q-1})|S^{\frac{\sigma_D}{2}+\beta-1}_\eps\tilde\om_t^n\\
&\geq Cq^{-\frac{1}{2}}\int_X |\p(u \tilde u^{q-1})|S^{\frac{\sigma_D}{2}+\beta-1+\sigma}_\eps\tilde\om_t^n\\
&\geq C q^{-\frac{1}{2}}[\int_X |\p (u\tilde u^{q-1}S^{k_\sigma}_\eps)|\tilde\om_t^n
-\int_X u\tilde u^{q-1} |\p S^{k_\sigma}_\eps|  \tilde\om_t^n].
\end{align*} 
The Sobolev imbedding theorem estimates the first part,
\begin{align*}
\int_X |\p (u\tilde u^{q-1}S^{k_\sigma}_\eps)|\tilde\om_t^n
&\geq C_S^{-1}\left[\int_X (u\tilde u^{q-1}S^{k_\sigma}_\eps)^{\chi}\tilde\om_t^n \right]^{\chi^{-1}}
-\int_X u\tilde u^{q-1}S^{k_\sigma}_\eps\tilde\om_t^n\\
&=C_S^{-1} LHS_0-\int_X u\tilde u^{q-1}S^{k_\sigma}_\eps\tilde\om_t^n.
\end{align*} 
Meanwhile, the second part is estimated by
\begin{align*}
-\int_X u\tilde u^{q-1} |\p S^{k_\sigma}_\eps|  \tilde\om_t^n
\geq -k_\sigma\int_X u\tilde u^{q-1} S^{k_\sigma-\frac{1}{2}}_\eps \tilde\om_t^n.
\end{align*} 
In summary, we use the volume ratio bound \eqref{Gradient estimate: volume ratio bound} again and combine these inequalities together to get the lower bound 
\begin{align*}
LHS_1&\geq C q^{-\frac{1}{2}}\left\{C_S^{-1}LHS_0
-RHS_2\right\}.
\end{align*} 
Therefore, the required inequality is obtained by inserting the inequality of $LHS_1$ to the integral inequality \eqref{Gradient estimate: Integral inequality degenerate}, namely
$
LHS_1+LHS_2\leq C  \cdot RHS_1.
$

\end{proof}

\begin{cor}\label{Gradient estimate: Rough iteration inequality cor}
\begin{align*}
\|\tilde u\|^{q}_{L^{q\chi}(\tilde\mu)}+\sqrt{q} LHS_2
\leq C (\sqrt{q}RHS_1+RHS_2+1).
\end{align*} 
\end{cor}
\begin{proof}
It follows from $u=\tilde u-K_0$ that
\begin{align*}
LHS^\chi_0=\int_X (u\tilde u^{q-1})^{\chi}\tilde\mu
\geq C(n)[\int_X \tilde u^{q\chi}\tilde\mu-K_0^{\chi}
\int_X \tilde u^{(q-1)\chi}\tilde\mu] .
\end{align*}
By Young's inequality, which states that
\begin{align*}
 \tilde u^{(q-1)\chi}\leq   \frac{q-1}{q}\tilde u^{q\chi}+\frac{1}{q}\leq \tilde u^{q\chi}+1,
 \end{align*}
 if follows
  \begin{align*}
LHS^\chi_0=\int_X (u\tilde u^{q-1})^{\chi}\tilde\mu
\geq C(n)[(1-K_0^{\chi})\int_X \tilde u^{q\chi}\tilde\mu-K_0^{\chi}\int_X\tilde\mu] .
\end{align*}
 Without loss of generality, we normalise $\int_X\tilde\mu=1$.
 Then picking small $K_0$ satisfying $K_0^{\chi}\leq \frac{1}{2}$, we get
 \begin{align*}
LHS^\chi_0
\geq C(n)[\frac{1}{2}\int_X \tilde u^{q\chi}\tilde\mu-K_0^{\chi}] 
\geq \frac{C(n)}{2}[\int_X \tilde u^{q\chi}\tilde\mu-1].
\end{align*}
Inserting to the rough iteration inequality, Proposition \ref{Gradient estimate: Rough iteration inequality}, we obtain the desired result. 

\end{proof}
\subsection{$L^p$ control}
In this section, we derive the $L^p$ estimate of $u$ from the integral inequality, Corollary \ref{Gradient estimate: Integral inequality degenerate}.
We set
\begin{align*}
\widetilde{LHS_2}:=\int_X   \tilde u^{p+\frac{1}{n}}   S_\eps^{-\frac{\beta-1+\sigma_D}{n}} \om_\psi^n,
\end{align*} 
which will be used to bound $RHS_1$.
\begin{prop}[$L^p$ control]\label{Gradient estimate: Lp control prop}
Assume that the degenerate exponent satisfies the condition \eqref{gradient estimate: degenerate exponent}. Then
\begin{align*}
&LHS_2\leq \widetilde{LHS_2} \leq C\cdot RHS_1
\leq  C(p), \quad \forall p\geq 2.
\end{align*} 
\end{prop}
\begin{proof}
We keep the first term on the left hand side of the integral inequality from \corref{Gradient estimate: Integral inequality degenerate}, 
\begin{align*}
&LHS_2=\int_X  \tilde u^{p-1} u^{1+\frac{1}{n}}   S_\eps^{-\frac{\beta-1+\sigma_D}{n}} \om_\psi^n\\&\leq  C RHS_1=  \int_X  \tilde u^{p-1}  [ u 
+1 
+u^{\frac{1}{2}} S_\eps^{\frac{\sigma_D-1}{2}} ] \om_\psi^n.\notag
\end{align*} 

Replacing $\tilde u$ with $u+K_0$, we get
\begin{align*}
 \widetilde{LHS_2}
\leq C(K_0)[LHS_2+\int_X \tilde u^{p-1}   \tilde   S_\eps^{-\frac{\beta-1+\sigma_D}{n}} \om_\psi^n].
\end{align*}
Then Young's inequality with conjugate exponents $\frac{p+\frac{1}{n}}{p-1}$ gives
\begin{align*}
\int_X \tilde u^{p-1}   \tilde   S_\eps^{-\frac{\beta-1+\sigma_D}{n}} \om_\psi^n
\leq \tau\widetilde{LHS_2}+C(\tau)\int_X   \tilde   S_\eps^{-\frac{\beta-1+\sigma_D}{n}} \om_\psi^n.
\end{align*}
The latter integral is bounded by a constant $C$ independent of $p$, since
\begin{align*}
-\frac{\beta-1+\sigma_D}{n}+2(\beta-1)+2n>0.
\end{align*}
In conclusion, we shows that
\begin{align}\label{Gradient estimate: Integral inequality degenerate partial}
 \widetilde{LHS_2}\leq C(LHS_2+1)\leq C(RHS_1+1).
\end{align}

In order to proceed further, we apply Young's inequality to each term in $RHS_1$.
For the first term in $RHS_1$, we pick the conjugate exponents
$
p_1=\frac{p+n^{-1}}{p},q_1=\frac{p+n^{-1}}{n^{-1}}
$, then
\begin{align*}
\int_X  \tilde u^{p}\om_\psi^n&=\int_X (\tilde  u^{p} S_\eps^{-\frac{\beta-1+\sigma_D}{n p_1}} )S_\eps^{\frac{\beta-1+\sigma_D}{n p_1}} \om_\psi^n 
\leq \tau\widetilde{LHS_2}
+C(p,\tau)\int_X S_\eps^{k_1}\om^n_\psi.\notag
\end{align*} 
The exponent $k_1$ is then computed as
\begin{align*}
k_1:=\frac{\beta-1+\sigma_D}{n}\frac{q_1}{p_1}
=(\beta-1+\sigma_D)p\geq 0.
\end{align*}

Similarly, we estimate the second term as
\begin{align*}
\int_X  \tilde  u^{p-1}\om_\psi^n&=\int_X ( \tilde u^{p-1} S_\eps^{-\frac{\beta-1+\sigma_D}{n p_2}} )S_\eps^{\frac{\beta-1+\sigma_D}{np_2}}\om_\psi^n 
\leq \tau\widetilde{LHS_2}
+C(p,\tau)\int_X S_\eps^{k_2}\om^n_\psi
\end{align*}
with the conjugate exponents $p_2=\frac{p+n^{-1}}{p-1}, q_2=\frac{p+n^{-1}}{1+n^{-1}}
$ and
\begin{align*}
k_2&:=\frac{(\beta-1+\sigma_D)(p-1)}{n+1}\geq 0.
\end{align*}

The third term reduces to
\begin{align*}
 \int_X \tilde  u^{p-\frac{1}{2}}S_\eps^{\frac{\sigma_D-1}{2}}  \om_\psi^n
 &=\int_X (\tilde  u^{p-\frac{1}{2}} S_\eps^{-\frac{\beta-1+\sigma_D}{n p_3}} )S_\eps^{\frac{\beta-1+\sigma_D}{n p_3}+\frac{\sigma_D-1}{2}} \om_\psi^n \\
&\leq \tau\widetilde{LHS_2}
+C(p,\tau)\int_X S_\eps^{k_3}\om^n_\psi
\end{align*}
with $p_3=\frac{p+n^{-1}}{p-\frac{1}{2}}, q_3=\frac{p+n^{-1}}{\frac{1}{2}+n^{-1}}
$ and
\begin{align*}
k_3:=(\frac{\beta-1+\sigma_D}{n p_3}+\frac{\sigma_D-1}{2})q_3
=\frac{(\beta-1)(2p-1)-(np+1)}{n+2}+\sigma_D p.
\end{align*}

In order to make the last integrand $S_\eps^{k_3}$ integrable, we need 
\begin{align*}
k_3+2(\beta-1)+2n>0,
\end{align*} 
 which is equivalent to
\begin{align*}
[(n+2)\sigma_D+2(\beta-1)] p+(\beta-1)(2n+3)+2n(n+2)>np+1.
\end{align*} 
Clearly, it is sufficient to ask $(n+2)\sigma_D+2(\beta-1)\geq n$.
Therefore, we see that the integrable condition holds when $\sigma_D\geq1-\frac{2\beta}{n+2}$.

Adding them together, we have obtained that there exists a constant $C(p)$ depending on $p$ such that
\begin{align*}
RHS_1\leq 3 \tau\widetilde{LHS_2}+C(p,\tau) \int_X (S_\eps^{k_1}+S_\eps^{k_2}+S_\eps^{k_3})\om^n_\psi<C(p)
\end{align*}
Inserting it into \eqref{Gradient estimate: Integral inequality degenerate partial}, we get
\begin{align*}
 \widetilde{LHS_2}\leq C[3\tau  \widetilde{LHS_2}+C(p,\tau)\int_X (S_\eps^{k_1}+S_\eps^{k_2}+S_\eps^{k_3})\om^n_\psi+1]
\end{align*}
Taking sufficiently small $\tau$, we obtain the estimate of  $\widetilde{LHS_2}$ and $RHS_1.$
\end{proof}

\subsection{Weighted inequality}

In this section, we deal with the term
\begin{align*}
  \int_X    \tilde u^{q} S_\eps^{\frac{\sigma_D}{2}+\sigma-k'}\om^n_\psi,
\end{align*} 
which appears on the right hand side of the rough iteration inequality, Corollary \ref{Gradient estimate: Rough iteration inequality cor}. We determine the range of $k'$ in this term such that it is bounded by $\|u\|^{q}_{L^{qa}(\tilde\mu)}$ for some $1<a<\frac{2n}{2n-1}$.

\begin{prop}[Weighted inequality]\label{Gradient estimate: Weighted inequality}
Assume $n\geq 2$ and $ k'<\frac{1}{2}$.
Then there exists $1<a<\frac{2n}{2n-1}$ such that
\begin{align*}
  \int_X    \tilde u^{q} S_\eps^{\frac{\sigma_D}{2}+\sigma-k'}\om^n_\psi
 \leq C_{5.1} \|u\|^{q}_{L^{qa}(\tilde\mu)}.
\end{align*} 
Here $\frac{1}{a}+\frac{1}{c}=1$ and $C_{5.1}=\| S_\eps^{k_\sigma-k_\sigma\chi-k'}\|_{L^{c}(\tilde\mu)}$ is finite for some $c>2n$.
\end{prop}
\begin{proof}
With the help of the bound of the volume ratio $\tilde F$, we get
\begin{align*}
 \int_X    \tilde u^{q} S_\eps^{\frac{\sigma_D}{2}+\sigma-k'}\om^n_\psi
  \leq C&\int_X \tilde u^{q}S_\eps^{k_\sigma-k'}  \tilde\om_t^n
=C\int_X \tilde u^{q}S_\eps^{k_\sigma-k_\sigma\chi-k'} \tilde\mu.
\end{align*}
By applying the generalisation of H\"older's inequality with the conjugate exponents $a,c$, it is
\begin{align*}
\leq C\|\tilde u\|^{q}_{L^{qa}(\tilde\mu)}(\int_X S_\eps^{(k_\sigma-k_\sigma\chi-k')c} \tilde\mu)^{\frac{1}{c}}.
\end{align*}
As we have seen before, the last integral is finite, if we let 
\begin{align*}
2(k_\sigma-k_\sigma\chi-k')c+2k_\sigma\chi+2n>0,
\end{align*} which is equivalent to
\begin{align*}
c<2n\frac{k_\sigma+2^{-1}(2n-1)}{k_\sigma+k'(2n-1)}:=c_0.
\end{align*} 
The assumption $k'<\frac{1}{2}$ infers that $c_0>2n$. Consequently, the exponent $c$ is chosen to be between $2n$ and $c_0$ such that $a<\frac{2n}{2n-1}$.
\end{proof}
\subsection{Inverse weighted inequality}
We further estimate the the right hans side of the rough iteration inequality, Corollary \ref{Gradient estimate: Rough iteration inequality cor}, which contains two parts, $RHS_1$ and $RHS_2$,
\begin{align*}
RHS_1&=  \int_X  \tilde u^{p-1}  [ u 
+1 
+u^{\frac{1}{2}} S_\eps^{\frac{\sigma_D-1}{2}} ] \om_\psi^n,\\
RHS_2&=\int_X \tilde u^{q-1} u S^{\frac{\sigma_D}{2}+\sigma}_\eps\om_\psi^n+k_\sigma\int_X \tilde u^{q-1} u S^{\frac{\sigma_D-1}{2}+\sigma}_\eps \om_\psi^n.
\end{align*} 

By comparing the weights in each term on the both sides of the rough iteration inequality,
\begin{align*}
-\frac{\beta-1+\sigma_D}{n}, \quad \frac{\sigma_D-1}{2},
\end{align*}roughly speaking,
we observe that the critical cone angle is 
\begin{align*}
 \frac{\beta-1}{n}=\frac{1}{2}.
\end{align*}
We will see accurate proof in the next inverse weighted inequality.

We examine all five integrals in $RHS_1$ and $RHS_2$.
\begin{align*}
& I:=\int_X \tilde u^{p-1} u \om_\psi^n ,\quad
 II:=\int_X \tilde u^{p-1}\om_\psi^n,\quad
III:=\int_X \tilde u^{p-1} u^{\frac{1}{2}}S_\eps^{\frac{\sigma_D-1}{2}}\om_\psi^n ,\\
 &VI:=\int_X\tilde u^{q-1} u S^{\frac{\sigma_D}{2}+\sigma}_\eps\om_\psi^n,\quad
 V:=\int_X\tilde u^{q-1} u S_\eps^{\frac{\sigma_D-1}{2}+\sigma}\om_\psi^n.
\end{align*}

\begin{prop}[Inverse weighted inequality]\label{Gradient estimate: Inverse weighted inequality}
Assume that the degenerate exponent satisfies $\sigma_D<1$, $\sigma=0$ and the condition \eqref{gradient estimate: degenerate exponent}, equivalently,
\begin{align}
 &\sigma_D=0,\text{ when }\frac{\beta-1}{n}>\frac{1}{2};\label{Gradient estimate: Inverse weighted inequality angle}\\
&\sigma_D>1-\frac{2\beta}{n+2}\geq 0,\text{ when } \frac{\beta-1}{n}\leq\frac{1}{2}.\label{Gradient estimate: Inverse weighted inequality sigmaD}
   \end{align}
Then there exists $0<k'<\frac{1}{2}$ such that
\begin{align*}
I,II,III,VI,V
\leq \eps  \widetilde{LHS_2}
+C(\eps, n)\int_X  \tilde u^{q}S_\eps^{\frac{\sigma_D}{2}+\sigma-k'}\om_\psi^n.
\end{align*} 
\end{prop}
\begin{proof}

We will apply the $\eps$-Young's inequality and proceed similar to the proof of Proposition \ref{Gradient estimate: Lp control prop}, by using the positive term
\begin{align*}
&LHS_2=\int_X  \tilde u^{p-1} u^{1+\frac{1}{n}}   S_\eps^{-\frac{\beta-1+\sigma_D}{n}} \om_\psi^n.
\end{align*} But the constant should be chosen to be independent of $p$.

In order to estimate the first one, we decompose the first term as
\begin{align*}
I&=\int_X   \tilde u^{p-1} u \om_\psi^n\\
&=\int_X   (  \tilde u^{p-1} u^{1+\frac{1}{n}}    S_\eps^{-\frac{\beta-1+\sigma_D}{n}})^{\frac{1}{a_1}} 
 \tilde u^{\frac{p-1}{b_1}} u^{1-(1+\frac{1}{n})\frac{1}{a_1}}S_\eps^{(\frac{\beta-1+\sigma_D}{n})\frac{1}{a_1}}\tilde\om_t^n\\
&\leq \eps  LHS_2
+C(\eps, n)\int_X  \tilde u^{p-1}u^{1-\frac{b_1}{a_1 n}}S_\eps^{k_1}\om_\psi^n.
\end{align*}
By the choice of the conjugate exponents $a_1:=\frac{n+2}{n}$ and $b_1:=\frac{n+2}{2}$,
we get $1-\frac{b_1}{a_1 n}=\frac{1}{2}$ and
\begin{align*}
I\leq \eps  LHS_2+C(\eps, n)\int_X  \tilde u^{q}S_\eps^{k_1}\om_\psi^n.
\end{align*}
We further compute 
\begin{align*}
k_1&=(\frac{\beta-1+\sigma_D}{n})\frac{b_1}{a_1}=\frac{\beta-1+\sigma_D}{2}.
\end{align*}
Due to the weighted inequality, Proposition \ref{Gradient estimate: Weighted inequality}, we need to put condition on
\begin{align*}
k_1'=\frac{\sigma_D}{2}+\sigma-k_1=\frac{-(\beta-1)}{2}+\sigma.
\end{align*}

The second one is direct, by $\tilde u\geq K_0$,
\begin{align*}
II=\int_X \tilde u^{p-1}\om_\psi^n\leq \int_X \tilde u^{q} K_0^{-\frac{1}{2}}\om_\psi^n.
\end{align*}
So, the exponent is
\begin{align*}
k_2'=\frac{\sigma_D}{2}+\sigma.
\end{align*}

The third term $III$ is
\begin{align*}
&\int_X \tilde u^{p-1} u^{\frac{1}{2}} S_\eps^{\frac{\sigma_D-1}{2}}\om_\psi^n\\
&=\int_X   (  \tilde u^{p-1} u^{1+\frac{1}{n}}   S_\eps^{-\frac{\beta-1+\sigma_D}{n}})^{\frac{1}{a_3}}  \tilde u^{\frac{p-1}{b_3}} u^{\frac{1}{2}-(1+\frac{1}{n})\frac{1}{a_3}}S_\eps^{\frac{\sigma_D-1}{2}+(\frac{\beta-1+\sigma_D}{n})\frac{1}{a_3}}\om_\psi^n.
\end{align*}
By Young's inequality, 
\begin{align*}
III\leq \eps  LHS_2
+C(\eps, n)\int_X  \tilde u^{p-1}u^{\frac{b_3}{2}-(1+\frac{1}{n})\frac{b_3}{a_3}}S_\eps^{k_3}\om_\psi^n.
\end{align*}
We choose the exponents $b_3=\frac{3n+4}{2n+4}> 1$ such that
\begin{align*}
\frac{b_3}{2}-(1+\frac{1}{n})\frac{b_3}{a_3}=\frac{1}{4}.
\end{align*}
Accordingly, 
\begin{align*}
III
\leq \eps  LHS_2
+C(\eps, n)\int_X  \tilde u^{p-\frac{3}{4}}S_\eps^{k_3}\om_\psi^n.
\end{align*}
Since $\tilde u\geq K_0$, we have $\tilde u^{-\frac{1}{4}}\leq K_0^{-\frac{1}{4}}$.
Hence,
\begin{align*}
III\leq \eps  LHS_2
+C(\eps, n, K_0)\int_X  \tilde u^{q}S_\eps^{k_3}\om_\psi^n.
\end{align*}
We calculate the exponent $k_3=[\frac{\sigma_D-1}{2}+(\frac{\beta-1+\sigma_D}{n})a_3^{-1}]b_3$, which is 
\begin{align*}
k_3=\frac{\sigma_D(na_3+2)-na_3+2(\beta-1)}{2n(a_3-1)}
=(\frac{\sigma_D}{2}+\sigma)-k'_3.
\end{align*}
Consequently, we have
\begin{align*}
k'_3=\frac{-\sigma_D(n+2)-2(\beta-1)+na_3}{2n(a_3-1)}+\sigma.
\end{align*}

We estimate the fourth one,
\begin{align*}
&IV=\int_X \tilde u^{q-1} uS^{\frac{\sigma_D}{2}+\sigma}_\eps\om_\psi^n\leq \int_X     \tilde u^{q}S_\eps^{k_4}\om_\psi^n.
\end{align*}
where
\begin{align*}
k_4=\frac{\sigma_D}{2}+\sigma
=(\frac{\sigma_D}{2}+\sigma)-k_4',\quad k'_4=0.
\end{align*}

The fifth term is then estimated by
\begin{align*}
V\leq\eps LHS_2+C(\eps,n)\int_X     \tilde u^{q}S_\eps^{k_5}\om_\psi^n.
\end{align*}
The exponent $k_5=[(\frac{\sigma_D-1}{2}+\sigma)+(\frac{\beta-1+\sigma_D}{n})a_5^{-1}]b_5$ is further reduced to
\begin{align*}
k_5=\frac{\sigma_D(na_5+2)+na_5(2\sigma -1)+2(\beta-1)}{2n(a_5-1)}
=(\frac{\sigma_D}{2}+\sigma)-k_5'
\end{align*}
and 
\begin{align*}
k_5'=\frac{-\sigma_D(n+2)-2(\beta-1)+na_5}{2n(a_5-1)}-\frac{\sigma}{a_5-1}.
\end{align*}

At last, we verify the conditions such that $k_i'<\frac{1}{2}, i=1,2,3,4,5$. 
From $\sigma=0$ and $\sigma_D<1$, we have
\begin{align*}
k_1'=\frac{-(\beta-1)}{2}+\sigma<\frac{1}{2},\quad k'_2=\frac{\sigma_D}{2}+\sigma<\frac{1}{2}, \quad
k'_4=0<\frac{1}{2}.
\end{align*}
By using the weight condition \eqref{Gradient estimate: Inverse weighted inequality angle}, namely $\beta-1>\frac{n}{2}$, we have
\begin{align*}
k'_3=\frac{-2(\beta-1)+na_3}{2n(a_3-1)}<\frac{-n+na_3}{2n(a_3-1)}=\frac{1}{2}.
\end{align*}
Meanwhile, under the weight condition \eqref{Gradient estimate: Inverse weighted inequality sigmaD}, i.e. $\sigma_D>1-\frac{2\beta}{n+2}$, we see that
\begin{align*}
k'_3&=\frac{-\sigma_D(n+2)-2(\beta-1)+na_3}{2n(a_3-1)}\\
&<\frac{-[1-\frac{2\beta}{n+2}](n+2)-2(\beta-1)+na_3}{2n(a_3-1)}
=\frac{-n+na_3}{2n(a_3-1)}=\frac{1}{2}.
\end{align*}
The expression of $k'_5$ is the same to $k_3'$ when $\sigma=0$. So, $$k'_5<\frac{1}{2},$$ provided the weight conditions \eqref{Gradient estimate: Inverse weighted inequality angle} and \eqref{Gradient estimate: Inverse weighted inequality sigmaD} respectively.
In conclusion, simplify choosing $k'=\max\{k_1',k_2',k_3',k_4',k_5'\}$, we arrive at our desired inverse weighted inequality.
\end{proof}



\subsection{Iteration inequality}

In this section, we combined all inequalities in the previous sections to obtain the iteration inequality and complete the iteration scheme. 
\begin{prop}[Iteration inequality]\label{Gradient estimate: Iteration inequality prop}
\begin{align}\label{Gradient estimate: Iteration inequality}
\|\tilde u\|_{L^{q\chi}(\tilde\mu)}
\leq C q^{\frac{1}{2q}} \|\tilde u\|_{L^{qa}(\tilde \mu)}, \quad 1<a<\chi.
\end{align} 
\end{prop}
\begin{proof}
Since $q=p-\frac{1}{2}>1$, we obtain the rough iteration inequality \corref{Gradient estimate: Rough iteration inequality cor}, which states
\begin{align*}
\|\tilde u\|^{q}_{L^{q\chi}(\tilde\mu)}+\sqrt{q}LHS_2
\leq C\sqrt{q} (RHS_1+RHS_2+1).
\end{align*} 
Applying the inverse weighted inequality, Proposition \ref{Gradient estimate: Inverse weighted inequality}, we have
\begin{align*}
RHS_1+RHS_2
\leq 4\tau  LHS_2
+C(\tau, n)\int_X  \tilde u^{q}S_\eps^{\frac{\sigma_D}{2}+\sigma-k'}\om_\psi^n.
\end{align*} 
Choosing sufficiently small $\tau$ and inserting back to the rough iteration inequality, we get
\begin{align*}
\|\tilde u\|^{q}_{L^{q\chi}(\tilde\mu)}
\leq C\sqrt{q}( \int_X  \tilde u^{q}S_\eps^{\frac{\sigma_D}{2}+\sigma-k'}\om_\psi^n+1).
\end{align*}
Then the weighted inequality, Proposition \ref{Gradient estimate: Weighted inequality}, implies the desired interation inequality
\begin{align*}
\|\tilde u\|^{q}_{L^{q\chi}(\tilde\mu)}
\leq C\sqrt{q}(  \|\tilde u\|^{q}_{L^{qa}(\tilde\mu)}+1).
\end{align*}

\end{proof}

Finally, we finish the proof of the gradient estimate, \thmref{gradient estimate}. We assume $\|\tilde u\|_{L^{q_0a}(\tilde \mu)} \geq 1$ for some $q_0\geq\frac{3}{2}$ and rewrite the iteration inequality \eqref{Gradient estimate: Iteration inequality} by using $\tilde\chi:=\frac{\chi}{a}>1$,
\begin{align*}
\|\tilde u\|_{L^{q \chi }(\tilde\mu)}
\leq C^{q^{-1}} q^{\frac{1}{2q}} \|\tilde u\|_{L^{qa}(\tilde \mu)} .
\end{align*} 
To proceed the iteration argument, we set $q=\tilde\chi^m$, then we have
\begin{align*}
&\|\tilde u\|_{L^{\tilde\chi^{m}\chi }(\tilde\mu)}
\leq C^{\tilde\chi^{-m}} \tilde\chi^{\frac{m}{2\tilde\chi^m}} \|\tilde u\|_{L^{\tilde\chi^m a}(\tilde \mu)}
= C^{\tilde\chi^{-m}}  \tilde\chi^{\frac{m}{2\tilde\chi^m}} \|\tilde u\|_{L^{\tilde\chi^{m-1} \chi}(\tilde \mu)}\\ 
&\leq C^{\tilde\chi^{-m}+\tilde\chi^{-m+1}}  \tilde\chi^{\frac{m}{2\tilde\chi^m}+\frac{m-1}{2\tilde\chi^{m-1}}} \|\tilde u\|_{L^{\tilde\chi^{m-2} a}(\tilde \mu)}\\
&\leq\cdots\leq  C^{\sum_{i_0\leq i\leq m} \tilde\chi^{-i}} \tilde\chi^{\sum_{i_0\leq i\leq m}\frac{i}{2\tilde\chi^i}} \|\tilde u\|_{L^{ \tilde\chi^{i_0-1} a}(\tilde \mu)}.
\end{align*} 
At the final step, we choose sufficiently large $m$ and let $i_0$ satisfy
\begin{align*}
\tilde q_0:=\tilde\chi^{i_0-1} a\geq q_0.
\end{align*}

Since two series in the coefficient are convergent, it remains to check the bound of the last integral
$\|\tilde u\|^{\tilde q_0}_{L^{ \tilde q_0}(\tilde \mu)}$, which is equals to 
\begin{align*}
\|\tilde u\|^{\tilde q_0}_{L^{ \tilde q_0}(\tilde \mu)}
=\int_X \tilde u^{\tilde q_0}S_\eps^{k_\sigma\chi}\tilde\om_t^n
=\int_X \tilde u^{\tilde q_0}S_\eps^{[\frac{\sigma_D}{2}+\beta-1+\sigma]\frac{2n}{2n-1}}\tilde\om_t^n.
\end{align*}
In order to estimate the bound of $\|\tilde u\|^{\tilde q_0}_{L^{ \tilde q_0}(\tilde \mu)}$ from the $L^p$ bound of $u$ proved in Proposition \ref{Gradient estimate: Lp control prop}, i.e.
\begin{align*}
\int_X \tilde  u^{\tilde q_0+\frac{1}{n}}   S_\eps^{-\frac{\sigma_D}{n}+\frac{n-1}{n}(\beta-1)} \tilde\om_t^n \leq  C(\tilde q_0),
\end{align*} 
we verify that $\tilde u^{\tilde q_0}<\tilde  u^{\tilde q_0+\frac{1}{n}} $, $ \sigma=0 $ and the weights
\begin{align*}
\frac{\sigma_D}{2}\frac{2n}{2n-1}>-\frac{\sigma_D}{n}, \quad(\beta-1)\frac{2n}{2n-1}>\frac{n-1}{n}(\beta-1).
\end{align*} 

Let  $m\rightarrow\infty$, we thus obtain the gradient estimate of $|\p\psi|^2$ from $\tilde u=e^H(|\p\psi|^2+K)+K_0$ and the bound of $e^H$.

\section{$W^{2,p}$-estimate}\label{W2p estimate}

\begin{thm}[$W^{2,p}$-estimate for singular equation]\label{w2pestimates degenerate Singular equation}
For any $p\geq 1$, there exits a constant $A_5$ such that
\begin{align*}
\int_X   (\tr_{\tilde\om_t}\om_\vphi)^{p}   |s_E|^{\sigma_E}_{h_E} S_\eps^{\sigma_D}  \tilde\om_t^n\leq A_5,\quad  \sigma_D>(\beta-1)\frac{n-2-2np^{-1}}{n-1+p^{-1}}
\end{align*} 
where, the singular exponent $\sigma_E$ is defined as following
\begin{align}\label{w2pestimates sigma E}
\sigma_E>2a_0\frac{(n-2)\sigma_s+(b_1-\sigma_sp)p(n-1)-2np^{-1}}{n-1+p^{-1}}.
\end{align} and $b_1$ is given in \eqref{W2p estimate AR b1} and \eqref{W2p estimate b1 2}. The constant $A_5$ depends on the $L^\infty$-estimates
\begin{align*}
&\sup_X(F_{t,\eps}-\sigma_s\phi_E),\quad\inf_X(F-\sigma_i\phi_E),\quad \|\vphi\|_\infty
\end{align*}
and the quantities of the background metric $\tilde\om_t$,
\begin{align*}
&-C_{1.1}=\inf_{i\neq j}R_{i\bar i j\bar j}(\tilde\om_t), \quad \|e^{-f}\|_{p_0,\tilde\om_t^n},p_0\geq p+1,\quad \sup_{(X,\tilde\om_t)} i\p\bar\p f,\\
 &\sup_X\phi_E,\quad
 \sup_{(X,(1+\eps)\om_{K})}\theta,\quad\inf_X R,\quad Vol([\tilde\om_t]),\quad n , \quad p.
\end{align*}
\end{thm}
\begin{rem}
When $n=2$, we have $\sigma_D=0$.
\end{rem}
We obtain the $W^{2,p}$-estimate for degenerate equation as below.
\begin{thm}\label{w2pestimates degenerate equation}
Suppose that $\Om=[\om_K]$ is K\"ahler. Then the singular exponent $\sigma_E$ vanishes and 
\begin{align*}
\int_X   (\tr_{\om_K}\om_{\vphi_\eps})^{p}   S_\eps^{\sigma_D}  \om_K^n\leq A_5,\quad  \sigma_D:=(\beta-1)\frac{n-2}{n-1+p^{-1}}.
\end{align*} 
Moreover, written in terms of the volume element $\om_{\vphi_\eps}^n$, 
\begin{align}\label{w2pestimates degenerate equation simple}
\int_X   (\tr_{\om_K}\om_{\vphi_\eps})^{p}   S_\eps^{-\frac{\beta-1}{n-1}}  \om_{\vphi_\eps}^n\leq A_6.
\end{align} 
\end{thm}

In the proof, we omit the indexes as before. We also write $\psi:=\tilde\vphi=\vphi-\phi_E$.
We will use the following notations in this section,
\begin{align}\label{w2p u}
v:=\tr_{\tilde\om_t}\om_\vphi,\quad \tilde v:=v+K,\quad  K\geq 0.
\end{align}
The proof is divided into the following steps.

\subsection{Differential inequality}
\begin{lem}\label{w2pestimates tri tilde v}
Let the constant $-C_{1.1}=\inf_{i\neq j}R_{i\bar i j\bar j}(\tilde\om_t)$. Then
 \begin{align*}  
\tri_\vphi \log \tilde v \geq - C_{1.1}\tr_\vphi\tilde\om_t+\frac{\tilde\tri \tilde F}{\tilde v},
\end{align*}
where $\tilde\tri$ is the Laplacian operator regarding to the metric $\tilde\om_t$.
\end{lem}
\begin{proof}
The proof of the Laplacian of $\tilde v=\tr_{\tilde\om_t}\om_\vphi+K$ is slightly different from Yau's computation for $v=\tr_{\tilde\om_t}\om_\vphi$. We include the proof of $\tri_\vphi \tilde v$ as below 
and refer to Lemma 3.4 \cite{MR3405866} and Proposition 2.22 \cite{MR4020314} for more references.
The Laplacian of $\tilde v$ is given by the identity
\begin{align*}
\tri_\vphi \tilde v
=g^{i\bar j}g_\vphi^{k\bar l}g_\vphi^{p\bar q}\p_{\bar l}g_{\vphi p\bar j}
\p_{k}g_{\vphi i\bar q}-\tr_{\tilde\om_t}Ric(\om_\vphi)
+g_\vphi^{k\bar l}{R^{i\bar j}}_{k\bar l}(\tilde\om_t)g_{\vphi i\bar j}.
\end{align*}
By the volume ratio $\om_\vphi^n=e^{\tilde F}\tilde\om_t^n$, we have 
$$Ric(\om_\vphi)=Ric(\tilde\om_t)-i\p\bar\p\tilde F$$ and then
\begin{align*}
\tri_\vphi \tilde v
=g^{i\bar j}g_\vphi^{k\bar l}g_\vphi^{p\bar q}\p_{\bar l}g_{\vphi p\bar j}
\p_{k}g_{\vphi i\bar q}+ \tilde\tri \tilde F-S(\tilde\om_t)
+g_\vphi^{k\bar l}{R^{i\bar j}}_{k\bar l}(\tilde\om_t)g_{\vphi i\bar j}.
\end{align*}
Actually, it holds under the normal coordinates that
\begin{align*}
g^{i\bar j}g_\vphi^{k\bar l}g_\vphi^{p\bar q}\p_{\bar l}g_{\vphi p\bar j}
\p_{k}g_{\vphi i\bar q}= \frac{1}{1+\vphi_{k\bar k}}\frac{1}{1+\vphi_{p\bar p}}\vphi_{p\bar i\bar k}\vphi_{i\bar p k}.
\end{align*}
By $1+\vphi_{p\bar p}\leq \sum_{p}(1+\vphi_{p\bar p})=v\leq \tilde v$, we have
\begin{align*}
g^{i\bar j}g_\vphi^{k\bar l}g_\vphi^{p\bar q}\p_{\bar l}g_{\vphi p\bar j}
\p_{k}g_{\vphi i\bar q}
\geq v^{-1} |\p v |^2_\vphi\geq \tilde v^{-1} |\p \tilde v|^2_\vphi=\tilde v |\p \log \tilde v|^2_\vphi.
\end{align*}
 The term involving curvature reduces to the following inequality, by M. Paun's trick \cite{MR2470619},
 \begin{align*}
&-S(\tilde\om_t)+g_\vphi^{k\bar l}{R^{i\bar j}}_{k\bar l}(\tilde\om_t)g_{\vphi i\bar j}
=R_{i\bar i k\bar k}(\tilde\om_t)\sum_{1\leq i,k\leq n}[\frac{1+\vphi_{i\bar i}}{1+\vphi_{k\bar k}}-1]\\
&=R_{i\bar i k\bar k}(\tilde\om_t)\sum_{1\leq i<k\leq n}[\frac{1+\vphi_{i\bar i}}{1+\vphi_{k\bar k}}-1]
+R_{i\bar i k\bar k}(\tilde\om_t)\sum_{1\leq k<i\leq n}[\frac{1+\vphi_{i\bar i}}{1+\vphi_{k\bar k}}-1].
\end{align*}
By symmetric of the Riemannian curvature, it becomes
 \begin{align*}
=R_{i\bar i k\bar k}\sum_{1\leq i<k\leq n}[\frac{1+\vphi_{i\bar i}}{1+\vphi_{k\bar k}}+\frac{1+\vphi_{k\bar k}}{1+\vphi_{i\bar i}}-2],
\end{align*}
 which is nonnegative and leads to
  \begin{align*}
&\geq -C_{1.1}\sum_{1\leq i<k\leq n}[\frac{1+\vphi_{i\bar i}}{1+\vphi_{k\bar k}}+\frac{1+\vphi_{k\bar k}}{1+\vphi_{i\bar i}}-2]\\
&\geq -C_{1.1} \tr_\vphi\tilde\om_t \cdot v
\geq -C_{1.1} \tr_\vphi\tilde\om_t \cdot \tilde v.
\end{align*}
 
 Inserting all these inequalities to the identity
 \begin{align*}  
\tri_\vphi \log \tilde v =\tilde v^{-1}\tri_\vphi \tilde v-|\p\log \tilde v|^2_\vphi,
\end{align*}
 we obtain that the inequality of $\tri_\vphi\log \tilde v$.
 
\end{proof}

Furthermore, we calculate the Laplacian of $u$, which multiplies $\tilde v$ with the weight $e^{-H}$, i.e.
\begin{align*}
u=e^{-H} \tilde v.
\end{align*}
Then
$
\tri_\vphi \log u =-\tri_\vphi H+\tri_\vphi \log \tilde v.
$ Combining with \eqref{w2pestimates tri tilde v}, we have
\begin{lem}
$ 
\tri_\vphi \log u \geq -\tri_\vphi H- C_{1.1}\tr_\vphi\tilde\om_t+\frac{\tilde\tri \tilde F}{\tilde v}.
$
\end{lem}

In particular, we define $H$ as
\begin{align}\label{w2p H}
H:=b_0 F+b_1\psi+b_2 f,\quad b_0\geq 0.
\end{align}
\begin{prop}[Differential inequality]\label{W2p estimate: Differential inequality prop}
\begin{align}\label{W2p estimate: Differential inequality}
\tri_\vphi u\geq  A_\theta u \tr_\vphi\tilde\om_t+(A_R-b_2 \tri_\vphi f) u+e^{-H}\tilde\tri \tilde F,
\end{align}
where we set the constants 
\begin{align}\label{W2p estimate AR b1}
A_\theta:=-b_0C_u+b_1- C_{1.1},\quad A_R:=b_0\inf_X R-b_1n<0,
\end{align}
and choose the positive $b_1$ sufficiently large such that $A_\theta\geq1$.
\end{prop}
\begin{proof}
We use the upper bound of $\theta$ \eqref{Cu}, namely $\tr_\vphi\theta\leq C_u\tr_{\vphi}\tilde\om_t$, to compute the Laplacian of the auxiliary function $H$, 
\begin{align*}
\tri_\vphi (-H)&= -b_0(\tr_\vphi\theta-R)-b_1(n-\tr_\vphi\tilde\om_t)-b_2 \tri_\vphi f\\
&\geq (-b_0C_u+b_1) \tr_\vphi\tilde\om_t+b_0R-b_1n-b_2 \tri_\vphi f.
\end{align*}

Adding the inequalities for $\tri_\vphi(-H)$ and $\tri_\vphi \log \tilde v$ together, we arrive at an inequality for $\log u$,
\begin{align*}
\tri_\vphi \log u\geq  A_\theta \tr_\vphi\tilde\om_t+(A_R-b_2 \tri_\vphi f)+\frac{\tilde\tri \tilde F}{\tilde v}.
\end{align*}
Alternatively, rewritten in the form of $u=e^{-H}v$, it reduces to the desired inequality for $u$.

\end{proof}
\subsection{Integral inequality}
We integrate the differential inequality \eqref{W2p estimate: Differential inequality} into an integral inequality, by multiplying it with $-u^{p-1}$, $\forall p\geq 1$ and integrating over $X$ with respect to $\om^n_\vphi$, 
\begin{align}\label{W2p semi positive LHSRHS}
LHS_1+LHS_2
\leq \int_X[-A_R+b_2 \tri_\vphi f] u^p\om_\vphi^n+II.
\end{align}
where, we denote $II:=-\int_Xu^{p-1} e^{-H}\tilde\tri \tilde F\om_\vphi^n$,
\begin{align*}
&LHS_1:=(p-1)\int_Xu^{p-2}|\p u|^2_\vphi\om_\vphi^n,\quad
\widetilde{LHS_2}:=  \int_X A_\theta u^p \tr_\vphi\tilde\om_t\om_\vphi^n.
\end{align*}
Aapplying the fundamental inequality
$
\tr_\vphi\tilde\om_t\geq e^{-\frac{\tilde F}{n-1}} v^{\frac{1}{n-1}},
$ we have
\begin{align*}
\widetilde{LHS_2}\geq LHS_2:=  \int_X  u^p e^{-\frac{\tilde F}{n-1}} v^{\frac{1}{n-1}}\om_\vphi^n.
\end{align*}

Now we deal with the second term $II$, which involves $\tilde\tri\tilde F$.

\begin{prop}\label{W2p II inequality}
Take $b_0=p$, then 
\begin{align*}
&\frac{3}{4}LHS_1+\widetilde{LHS_2}\leq- \int_XA_Ru^p\om_\vphi^n
+RHS_2
\end{align*}
where
\begin{align*}
RHS_2&:=b_2 \int_X\tri_\vphi f u^p\om_\vphi^n
+ \frac{b_0+b_2}{b_0-1}\int_Xu^{p-1} e^{-H} \tilde\tri f  \om_\vphi^n\\
&+\int_X \frac{b_1}{b_0-1} u^{p-1} e^{-H}v     \om_\vphi^n.
\end{align*}
\end{prop}
\begin{proof}
By $\om_\vphi^n=e^{\tilde F}\tilde \om_t^n$, we get 
$II=-\int_Xu^{p-1} e^{\tilde F-H}\tilde\tri \tilde F\tilde\om_t^n,$ which is decomposed to
\begin{align*}
II=II_1+II_2:=\frac{1}{b_0-1}\int_Xu^{p-1} e^{\tilde F-H}\tilde\tri [(\tilde F-H)+(H-b_0\tilde F)]\tilde\om_t^n.
\end{align*}

By integration by parts, the first part $II_1$ becomes
\begin{align*}
&=-\frac{p-1}{b_0-1}\int_Xu^{p-2} e^{\tilde F-H}(\p u,\p (\tilde F-H))_{\tilde\om_t}\tilde\om_t^n\\
&-\frac{1}{b_0-1}\int_Xu^{p-1} e^{\tilde F-H}|\p (\tilde F-H)|^2_{\tilde\om_t}\tilde\om_t^n.
\end{align*}
We choose $b_0>1$ such that the constant before the second integral is negative. Then H\"older inequality gives us the upper bound of $II_1$,
\begin{align*}
II_1\leq\frac{(p-1)^2}{4(b_0-1)}\int_Xu^{p-3} e^{\tilde F-H}|\p u|^2_{\tilde\om_t}\tilde\om_t^n.
\end{align*}
Using $|\p u|^2_{\tilde\om_t}\leq v |\p u|^2_{\vphi}$ and $u=e^{-H} v$, we deduce that
\begin{align*}
II_1\leq\frac{(p-1)^2}{4(b_0-1)}\int_Xu^{p-2} |\p u|^2_{\vphi}\om_\vphi^n=\frac{p-1}{4(b_0-1)}LHS_1.
\end{align*}

In order to estimate $II_2$, we calculate $$H-b_0\tilde F=b_0F+b_1\tilde\vphi-b_0\tilde F+b_2f=(b_0+b_2) f+b_1\tilde\vphi$$ and
\begin{align*}
\tilde\tri (H-b_0\tilde F)=(b_0+b_2) \tilde\tri f+b_1 v-b_1 n.
\end{align*}
By substitution into the part $II_2$, we get 
\begin{align*}
II_2=\frac{1}{b_0-1}\int_Xu^{p-1} e^{-H} [(b_0+b_2) \tilde\tri f+b_1 v-b_1 n]      \om_\vphi^n.
\end{align*}
If we further choose $b_0>1, b_1>0$, the negative term could be dropped immediately. Hence, $II_2$ reduces to
\begin{align*}
\leq \frac{1}{b_0-1}\int_Xu^{p-1} e^{-H} [(b_0+b_2) \tilde\tri f+b_1 v]      \om_\vphi^n.
\end{align*}

Inserting $II_1$ and $II_2$ back to \eqref{W2p semi positive LHSRHS} and choosing $b_0$ depending on $p$ such that $$1-\frac{p-1}{4(b_0-1)}>0,$$ we have arrives at the desired weighted inequality.
\end{proof}

In order to estimate $\tilde\tri f$, we need to use the upper bound of $i\p\bar\p f$.

\begin{lem}
If $b_2=0$, then $H=b_0 F+b_1\psi$ and
\begin{align}\label{w2p estimates: upper bound}
LHS_2\leq C\int_X[ u^p+u^{p-1} e^{-H} +u^{p-1} e^{-H}  v  ]  \om_\vphi^n.
\end{align}
\end{lem}
\begin{proof}
Taking $b_2=0$ in \lemref{W2p II inequality}, we get
\begin{align*}
RHS_2= \frac{b_0}{b_0-1}\int_Xu^{p-1} e^{-H} \tilde\tri f  \om_\vphi^n
+\frac{b_1}{b_0-1}\int_X u^{p-1} e^{-H}v     \om_\vphi^n.
\end{align*}
We make use of the particular property of $i\p\bar\p f$ in the degenerate situation as shown in \lemref{nef tilde f}, i.e. it is bounded above. So, we obtain that
\begin{align*}
RHS_2\leq \frac{b_0 n \sup_{(X,\tilde\om_t)} i\p\bar\p f   }{b_0-1}\int_Xu^{p-1} e^{-H}     \om_\vphi^n+ \frac{b_1}{b_0-1}\int_Xu^{p-1} e^{-H}  v  \om_\vphi^n.
\end{align*} 
\end{proof}

\begin{cor}We further take $K=0$ and rewrite \eqref{w2p estimates: upper bound} in the form
\begin{align}\label{w2p estimates: upper bound cor}
LHS_2=\int_X   v^{p+\frac{1}{n-1}}  \mu
\leq C\int_X (v^p+v^{p-1})   e^{\frac{\tilde F}{n-1}}\mu
\end{align}
where we denote
\begin{align}\label{W2p mu}
L:=-pH+\tilde F-\frac{\tilde F}{n-1},\quad \mu=e^L\tilde\om_t^n.
\end{align}

\end{cor}
We further treat the terms on the right hand side. 
\begin{prop}
\begin{align}\label{W2p semi positive LHSRHS 1}
\int_X   v^{p+\frac{1}{n-1}}  \mu
\leq C\int_X v^{p-1}  e^h \mu,\quad e^h:=e^{\frac{n}{n-1}\sigma_s\phi_E+\frac{-f}{n-1}}.
\end{align}
where the constant $C$ depends on 
\begin{align*}
\sup_X(F-\sigma_s\phi_E),\quad A_\theta,\quad A_R,\quad b_0(p),\quad b_1, \quad\sup_{(X,\tilde\om_t)} i\p\bar\p f.
\end{align*}
\end{prop}
\begin{proof}
Applying Young's inequality with $a=\frac{n}{n-1}$ and $b=n$, we have
\begin{align*}
&\int_X v^p e^{\frac{\tilde F}{n-1}}  \mu
=\int_X v^{(p+\frac{1}{n-1})\frac{1}{a}}  v^{\frac{p}{b}-\frac{1}{(n-1)a}}e^{\frac{\tilde F}{n-1}}  \mu\\
&\leq \tau LHS_2+\int_X  v^{p-\frac{b}{(n-1)a}}e^{\frac{b\tilde F}{n-1}}  \mu
=\tau LHS_2+\int_X  v^{p-1}e^{\frac{n\tilde F}{n-1}}  \mu.
\end{align*}
Inserting in to \eqref{w2p estimates: upper bound cor}, we get
\begin{align*}
LHS_2
\leq C[\int_X v^{p-1}  \cdot e^{\frac{\tilde F}{n-1}}\mu+\tau LHS_2+\int_X  v^{p-1}e^{\frac{n\tilde F}{n-1}}  \mu].
\end{align*}
Choosing sufficiently small $\tau$, we obtain
\begin{align*}
LHS_2
\leq C\int_X v^{p-1} ( e^{\frac{\tilde F}{n-1}}+e^{\frac{n\tilde F}{n-1}})  \mu.
\end{align*}

Recall the volume ratio estimate from \eqref{Gradient estimate: volume ratio bound},
$
C^{-1} e^{\sigma_i\phi_E-f}   \leq e^{\tilde F}\leq C e^{\sigma_s\phi_E-f} 
$ and $e^{-f}$ is of the order $S_\eps^{\beta-1}$.
We could compare the weights $e^{\frac{n\tilde F}{n-1}} $ and $e^{\frac{\tilde F}{n-1}} $ and observe that both of them are estimated by $e^h$. Thus we have obtained \eqref{W2p semi positive LHSRHS 1}.
\end{proof}

Now we are ready to derive the $W^{2,p}$-estimates from \eqref{W2p semi positive LHSRHS 1}.
We first compute the weight $\mu=e^L$.
\begin{lem}\label{W2p eL}We have
\begin{align*}
&L\leq[pb_1-\sigma_i(pb_0-\frac{n-2}{n-1})]\phi_E-\frac{n-2}{n-1}f+C,\\
&L\geq[pb_1-\sigma_s(pb_0-\frac{n-2}{n-1})]\phi_E-\frac{n-2}{n-1}f-C.
\end{align*}
The constant $C$ depends on $\|\vphi\|_\infty$, $\sup_X(F-\sigma_s\phi_E)$, $\inf_X(F-\sigma_i\phi_E)$.
\end{lem}
\begin{proof}
We compute that
\begin{align*}
L&=-pb_0 F-pb_1\psi+\frac{n-2}{n-1}\tilde F\\
&=(-pb_0+\frac{n-2}{n-1})F-\frac{n-2}{n-1}f-pb_1\vphi+pb_1\phi_E.
\end{align*}
Making use of the bound of $F$ namely, 
\begin{align*}
\sigma_i\phi_E-C\leq F\leq \sigma_s\phi_E+C
\end{align*}
and the bound of $\vphi$ from \thmref{L infty estimates Singular equation}, we obtain the bound of $L$.

\end{proof}
Then we estimate $LHS_2$.
\begin{prop}
We choose 
\begin{align}\label{W2p estimate b1 2}
b_1>(b_0-\frac{n-2}{p(n-1)})\sigma_i-(1+\frac{1}{p(n-1)})\sigma_s-\frac{2n}{p}.
\end{align}
 Then
\begin{align*}
\int_X   v^{p+\frac{1}{n-1}}  \mu
\leq C.
\end{align*}
\end{prop}
\begin{proof}
We apply the H\"older inequality with $a=\frac{p+\frac{1}{n-1}}{p-1}$ to \eqref{W2p semi positive LHSRHS 1}
\begin{align*}
\int_X   v^{p+\frac{1}{n-1}}  \mu
\leq C\int_X v^{p-1}  e^h \mu
\leq C(\int_X v^{p+\frac{1}{n-1}}\mu)^{\frac{1}{a}}
(\int_X e^{\frac{ah}{a-1}}\mu)^{1-\frac{1}{a}}.
\end{align*}

We now estimate the integral $\int_X e^{\frac{ah}{a-1}}\mu$.
We calculate that $\frac{a}{a-1}=\frac{p(n-1)+1}{n}$. Then we insert $h$ \eqref{W2p semi positive LHSRHS 1} and the estimate of $\mu$ \eqref{W2p mu} from \lemref{W2p eL} to the last integrand
\begin{align*}
 e^{\frac{ah}{a-1}}\mu=e^{\frac{p(n-1)+1}{n-1}\sigma_s\phi_E+\frac{-f}{n-1}\frac{p(n-1)+1}{n}}e^L\tilde\om_t^n.
\end{align*}
We write 
\begin{align*}e^{\frac{ah}{a-1}}\mu\leq  e^{k_1\phi_E+k_2(-f)}
\end{align*} and compare the coefficients 
\begin{align*}
&k_1=pb_1-(pb_0-\frac{n-2}{n-1})\sigma_i+(p+\frac{1}{n-1})\sigma_s,\\
&k_2=\frac{p}{n}+\frac{1}{(n-1)n}+\frac{n-2}{n-1}.
\end{align*}
Since $k_2>0$, we have $ e^{k_2(-f)}$ is bounded above. Also, $e^{k_1\phi_E}$ is integrable, if $k_1+2n>0$.

\end{proof}
\begin{rem}
We further expand $k_1$,
\begin{align*}
k_1=pb_1-(pb_0-\frac{n-2}{n-1})C_u+(p+\frac{1}{n-1})C_l-(pb_0-\frac{n-3}{n-1}+p)\tau.
\end{align*}
Since $\tau$ could be very small, it is sufficient to ask $$pb_1-(pb_0-\frac{n-2}{n-1})C_u+(p+\frac{1}{n-1})C_l\geq 0.$$
Choosing large $b_1$ as in Proposition \ref{W2p estimate: Differential inequality prop}, for example,
$$
b_1= C_ub_0-\frac{n}{n-1}C_l+C_{1.1}+1,
$$ we also obtain 
\begin{align*}
&k_1\geq -\tau (pb_0-\frac{n-3}{n-1}+p).
\end{align*}
Hence, $e^{k_1\phi_E}$ is integral once $\tau$ is small enough.
\end{rem}

Therefore, we complete the proof of the $W^{2,p}$-estimates as following.
\begin{proof}[Proof of \thmref{w2pestimates degenerate Singular equation}]
We set $$\tilde L=k_1\phi_E+k_2(-f),$$ where $k_1,k_2$ are will be determined in the following argument.
The H\"older inequality with $a=\frac{p+\frac{1}{n-1}}{p}$ gives
\begin{align*}
\int_X   v^{p} e^{\tilde L} \tilde\om^n_t
=  \int_X   v^{p} e^{\frac{L}{a}}  e^{\tilde L-\frac{L}{a}}\tilde\om^n_t
\leq (\int_X   v^{p+\frac{1}{n-1}} e^L \tilde\om^n_t)^{\frac{1}{a}}(\int_X e^{\frac{a\tilde L-L}{a-1}} \tilde\om^n_t)^{\frac{a-1}{a}}.
\end{align*}
By \lemref{W2p eL}, we get
\begin{align*}
-L&\leq [-pb_1+\sigma_s(pb_0-\frac{n-2}{n-1})]\phi_E+\frac{n-2}{n-1}f+C.
\end{align*}
In order to guarantee the last integral to be finite, we need
\begin{align*}
&\frac{1}{a-1}[ak_1-pb_1+\sigma_s(pb_0-\frac{n-2}{n-1})]+2n>0,\\
&\frac{1}{a-1}(ak_2-\frac{n-2}{n-1})+2n>0.
\end{align*}
We calculate that $\frac{1}{a-1}=p(n-1)$.  These are exactly the hypotheses in \thmref{w2pestimates degenerate Singular equation}.
\end{proof}

\bibliographystyle{abbrv}
\bibliography{Large}

\end{document}